\documentclass[reqno,11pt]{amsart}
\usepackage{amssymb, amsmath,amsmath,latexsym,amssymb,amsfonts,amsbsy, amsthm,mathtools,graphicx,CJKutf8,CJKnumb,CJKulem,color,cases,empheq,extpfeil,bm}

\renewcommand{\theequation}{\thesection.\arabic{equation}}
\numberwithin{equation}{section}

\allowdisplaybreaks

\newtheorem{remark}{Remark}[section]
\newtheorem{lemma}{Lemma}[section]
\newtheorem{definition}{Definition}[section]
\newtheorem{corollary}{Corollary}[section]
\newtheorem{proposition}{Proposition}[section]
\newtheorem{theorem}{Theorem}[section]

\begin{document}


\title[Strong solutions to a full non-Newtonian fluid with far field vacuum]
{Global strong solutions to a one-dimensional full non-Newtonian fluid with far field vacuum}
\author[Fang]{Li Fang}
\address{Li Fang $\newline$ School of Mathematics and CNS, Northwest University, Xi'an, P. R. China}
\email{fangli@nwu.edu.cn}

\author[Wang]{Yu Wang}
\address{Yu Wang $\newline$ School of Mathematics and CNS, Northwest University, Xi'an, P. R. China}
\email{WYbbdwy@163.com}

\author[Zang]{Aibin Zang}
\address{Aibin Zang $\newline$ School of Artificial Intelligence and Information Engineering \& The Center of Applied Mathematics, Yichun university, Yichun, Jiangxi, P. R. China}
\email{abzang@jxycu.edu.cn}


\begin{abstract}
In this paper, the Cauchy problem for a one-dimensional heat conducting compressible non-Newtonian fluid is considered. The constitute equation of the non-Newtonian fluid is determined by two nonlinear terms $(|u_x|^{q-2}u_x)_x$ and $(|\theta_x|^{p-2}\theta_x)_x$ with $1<p,q<2.$ When the vacuum occurs at the far field, the local and global existence of strong solutions are established for the Cauchy problem.
The results indicate that the non-Newtonian fluid possesses time-dependence boundedness of the energy if the vacuum occurs at the far field and the density decays slowly at the far field as time goes to infinity. This is the key difference from the well-known result developed by Li and Xin (Advances in Mathematics 361(2020), 106923) for the one-dimensional heat conductive compressible Navier-Stokes system.
\end{abstract}


\maketitle

\vskip 0.2cm


\renewcommand{\theequation}{\thesection.\arabic{equation}}
\setcounter{equation}{0}

\section{Introduction and Main results}\label{S1}

\subsection{Introduction}

The non-Newtonian fluid has been studied either for applications in drag redacting or theoretical analysis see\cite{Figueredo-2003,Mowla-2006} (and therein references). Mathematical theory on non-Newtonian fluids goes back to Ladyzhenskaya (see \cite{Ladyzhenskaya-1967, Ladyzhenskaya-1969,Ladyzhenskaya-1970}). It is well-known that the mathematical model of a compressible nonhomogeneous non-Newtonian heat-conducting fluid governed by the following system of equations (see \cite{Bartlomiej-Aneta-2018})
 \begin{align}
&\partial_t\rho+\mbox{div}_{\bf x}(\rho {\bf u})=0\mbox{ in }{\mathbb R}^d,\label{E0.1}\\
&\partial_t(\rho {\bf u})+\mbox{div}_{\bf x}(\rho {\bf u}\otimes {\bf u})-\mbox{div}_{\bf x}({\mathbb S}(x,\rho,\theta,{\mathbb D}{\bf u})+\nabla_{x}p=\rho{\bf f}\mbox{ in }{\mathbb R}^d,\label{E0.2}\\
&\partial_t(\rho\theta)+\mbox{div}_{\bf x}(\rho{\bf u}\theta)-\mbox{div}_{\bf x}{\bf q}(\rho,\theta,\nabla_{ x}\theta)={\mathbb S}(x,\rho,\theta,{\mathbb D}{\bf u}):{\mathbb D}{\bf u}\mbox{ in }{\mathbb R}^d,\label{E0.3}\\
&{\bf u}(0,x)={\bf u}_0\mbox{ in }{\mathbb R}^d,\label{E0.4}\\
&\rho(0,x)=\rho_0\mbox{ in }{\mathbb R}^d,\label{E0.5}
\end{align}
where the unknown function $\rho:{\mathbb R}^d\rightarrow[0,+\infty)$ denotes the fluid density, ${\bf u}:{\mathbb R}^d\rightarrow{\mathbb R}^3$ stands for the fluid velocity, $p:{\mathbb R}^d\rightarrow[0,+\infty)$ is a pressure function, 
${\bf f}:{\mathbb R}^d\rightarrow{\mathbb R}^3$ is a given outer force. $\theta\geqslant 0$ denotes the temperature. Here, 
${\mathbb S}{\bf u}-p(\rho){\mathbb I}$ is the total stress tensor. The viscous stress tensor ${\mathbb S}{\bf u}$ characterizes the measure of resistance of the fluid to flow. In accordance with the principle of material frame indifference, the viscous stress tensor ${\mathbb S}{\bf u}$ depends on the strain tensor ${\mathbb D}{\bf u}~(=\frac12(\nabla_x {\bf u}+\nabla_x{\bf u}^t)),$ which is a symmetric part of the velocity gradient.

In order to formulate the growth conditions of the stress tensor, we assume that stress tensor ${\mathbb S}{\bf u}=|{\mathbb D}{\bf u}|^{q-2}{\mathbb D}{\bf u}$ with $q>1.$ The heat flux ${\bf q},$ as usually, is set to be less general. Similarly as in \cite{Yuan-Si-Feng-2019}, we expect ${\bf q}$ to be the form
$${\bf q}(\rho,\theta,\nabla_x\theta)=|\nabla_x \theta|^{p-2}\nabla_x \theta$$
for some $p>1.$ In order to close the system (\ref{E0.1})-(\ref{E0.3}), some constitutive equations are complemented. We assume that the pressure $p$ takes the form
$$p=R\rho \theta$$
for a positive constant $R.$

The main reason to investigate such form of stress tensor ${\mathbb S}{\bf u}$, is the phenomena of rapidly increasing fluid viscosity under various stimuli as shear rate. Our assumptions is power-law models which are quite popular among rheologists, chemical engineering and colloidal mechanics. 


The initial model studied by Ladyzhenskaya \cite{Ladyzhenskaya-1967}  is incompressible, which the system (\ref{E0.1})-(\ref{E0.3}) without the temperature fluctuations  becomes
 \begin{align}
&\partial_t{\bf u}+\mbox{div}_{\bf x}({\bf u}\otimes {\bf u})+\nabla_{\bf x}\pi=\mbox{div}_{\bf x}({\mathbb S}{\bf u}),\label{1dlg-E1.6}\\
&\mbox{div}_{\bf x}{\bf u}=0,\label{1dlg-E1.7}
\end{align}
where the stress tensor ${\mathbb S}{\bf u}$ depends on the strain tensor ${\mathbb D}{\bf u}$ in a non-linear way. Ladyzhenskaya selected the following monotone vector field
\begin{align}
{\mathbb S}{\bf u}=\nu_1 {\mathbb D}{\bf u}+\nu_2|\mathbb{D}{\bf u}|^{q-2}{\mathbb D}{\bf u}\label{1dlg-E1.8}
\end{align}
as constitutive relation, where constants $q,\nu_1,\nu_2$ satisfy $q>1,$ $\nu_1\geqslant 0$ and $\nu_2>0.$ The basic analysis of the system (\ref{1dlg-E1.6})-(\ref{1dlg-E1.8}) was given by Ladyzhenskaya \cite{Ladyzhenskaya-1967} and Lions (see also \cite{Lions-1969}). In recent  years, considerable attention was paid to the mathematical model of incompressible non-Newtonian fluid, such as behavior and rapidly grew-up of solutions.
Bellout, Bloom and Ne\v{c}as \cite{Bellout-1994} proved that there exist Young-measure-valued solutions to the incompressible non-Newtonian fluids for space periodic problems under some conditions. Zang obtained the global existence of the models in the half space with slip boundary conditions in \cite{Zang-2018}, under the suitable assumption on the initial velocity. Ne\v{c}asov\'{a} and Penel \cite{Necasova-2001} proved $L^2$-decay for weak solution to an incompressible non-Newtonian fluid in whole space under some assumptions. Guo and Zhu \cite{Guo-Zhu-2002} investigated the partial regularity of the weak solutions to the generalized Navier-Stokes equations which describe the dynamics of the incompressible monopolar non-Newtonian fluids. For more results on the incompressible non-Newton fluids, we can refer the  monographs \cite{Chhabra-2008, Malek-1997}, papers \cite{Bellout-1994,Bohme-1987,Diening-2010,Guo-Zhu-2002,Guo-Shang-2006,Zhikov-2009,Zang-2018} and therein references.

When the temperature fluctuations can be neglected, the system (\ref{E0.1})-(\ref{E0.3}) becomes the isentropic compressible non-Newtonian fluid
\begin{align}\label{1dlg-E1.9}
\begin{cases}
\partial_t\rho+\mbox{div}_{\bf x}(\rho {\bf u})=0,\\
\partial_t(\rho {\bf u})+\mbox{div}_{\bf x}(\rho {\bf u}\otimes {\bf u})+\nabla_{\bf x}p(\rho)=\mbox{div}_{\bf x}({\mathbb S}{\bf u}),
\end{cases}
\end{align}
where the relation between ${\mathbb S}{\bf u}$  and  ${\mathbb D}{\bf u}$ is nonlinear.
For the weak solution, the measure-valued solutions to the equations (\ref{1dlg-E1.9}) was obtained in \cite{Necasova-1993}. Later, Mamontov \cite{Mamontov-1999,Mamontov-1999-1} illustrated the global existence of  weak solution to the equations (\ref{1dlg-E1.9}) under the assumptions of an exponentially growing viscosity and isothermal pressure. Recently, Abbatiello, Feireisl and Novot\'{n}y proved the existence of so-called dissipative solution in \cite{Abbatiello-2020}. For the existence of the strong solution, the local-in-time existence were established for the absence of vacuum in \cite{Kalousek-Macha-Necasova-2020}. Fang and Guo discussed the analytical solution to a class of non-Newtonian fluids with free boundary in \cite{Fang-Guo-2012}. Xu and Yuan \cite{Yuan-Xu-2008} proved the local-in-time existence in one space dimension with singularity and vacuum. When the initial energy is small, Yuan, Si and Feng \cite{Yuan-Si-Feng-2019} established the global well-posedness of strong solutions for the initial-boundary-value problems of the one-dimensional model (\ref{1dlg-E1.9}).

\subsection{Methodology and main Results}
To our knowledge, there are few results about the global well-posedness for the Cauchy problem of the heat conducting compressible non-Newtonian fluid. The current theme of this article is to investigate the mathematical model of a compressible nonhomogeneous non-Newtonian heat-conducting fluid, which is governed by the following system of equations
\begin{eqnarray}
&& \rho_t+(\rho u)_x=0,\label{1dlg-E1.1} \\
&&(\rho u)_t+(\rho u^2)_x-(|u_x|^{q-2}u_x)_x+(R\rho\theta )_x=0, \label{1dlg-E1.2}\\
&&(\rho \theta )_t+(\rho u \theta)_x-(|\theta_x|^{p-2}\theta_x)_x+R\rho\theta u_x=|u_x|^{q}.\label{1dlg-E1.3}
    \end{eqnarray}
Recently, the entropy bounded solutions to a one-dimensional compressible Navier-Stokes equations with zero heat conduction and far field vacuum was established by Li and Xin \cite{Li-Xin-2020}. Inspired by the ideas in \cite{Li-Xin-2020,Fang-Zang-2023}, we intend to find new phenomena of full compressible non-Newtonian fluid and investigate the existence of strong solutions to the system (\ref{1dlg-E1.1})-(\ref{1dlg-E1.3}) with $p,q\in (1,2)$ in the presence of vacuum at far field.

Recall the flow map $\eta(y, t)$ governed by $u$ as follows
\begin{equation*}
    \left\{
    \begin{array}{ll}
    \eta_t(y,t)=u(\eta(y,t),t), \\
    \eta(y,0)=y.
    \end{array}
    \right.
\end{equation*}
Denote $\varrho(y,t):= \rho(\eta(y,t),t),~v(y,t):= u(\eta(y,t),t),~\Theta(y,t):= \theta(\eta(y,t),t).$
Setting
\begin{equation*}
 J=J(y,t)=\eta_y(y,t),
    \end{equation*}
one finds that
\begin{equation*}
 J_t=v_y.
    \end{equation*}
Let $\varrho|_{t=0}=\varrho_0$ and $J|_{t=0}=J_0.$  It is easy to get that
\begin{equation}\label{1dlg-E16}
 J\varrho=J_0\varrho_0.
    \end{equation}
Now, the system (\ref{1dlg-E1.1})-(\ref{1dlg-E1.3}) is transformed to the following in the Lagrangian coordinate
\begin{eqnarray}
&& J_t=v_y,\label{1dlg-E13} \\
&&\varrho_0 v_t-\frac{1}{J_0}\left(|\frac{v_y}J|^{q-2}\frac{v_y}J\right)_y+R\frac{1}{J_0}(\varrho\Theta)_y=0,\label{1dlg-E14} \\
&&\varrho_0\Theta_t-\frac{1}{J_0}\left(|\frac{\Theta_y}J|^{p-2}\frac{\Theta_y}J\right)_y+\frac{1}{J_0}R\varrho\Theta v_y=\frac{J}{J_0}|\frac{v_y}J|^{q},\label{1dlg-E15}
    \end{eqnarray}
which is completed with the initial condition
\begin{equation}
( J,\varrho,v,\Theta)=(J_0,\varrho_0,v_0,\Theta_0)\label{1dlg-E17}.
\end{equation}
Here, $J_0$ has uniform positive lower and upper bounds. According to the definition of $J,$ the initial $J_0$ could be assumed as the identical one.
However, in order to extend a local solution $(J,\varrho, v,\Theta)$ to be a global one, we find the local existence of solutions to the problem (\ref{1dlg-E16})-(\ref{1dlg-E17}) with initial $J_0$ not being identical one (the details can be found in Proposition \ref{5dlg-P5.2} and the proof of Theorem \ref{1dlg-thm2}).

Before stating our main results, we explain the notations and conventions used throughout this paper. Set
$$\int fdx\triangleq \int_{\mathbb R} fdx.$$
Moreover, for $q\in[1,\infty]$ and a positive integer $k\geqslant 1,$ the standard homogeneous Sobolev spaces are defined as follows
\begin{eqnarray*}
&L^q=L^2({\mathbb R}),~W^{k,q}=W^{k,q}({\mathbb R}),~H^{k}=H^{k,2}({\mathbb R}).
\end{eqnarray*}

Now, we give the definitions of local strong solution and global strong solution to the problem (\ref{1dlg-E16})-(\ref{1dlg-E17}) as follows.

\begin{definition}({Local strong solution})\label{1dlg-D1}
Let $1<p, q<2 $ and $\alpha<\min\{-\frac{q}{2(q-1)},-\frac{4-p}{2-p}\}.$ For a given positive time $T \in (0, \infty)$, a quaternion $(J,\varrho,v,\Theta)$ is called a strong solution to the problem (\ref{1dlg-E16})-(\ref{1dlg-E17}) on $\mathbb R \times [0,T]$, if it has the properties that
\begin{eqnarray}
&& \inf\limits_{y\in\mathbb R, t\in[0,T]}J(y,t)>0,\quad \varrho\geqslant0\  on\ \mathbb R \times [0,T],  \notag\\
&&J-J_0\in C([0,T];L^2),\quad J_t\in L^\infty(0,T;L^q)\cap L^\infty(0,T;L^2), \notag\\
&&\varrho_0^{\frac\alpha 2}J_y\in L^\infty(0,T;L^2),\quad\varrho_0^{\frac{\alpha+1} 2}v\in C([0,T];L^2),\quad\varrho_0^{\frac{\alpha+2} 2}\Theta\in C([0,T];L^2)\notag\\
&&\varrho_0^{\frac{\alpha+1} 2}v_y\in L^\infty(0,T;L^2),\quad\varrho_0^{\frac{\alpha+1} 2}\Theta_y\in L^\infty(0,T;L^2), \notag\\
&&\varrho_0^{\frac{\alpha+1} 2}v_t\in L^2(0,T;L^2),\quad\varrho_0^{\frac{\alpha+3} 2}\Theta_t\in L^2(0,T;L^2), \notag\\
&&\varrho_0^{\frac\alpha q}|\frac {v_y} J|J^\frac 1 q\in L^\infty(0,T;L^q)\cap L^q(0,T;L^q),\quad\varrho_0^{\frac\alpha p}|\frac {\Theta_y} J|J^\frac 1 p\in L^\infty(0,T;L^p)\cap L^p(0,T;L^p),\notag\\
&&\varrho_0^{\frac\alpha 2}|\frac {v_y} J|^{\frac{q-2}2}\frac {v_{yy}}{\sqrt J}\in L^2(0,T;L^2),\quad\varrho_0^{\frac\alpha 2}|\frac {\Theta_y} J|^{\frac{p-2}2}\frac {\Theta_{yy}}{\sqrt J}\in L^2(0,T;L^2);\notag
    \end{eqnarray}
the quaternion $(J,\varrho,v,\Theta)$ satisfies the system (\ref{1dlg-E16})-(\ref{1dlg-E15}) a.e. in $\mathbb R \times [0,T] $ and fulfills the initial condition (\ref{1dlg-E17}).
\end{definition}

\begin{definition}({Global strong solution})\label{1dlg-D1.2}
A quaternion $(J,\varrho,v,\Theta)$ is called a global strong solution to the problem (\ref{1dlg-E16})-(\ref{1dlg-E17}), if it is a strong solution to the problem (\ref{1dlg-E16})-(\ref{1dlg-E17}) on $\mathbb R \times [0,T] $ for any positive time $T\in(0,\ \infty).$
\end{definition}

Our first main result focuses on the existence of a local strong solution to the problem (\ref{1dlg-E16})-(\ref{1dlg-E17}), which is stated as follows.

\begin{theorem}(Local well-posedness)\label{1dlg-thm1}
Let $p,q\in(1,2)$ and $\alpha<\min\{-\frac{q}{2(q-1)},-\frac{4-p}{2-p}\}.$ Assume that the initial data $(J_0,\varrho_0,v_0,\Theta_0)$ satisfies
\begin{equation}
 \inf\limits_{y\in(-r, r)}\varrho_0(y)>0\ (\forall r\in (0, \infty)),\quad \varrho_0\leqslant \overline\varrho \ on \ {\mathbb R}\mbox{ with } \overline{\varrho}\geqslant1 ,\quad(\varrho_0^\alpha)_y\in L^p\cap L^\infty,\tag{H1}
    \end{equation}
 \begin{equation}
\underline J\leqslant J_0\leqslant\overline J\ on \ \mathbb R,\quad \varrho_0^{\frac{\alpha}{ 2}}J_0^{\prime}\in L^2,\quad \varrho_0^{\alpha}J_0^{\prime}\in L^1,\tag{H2}
    \end{equation}
\begin{equation}
\varrho_0^{\frac{\alpha+1} 2}v_0\in L^2,\quad\varrho_0^{\frac{\alpha+1} 2}v_0^{\prime}\in L^2,\quad\varrho_0^{\frac\alpha q}|\frac{v_0^{\prime}}{J_0}|\in L^q,\tag{H3}
    \end{equation}
\begin{equation}
\varrho_0^{\frac{\alpha+2} 2}\Theta_0\in L^2,\quad\varrho_0^{\frac{\alpha+1} 2}\Theta_0^{\prime}\in L^2,\quad\varrho_0^{\alpha+1}\Theta_0\in L^2,\quad\varrho_0^{\frac{\alpha+2} p}|\frac{\Theta_0^{\prime}}{J_0}|\in L^p,\tag{H4}
    \end{equation}
for positive constants $\overline{\varrho},$ $\overline {J}$ and $\underline{ J}.$ Then there exists a positive time $T$, depending only on $p,q,\underline J,$ $\overline J,$ $\overline\varrho,$ $\|\varrho_0^{\frac\alpha 2}J_0^{\prime}\|_{L^2},\|(\varrho_o^\alpha)_y\|_{L^p},\|(\varrho_o^\alpha)_y\|_{L^\infty}$, such that the problem (\ref{1dlg-E16})-(\ref{1dlg-E17}) admits a unique strong solution $(J,\varrho,v,\Theta)$ on  $\mathbb R \times [0,T] $.
\end{theorem}


\begin{remark}\label{1dlg-R1.1}
For constants $p,q\in(1,2)$ and $\alpha<\min\{-\frac{q}{2(q-1)},-\frac{4-p}{2-p}\},$ it can be verified that the condition $(\varrho_0^\alpha)_y\in L^p\cap L^\infty$ implies that
\begin{equation}
\int |\varrho_0^{\alpha-1}\varrho_0^{\prime}(y)|^p\ dy<\infty. \notag \\
\end{equation}
If $ \varrho_0=\frac{ K }{(1+|y|)^l}\ for\ some\ 0 < K <\infty\ and\ 0 < l<\infty,$
then
\begin{equation*}
\int |\varrho_0^{\alpha-1}\varrho_0^{\prime}(y)|^p\ dy<\infty\ implies\ 0<l<\frac {p-1} {-\alpha p}.
\end{equation*}
\end{remark}

Our second main result devotes to the global existence of strong solution to the problem (\ref{1dlg-E16})-(\ref{1dlg-E17}). To this end, we make some additional assumptions on
the initial data $\varrho_0(y).$



\begin{theorem}(Global well-posedness).\label{1dlg-thm2}
Let $p,q\in(1,2),$ $\alpha<\min\{-\frac{q}{2(q-1)},-\frac{4-p}{2-p}\}$ and $(\mathrm H1)$-$(\mathrm H4)$ hold. Assume that there exists a positive constant
$$l\in(0,\min\{1,-\frac{3p-2}{(2-p)\alpha+3},-\frac{p-1}{\alpha p}\})$$ such that
\begin{equation}
\varrho_0(y)\geqslant \frac {A_0} {(1+|y|)^{l}}\quad (\forall y \in \mathbb R)\tag{H5}
    \end{equation}
holds for some positive constant $A_0$. Then the problem (\ref{1dlg-E16})-(\ref{1dlg-E17}) admits a unique global strong solution $(J,\varrho,v,\Theta).$
\end{theorem}



\begin{remark}\label{1dlg-R1.4}
It should be mentioned here that the index $\alpha$ of weighted function is controlled by the values of $p$ and $q$ in  Theorem \ref{1dlg-thm1} and Theorem \ref{1dlg-thm2}.
Moreover, the requirement of initial density in  Theorem \ref{1dlg-thm2} depends on the values of $\alpha,$ $p$ and $q.$ These facts not only explain the differences between isentropic compressible non-Newtonian fluid and full compressible non-Newtonian fluid, but also show the differences between full compressible Navier-Stokes equations and full compressible non-Newtonian fluid.
\end{remark}


We explain main ideas to prove Theorem \ref{1dlg-thm1}. Indeed, the local existence of the strong solution $(J,\varrho,v,\Theta)$ to the problem (\ref{1dlg-E16})-(\ref{1dlg-E17}) can be proved by the similar arguments in \cite{Li-Xin-2020,Fang-Zang-2023}. However, some new difficulties appear 
in comparison with Navier-Stokes equations in \cite{Li-Xin-2020} and isentropic non-Newtonian fluid in \cite{Fang-Zang-2023}. Main difficulty comes from the nonlinear term $\left(|\frac{v_y}J|^{q-2}\frac{v_y}J\right)_y$ with $1<q<2$ and the term $\left(|\frac{\Theta_y}J|^{p-2}\frac{\Theta_y}J\right)_y$ with $1<p<2,$ due to certain singularity. To overcome this obstacle, we explore the iterative method to derive the estimate of uniform bound, which is independent of the lower bound of density (for the details to see Proposition \ref{3dlg-P3.2} in Section \ref{3dlg-S3}). The higher order regularities of the solution is obtained by the iterative method.

In order to extend the solution globally in time,  we need to control $\inf\limits_{y\in {\mathbb R}}J(y,t)$ for the problem (\ref{1dlg-E16})-(\ref{1dlg-E17}). We find that $\inf\limits_{y\in {\mathbb R}}J(y,t)$ goes to zero as the time goes to infinity, while the term $\inf\limits_{y\in {\mathbb R}}J(y,t)$ in \cite{Li-Xin-2020} for Navier-Stokes equations is controlled by initial data. 
Fortunately, $\inf\limits_{y\in {\mathbb R}}J(y,t)$ is strictly positive on the existence-time interval of local strong solution to the problem (\ref{1dlg-E16})-(\ref{1dlg-E17}). This observation helps us to find a special function about time to control $\inf\limits_{y\in {\mathbb R}}J(y,t)$ (the details can be found Proposition \ref{5dlg-P5.2} in Section \ref{5dlg-S5}).  With aid of these auxiliary function, we construct a divergent series of the time 
to get the global existence of strong solution (the detail can be found in the proof of Theorem \ref{1dlg-thm2}).

\subsection{Preliminaries}
In this subsection, we introduce some basic lemmas to be used later. The first one is the well-known integration inequality

\begin{lemma}(\cite{Fang-Zang-2023})\label{2dlg-L2.1}
Let $p$ and $l$ be given real constants satisfying $l + 2 > 0$ and $p < l +\frac 3 2$. Then there exists a positive constant $C$ such that
\begin{equation}\label{2dlg-E1}
\sup\limits_{x\in\mathbb R}|\varphi^{\prime}(x)|\leqslant C\left(\int |\varphi^{\prime}(x)|^{2(l-p+2)}dx\right)^{\frac1 {2(l+2)}}\left(\int |\varphi^{\prime}(x)|^{2(p-1)}|\varphi^{\prime\prime}(x)|^2dx\right)^{\frac1 {2(l+2)}}
    \end{equation}
holds for any $\varphi\in C^{2+\nu}(\mathbb R)$ with $0 < \nu < 1,$ if each term in the right-hand side is finite.
\end{lemma}

The second lemma gives a local version of Gronwall's lemma, which is a generalization of Lemma 24 in \cite{Diening-2010} and proved in \cite{Fang-Zang-2023}.

\begin{lemma}(\cite{Fang-Zang-2023})\label{2dlg-L2.2}
(Local version of Gronwall's inequality). Let $T, T_0, \sigma, c_0$ be given positive constants, $h\in L^1(T_0,T)$ with $h\geqslant0$ a.e. on $[T_0,T]$, $f \in C^1[T_0,T]$ satisfying $f \geqslant0$ and
\begin{equation}
f^{\prime}(t)\leqslant h(t) + c_0f^{1+\sigma}(t).
    \end{equation}
Suppose that there exists $t_0\in[T_0,T]$ such that $\sigma c_0H^{\sigma}(t_0)(t_0-T_0)<1$, where
\begin{equation*}
H(t):=f(T_0)+\int_{T_0}^t h(s)ds.
    \end{equation*}
Then
\begin{equation*}
f(t)\leqslant H(t)(1-\sigma c_0H^{\sigma}(t)(t-T_0))^{-\frac1 \sigma}
    \end{equation*}
holds for all $t\in[T_0,t_0]$.
\end{lemma}

Next, we need the following inequalities to deal with the nonlinear terms in (\ref{1dlg-E14}) and (\ref{1dlg-E15}).

\begin{lemma}\label{2dlg-L2.3}(\cite{Lindqvist-2006})
For any fixed constant $r\in(1, 2)$, there exist positive constants $C_1, C_2$, depending on $r$ such that
\begin{equation*}
\begin{aligned}
&\left(|\eta|^{r-2}\eta-|\eta^{\prime}|^{r-2}\eta^{\prime}\right)(\eta-\eta^{\prime})\geqslant C_1\left(1+|\eta|^2+|\eta^{\prime}|^2\right)^{\frac{r-2} 2}|\eta-\eta^{\prime}|^2,\notag\\
&\ ||\eta|^{r-2}\eta-|\eta^{\prime}|^{r-2}\eta^{\prime}|\leqslant C_2|\eta-\eta^{\prime}|^{r-1}
\end{aligned}
    \end{equation*}
hold for any $\eta, \eta^{\prime}\in{\mathbb R}^d$ with $|\eta|+|\eta^{\prime}|>0.$
\end{lemma}

The rest of the paper is organized as follows. 
The existence and uniqueness of solution to the problem (\ref{1dlg-E16})-(\ref{1dlg-E17}) in the absence of vacuum is established in Section \ref{3dlg-S3}.
The proof of the local-in-time existence and uniqueness is given in Section \ref{4dlg-S4}, by adopting a similar framework in \cite{Li-Xin-2020,Fang-Zang-2023}. In Section \ref{5dlg-S5} we prove the global existence of strong solutions.



\section{Local existence in the absence of vacuum}\label{3dlg-S3}

In this section, we will derive some necessary {\sl a priori} estimates for strong solutions to the Cauchy problem (\ref{1dlg-E16})-(\ref{1dlg-E17}) in the absence of vacuum. For convince, it is always assumed that $J_0\equiv1$ throughout this section. Since the arguments are similar in the framework of Li-Xin \cite{Li-Xin-2020} and Fang-Zang \cite{Fang-Zang-2023}, we state the result and just give the outline of the proof as follows.

\begin{proposition}\label{3dlg-P3.1}
Let $p,q\in(1,2)$ and $\alpha<\min\{-\frac{q}{2(q-1)},-\frac{4-p}{2-p}\}.$ Assume that
\begin{equation}\label{3dlg-E1}
\left\{\begin{array}{l}
0<\underline{\varrho}\leqslant\varrho_0(y)\leqslant\overline\varrho<\infty,\quad0<\underline J\leqslant J_0(y)\leqslant\overline J<\infty,\\
\varrho_0^{\prime}\in L^{\infty},\quad J_0^{\prime}\in L^2\cap L^1,\quad v_0\in W^{1,q},\quad \Theta_0\in W^{1,p}
\end{array}
  \right.
    \end{equation}
for positive constants $\underline\varrho,$ $\overline\varrho,$ $\underline J$ and $\overline J.$

Then there exists a positive time $T$ such that the problem (\ref{1dlg-E16})-(\ref{1dlg-E17}) admits a unique local strong solution $(J,\varrho, v,\Theta)$ on ${\mathbb R}\times [0, T]$ satisfying
\begin{eqnarray*}
&&\frac3 4\underline J\leqslant J\leqslant\frac3 4\overline J,~\varrho\geqslant C\underline{\varrho}>0\ on\ \mathbb R \times [0,T],~J-J_0\in C([0,T];W^{1,q}),~J_t\in L^q(0,T;L^q), \\
&&v\in C([0,T];W^{1,q}),|v_y|^{\frac{q-2} 2}v_{yy}\in L^2(0,T;L^2),\Theta\in C([0,T];W^{1,p}),|\Theta_y|^{\frac{p-2} 2}\Theta_{yy}\in L^2(0,T;L^2),
    \end{eqnarray*}
where the time $T$ depends only on $\underline\varrho, \overline\varrho$ and $\mathcal{N}_0$ with
\begin{equation*}
\mathcal{N}_0=\frac1 {\underline J}+\overline J+\| J_0^{\prime}\|_{L^2}+\| v_0^{\prime}\|_{W^{1,q}}+\|\Theta_0^{\prime}\|_{W^{1,p}}.
    \end{equation*}
\end{proposition}

\begin{proof}
The proof is given by standard iterative method and the fixed-point theory. Let $M$ and $T$ be two positive constants with $T$ being suitably small, to be determined later by the quantities $\underline\varrho, \overline\varrho$ and $\mathcal{N}_0.$ Denote
\begin{equation*}
\begin{aligned}
X_{M,T}=\left\{(v,\Theta):\|v\|_{{L^q}(0,T;{W^{1,q}})}\leqslant M,\ \Big|\!\Big||v_y|^{\frac{q-2} 2}v_{yy}\Big|\!\Big|_{{L^2}(0,T;{L^2})}\leqslant M,\right.\\
\phantom{=\;\;}
\left.\|\Theta\|_{{L^p}(0,T;{W^{1,p}})}\leqslant M\ and\ \Big|\!\Big||\Theta_y|^{\frac{p-2} 2}\Theta_{yy}\Big|\!\Big|_{{L^2}(0,T;{L^2})}\leqslant M\right\}.\notag
\end{aligned}
\end{equation*}
Given $(v, \Theta)\in X_{M,T}$, it is easy to get that
\begin{equation*}
\begin{aligned}
\int_0^T\int|\Theta_{yy}|^pdyds
&\leqslant\left(\int_0^T\int\left(|\Theta_y|^{\frac{p-2} 2}|\Theta_{yy}|\right)^2dyds\right)^{\frac p 2}\left(\int_0^T\int|\Theta_y|^pdyds\right)^\frac{2-p}2
\end{aligned}
    \end{equation*}
and
\begin{equation*}
\begin{aligned}
\int_0^T\int|v_{yy}|^qdyds
&\leqslant\left(\int_0^T\int\left(|v_y|^{\frac{q-2} 2}|v_{yy}|\right)^2dyds\right)^{\frac q 2}\left(\int_0^T\int|v_y|^qdyds\right)^\frac{2-q}2.
\end{aligned}
    \end{equation*}
Then, $(v, \Theta)\in X_{M,T}$ implies that $v\in{L^q}(0,T;{W^{2,q}})$ and $\Theta\in{L^p}(0,T;{W^{2,p}}).$ It follows from (\ref{1dlg-E13}) that
\begin{equation*}
J-J_0=\int_0^t v_y ds\in C([0,T];W^{1,q})\mbox{ and }J_t=v_y\in L^q(0,T;W^{1,q}).
    \end{equation*}
According to  Lemma \ref{2dlg-L2.1} and H\"{o}lder inequality, one obtains that
\begin{equation*}
\begin{aligned}
\|J-J_0\|_{L^\infty}
&\leqslant C\left(\int_0^t\|v_y\|_{L^q}^qds\right)^\frac1 {2q}\left(\int_0^t\Big|\!\Big||v_y|^{\frac{q-2} 2}v_{yy}\Big|\!\Big|_{L^2}^2ds\right)^\frac1 {2q}\left(\int_0^t1ds\right)^\frac{q-1} q\\
&\leqslant CM^{\frac{2+q}{2q}}t^{\frac{q-1} q}\leqslant\frac1 4\underline J,
\end{aligned}
    \end{equation*}
as long as $0<t\leqslant T$ with
\begin{equation*}
T=\left(\frac{\underline J} {4CM^{\frac{2+q} {2q}}}\right)^{\frac q {q-1}}.
    \end{equation*}
So, $\frac{3}{ 4}\underline J\leqslant J\leqslant\frac{5}{ 4}\overline J$ on\ $\mathbb R \times [0,T].$

For given $(v, \Theta)\in X_{M,T},$ let $J$ be the unique solution to the initial-value problem
\begin{equation}\label{3dlg-E2}
\begin{cases}
J_t=v_y,\\
J|_{t=0}=J_0\\
\end{cases}
\end{equation}
and $\varrho$ be the unique solution to the initial-value problem
\begin{equation}\label{3dlg-E3}
\begin{cases}
 J\varrho=J_0\varrho_0,\\
\varrho|_{t=0}=\varrho_0.\\
\end{cases}
    \end{equation}
For $J$ and $\varrho$ determined by (\ref{3dlg-E2}) and (\ref{3dlg-E3}), it is deduced from the classical theory for quasi-linear parabolic equation that
\begin{eqnarray*}
\begin{cases}
V_t-\frac1 {J_0\varrho_0}\left(|\frac{V_y}J|^{q-2}\frac{V_y}J\right)_y+\frac1 {J_0\varrho_0}R(\varrho\vartheta)_y=0,\\
V|_{t=0}=v_0
\end{cases}
\end{eqnarray*}
and
\begin{eqnarray*}
\begin{cases}
\vartheta_t-\frac1 {J_0\varrho_0}\left(|\frac{\vartheta_y}J|^{p-2}\frac{\vartheta_y}J\right)_y+\frac1 {J_0\varrho_0}R\varrho\vartheta V_y=\frac1 {J_0\varrho_0}|\frac{V_y}J|^{q}J,\\
\vartheta|_{t=0}=\Theta_0
\end{cases}
\end{eqnarray*}
admit the unique solution $V$ and $\vartheta$ respectively, which possess the following properties
\begin{eqnarray*}
&&V\in L^q(0,T;W^{1,q}),\  |V_y|^{\frac{q-2} 2}V_{yy}\in L^2(0,T;L^2)\mbox{ and }V_t\in L^2(0,T;L^2);\\
&&\vartheta\in L^p(0,T;W^{1,p}),\  |\vartheta_y|^{\frac{p-2} 2}\vartheta_{yy}\in L^2(0,T;L^2)\mbox{ and }\vartheta_t\in L^2(0,T;L^2).
\end{eqnarray*}

As stated above, we do define a mapping ${\mathcal T}:~X_{M,T}\rightarrow X_{M,T}\mbox{ as }$
$${\mathcal T}(v,\Theta)=(V,\vartheta).$$
A standard argument shows that ${\mathcal  T}$ is a contractive mapping on $X_{M,T}$ for some $M$ and $T$ depending only on $\underline{\varrho},~\bar{\varrho}$ and ${\mathcal  N}_0.$ So, there is a unique fixed point to ${\mathcal  T}$ in $X_{M,T},$ which is denoted by $(v,\Theta).$ Thus, the problem (\ref{1dlg-E16})-(\ref{1dlg-E17}) admits a unique local strong solution $(J,\varrho, v,\Theta)$ on ${\mathbb R}\times [0, T].$ The details is completed by standard arguments, and we omit this process here.

In fact, the local strong solution $(J,\varrho, v,\Theta)$ to the problem (\ref{1dlg-E16})-(\ref{1dlg-E17}) on ${\mathbb R}\times[0, T]$ satisfies that
\begin{equation}\tag{[\bf H]}
\left\{\begin{array}{l}
\frac{3}{ 4}\underline J\leqslant J\leqslant\frac{5} {4}\overline J,\quad\varrho\geqslant C\underline{\varrho}>0\ on\ \mathbb R \times [0,T],\\
J-J_0\in C([0,T];W^{1,q}),\quad J_t\in L^q(0,T;L^q),\\
v\in C([0,T];W^{1,q}),\quad |v_y|^{\frac{q-2} 2}v_{yy}\in L^2(0,T;L^2),\quad v_t\in L^2(0,T;L^2),\\
\Theta\in C([0,T];W^{1,p}),\quad |\Theta_y|^{\frac{p-2} 2}\Theta_{yy}\in L^2(0,T;L^2),\quad \Theta_t\in L^2(0,T;L^2).\\
\end{array}
  \right.
    \end{equation}
\end{proof}

Next, we expect to derive the uniform bounds, which do not rely on the lower bound of the initial density $\varrho_0$. 

\begin{proposition}\label{3dlg-P3.2}
Let $p,q\in(1,2)$ and $\alpha<\min\{-\frac{q}{2(q-1)},-\frac{4-p}{2-p}\}.$ Suppose that a quaternion $(J,\varrho, v,\Theta)$ is a strong solution to the problem (\ref{1dlg-E16})-(\ref{1dlg-E17}) on ${\mathbb R}\times [0, T].$ Then it holds that
\begin{eqnarray}
&&\int\varrho_0^{\alpha}\Big(\varrho_0v^2+\varrho_0\Theta+\varrho_0^2\Theta^2+\varrho_0v_y^2+\varrho_0\Theta_y^2+|\frac{v_y}{J}|^qJ+\varrho_0|\frac{\Theta_y} J|^pJ+ J_y^2\Big)dy\nonumber\\
&&\leqslant CH_0(t)\bigg(1-\frac{q+1}{q-1}H_0^{\frac{q+1}{q-1}}(t)t\bigg)^{-\frac{q-1}{q+1}},\label{3dlg-E4}\\
&&\bigg(\int\varrho_0^{\alpha}\Big(\varrho_0v^2+\varrho_0\Theta+\varrho_0^2\Theta^2+\varrho_0v_y^2+\varrho_0\Theta_y^2+|\frac{v_y}{J}|^qJ+\varrho_0|\frac{\Theta_y} J|^pJ+ J_y^2\Big)dy\bigg)(t)\nonumber\\
&&+\int_0^t\int\varrho_0^\alpha\Big(|\frac{v_y}{J}|^qJ+\varrho_0|\frac{\Theta_y} J|^pJ+|\frac{v_y}{J}|^{q-2}\frac { v_{yy}^2}J+|\frac{\Theta_y} J|^{p-2}\frac {\Theta_{yy}^2}J+\varrho_0^{3}\Theta_t^2+\varrho_0v_t^2\Big)dyds\nonumber\\
&&\leqslant CH_0(t)+C\int_0^tH_0^{\frac{2q} {q-1}}(s)\bigg(1-\frac{q+1}{q-1}H_0^{\frac{q+1}{q-1}}(s)s\bigg)^{-\frac{2q}{q+1}}ds\label{3dlg-E5}
\end{eqnarray}
for all $t\in[0,\widetilde{t_0}]$ with $\widetilde{t_0}<T,$ where 
\begin{equation}\label{3dlg-E6}
\begin{aligned}
H_0&=\Big(\|\varrho_0^{\frac {\alpha+2} 2}v_0\|_{L^2}^2+\|\varrho_0^{\alpha+1}\Theta_0\|_{L^1}+\|\varrho_0^{\frac {\alpha+2} 2}\Theta_0\|_{L^2}^2+\|\varrho_0^{\frac {\alpha+1} 2}v_0^{\prime}\|_{L^2}^2+\|\varrho_0^{\frac {\alpha+1} 2}\Theta_0^{\prime}\|_{L^2}^2\\
&+\|\varrho_0^{\frac \alpha q}|\frac {v_0^{\prime}} {J_0}|{J_0}^{\frac 1 q}\|_{L^q}^q+\|\varrho_0^{\frac {\alpha+2} p}|\frac {\Theta_0^{\prime}} {J_0}|{J_0}^{\frac 1 p}\|_{L^p}^p+\|\varrho_0^{\frac \alpha 2}J_0^{\prime}\|_{L^2}^2\Big)+G_0t\notag
\end{aligned}
 \end{equation}
on $[0,\widetilde{t_0}]$ with
\begin{equation*}
G_0=C{\overline J}^{12}\Big(\frac1{\underline J}\Big)^{\frac{12q}{2-q}}{\overline \varrho}^{-\min\{\frac{(3q-2)\alpha+q}{2-q},\frac{8\alpha}{q-1}\}}
\left(\|(\varrho_0^\alpha)_y\|_{L^p}\|(\varrho_0^\alpha)_y\|_{L^\infty}\right)^{\frac{4q}{2-q}}+C\overline{\varrho}^{-\frac{2\alpha+1}{2(q+1)}}
\end{equation*}
and
\begin{equation}\label{3dlg-E7}
\frac{q+1}{q-1}H_0^{\frac{q+1}{q-1}}(\widetilde{t_0})\widetilde{t_0}<1.
    \end{equation}
\end{proposition}

\begin{proof}
The proof is divided into several steps.

{\sl\textit{Step 1.}} Multiplying equation (\ref{1dlg-E14}) and  (\ref{1dlg-E15}) by $\varrho_0^\alpha v$ and $\frac{1}{2}\varrho_0^\alpha$ respectively, adding the resultant and integrating over $\mathbb R$, one obtains that
\begin{eqnarray}
&&\frac1 2\frac d {dt}\int\Big(\varrho_0^{\alpha+1}v^2+\varrho_0^{\alpha+1}\Theta\Big)dy+\frac1 2\int\varrho_0^\alpha|\frac{v_y}{J}|^qJdy\nonumber\\
&&=-\int|\frac{v_y}{J}|^{q-2}\frac {v_y}J(\varrho_0^\alpha)_yvdy+\int R\varrho\Theta(\varrho_0^\alpha)_yvdy-\frac{1}{2}\int|\frac{\Theta_y} J|^{p-2}\frac {\Theta_y}J(\varrho_0^\alpha)_ydy\nonumber\\
&&+\frac{1}{2}\int R\varrho\Theta v_y \varrho_0^\alpha dy:=\sum_{i=1}^4I_i,\label{3dlg-E8}
    \end{eqnarray}
where
\begin{eqnarray*}
&|I_1|&=\Big|-\int|\frac{v_y}{J}|^{q-2}\frac {v_y}J(\varrho_0^\alpha)_yvdy\Big|\\
&&\leqslant C\overline\varrho^{-\frac{(3q-2)\alpha+q}{2q}}\Big(\frac1 {\underline J}\Big)^{\frac{q-1}{q}}\|(\varrho_0^\alpha)_y\|_{L^\frac{2q}{2-q}}\Big(\int\varrho_0^{\alpha+1}v^2dy\Big)^{\frac1 2}\Big(\int\varrho_0^\alpha|\frac{v_y}{J}|^qJdy\Big)^{\frac{q-1}{q}}\\
&&\leqslant \Big(\int\varrho_0^{\alpha+1}v^2dy\Big)^{\frac{2q}{q-1}}+\epsilon\int\varrho_0^\alpha|\frac{v_y}{J}|^qJdy+C(\epsilon)\left(\overline\varrho^{-\frac{(3q-2)\alpha+q}{2q}}\Big(\frac1 {\underline J}\Big)^{\frac{q-1}{q}}\|(\varrho_0^\alpha)_y\|_{L^\frac{2q}{2-q}}\right)^{\frac{4q}{5-q}},\\
&|I_2|&=\Big|\int R\varrho\Theta(\varrho_0^\alpha)_yvdy\Big|\\
&&\leqslant C\overline\varrho^{-\frac{2\alpha+1} 2}\Big(\frac1 {\underline J}\Big)\|(\varrho_0^\alpha)_y\|_{L^\infty}\Big(\int\varrho_0^{\alpha+2}\Theta^2dy\Big)^{\frac1 2}\Big(\int\varrho_0^{\alpha+1}v^2dy\Big)^{\frac1 2}\\
&&\leqslant \Big(\int\varrho_0^{\alpha+2}\Theta^2dy+\int\varrho_0^{\alpha+1}v^2dy\Big)^{\frac{2q} {q-1}}+C\overline\varrho^{-\frac{(2\alpha+1)q} {q+1}}\Big(\frac1 {\underline J}\Big)^{\frac {2q} {q+1}}\|(\varrho_0^\alpha)_y\|_{L^\infty}^{\frac{2q} {q+1}},\\
&|I_3|&=\Big|-\int|\frac{\Theta_y} J|^{p-2}\frac {\Theta_y}J(\varrho_0^\alpha)_ydy\Big|\\
&&\leqslant C\overline\varrho^{-(\alpha+1)\frac{p-1} p}\Big(\frac1 {\underline J}\Big)^{\frac{p-1} p}\|(\varrho_0^\alpha)_y\|_{L^p}\Big(\int\varrho_0^{\alpha+1}|\frac{\Theta_y} J|^pJdy\Big)^{\frac{p-1} p}\\
&&\leqslant C(\eta)\overline\varrho^{-(\alpha+1)(p-1)}\Big(\frac1 {\underline J}\Big)^{p-1}\|(\varrho_0^\alpha)_y\|_{L^p}^p+\eta\int\varrho_0^{\alpha+1}|\frac{\Theta_y} J|^pJdy,\\
&|I_4|&=\Big|\frac{1}{2}\int R\varrho\Theta v_y \varrho_0^\alpha dy\Big|\\
&&\leqslant C\Big(\frac1 {\underline J}\Big)\overline{J}\left(\int\varrho_0^{\alpha+1}\Theta dy\right)
\left(\int\varrho_0^{\alpha+1}v_y^2 dy\right)^{\frac{1}{q+2}}\left(\int\varrho_0^{\alpha}\left|\frac{v_y}{J}\right|^{q-2}\frac{v_{yy}^2}{J} dy\right)^{\frac{1}{q+2}}\\
&&\leqslant\Big(\int\varrho_0^{\alpha+1}\Theta dy+\int\varrho_0^{\alpha+1}v_y^2dy\Big)^{\frac{2q} {q-1}}+\eta\int\varrho_0^{\alpha}\left|\frac{v_y}{J}\right|^{q-2}\frac{v_{yy}^2}{J} dy+\
C(\eta)\left(\Big(\frac1 {\underline J}\Big)\overline{J}\right)^{\frac{2q(q+2)}{q^2-2q+7}}
\end{eqnarray*}
hold for any fixed $\epsilon\in(0,1)$ and $\eta\in(0,1).$ Thus, one obtains that
\begin{eqnarray}\label{3dlg-E9}
&&\frac d {dt}\int\Big(\varrho_0^{\alpha+1}v^2+\varrho_0^{\alpha+2}\Theta\Big)dy+\int\varrho_0^\alpha|\frac{v_y}{J}|^qJdy\\
&&\leqslant \Big(\int\varrho_0^{\alpha+1}v^2dy+\int\varrho_0^{\alpha+1}\Theta dy+\int\varrho_0^{\alpha+1}v_y^2dy+\int\varrho_0^{\alpha+2}\Theta^2dy\Big)^{\frac{2q}{q-1}}\nonumber\\
&&+\eta\int\varrho_0^{\alpha+1}|\frac{\Theta_y} J|^pJdy+\eta\int\varrho_0^{\alpha}\left|\frac{v_y}{J}\right|^{q-2}\frac{v_{yy}^2}{J} dy+C(\eta)\overline\varrho^{-\frac{(3q-2)\alpha+q}{2}}\overline{J}^3\Big(\frac1 {\underline J}\Big)^3\|(\varrho_0^\alpha)_y\|_{L^\infty}^{2}\|(\varrho_0^\alpha)_y\|_{L^p}^2\nonumber
    \end{eqnarray}
holds for any fixed $\eta\in(0,1).$

{\sl\textit{Step 2.}} Multiplying equation (\ref{1dlg-E15}) by $\varrho_0^{\alpha+1}\Theta$ and integrating the resultant over $\mathbb R,$ one gets that
\begin{eqnarray}\label{3dlg-E10}
&&\frac1 2\frac d {dt}\int\varrho_0^{\alpha+2}\Theta^2dy+\int\varrho_0^{\alpha+1}|\frac{\Theta_y} J|^pJdy\nonumber\\
&&=-\int|\frac{\Theta_y} J|^{p-2}\frac {\Theta_y}J(\varrho_0^{\alpha+1})_y\Theta dy-\int\varrho_0^{\alpha+1}R\varrho\Theta^2{v_y}dy-\int\varrho_0^{\alpha+1}\Theta|\frac{v_y}J|^{q}Jdy\nonumber\\
&&:=\sum_{i=1}^3{\mathbb A}_i.
\end{eqnarray}
Using Lemma \ref{2dlg-L2.1}, one estimates each term on the right hand of (\ref{3dlg-E10}) as follows
\begin{eqnarray*}
&|{\mathbb A}_1|&=\Big|-\int|\frac{\Theta_y} J|^{p-2}\frac {\Theta_y}J(\varrho_0^{\alpha+1})_y\Theta dy\Big|\nonumber\\
&&\leqslant C\overline\varrho^{{\frac{-\alpha} {2}}-{\frac{(1+\alpha)(p-1)}{p}}}\Big(\frac1 {\underline J}\Big)^{\frac{p-1}p}\|(\varrho_0^{\alpha})_y\|_{L^{\frac{2p}{2-p}}}\Big(\int\varrho_0^{\alpha+1}|\frac{\Theta_y} J|^pJdy\Big)^{\frac{p-1}{p}}\Big(\int\varrho_0^{\alpha+2}\Theta^2dy\Big)^{\frac1 2}\nonumber\\
&&\leqslant \Big(\int\varrho_0^{\alpha+2}\Theta^2dy\Big)^{\frac{2q} {q-1}}+\epsilon\int\varrho_0^{\alpha+1}|\frac{\Theta_y} J|^pJdy+ C(\epsilon)\overline\varrho^{-\frac{(\alpha-1)4pq}{p-pq+4q}}\Big(\frac1 {\underline J}\Big)^{\frac{4q(p-1)}{p-pq+4q}}\|(\varrho_0^\alpha)_y\|_{L^{\frac{2p}{2-p}}}^{\frac{4pq} {p-pq+4q}},\\
&|{\mathbb A}_2|&=\Big|-\int\varrho_0^{\alpha+1}R\varrho\Theta^2{v_y}dy\Big|\nonumber\\
&&\leqslant \Big(\int\varrho_0^{\alpha+2}\Theta^2dy\Big)^{\frac{2q} {q-1}}+\eta\int\varrho_0^\alpha|\frac{v_y}{J}|^qJdy+\eta\int\varrho_0^\alpha|\frac{v_y}{J}|^{q-2}\frac {v_{yy}^2}Jdy+ C(\eta)\overline\varrho^{\frac{-2\alpha} {q-1}}\Big(\overline J\Big)^2,\nonumber\\
&|{\mathbb A}_3|&=\left|-\int\varrho_0^{\alpha+1}\Theta|\frac{v_y}J|^{q}Jdy\right|\\
&&\leqslant C\overline{\varrho}\left(\int\varrho_0^{\alpha}|\frac{v_y}J|^{q}Jdy\right)\|\Theta\|_{L^\infty}\\
&&\leqslant C\overline{\varrho}^{-\frac{2\alpha+1}{4}}\left(\int\varrho_0^{\alpha}|\frac{v_y}J|^{q}Jdy\right)
\left(\int\varrho_0^{\alpha+1}\Theta_y^2dy\right)^{\frac14}
\left(\int\varrho_0^{\alpha+2}\Theta^2dy\right)^{\frac14}\\
&&\leqslant
\left(\int\varrho_0^{\alpha+1}\Theta_y^2dy+\int\varrho_0^{\alpha+2}\Theta^2dy+\int\varrho_0^{\alpha}|\frac{v_y}J|^{q}Jdy\right)^{\frac{2q}{q-1}}
+C\left(\overline{\varrho}^{-\frac{2\alpha+1}{4}}\right)^{\frac{2q}{q+1}}
    \end{eqnarray*}
hold for any fixed $\epsilon\in(0,1)$ and $\eta\in(0,1).$  So, it is deduced from (\ref{3dlg-E10})
that
\begin{eqnarray}\label{3dlg-E11}
&&\frac d {dt}\int\varrho_0^{\alpha+2}\Theta^2dy+\int\varrho_0^{\alpha+1}|\frac{\Theta_y} J|^pJdy\nonumber\\
&&\leqslant \left(\int\varrho_0^{\alpha+1}\Theta_y^2dy+\int\varrho_0^{\alpha+2}\Theta^2dy+\int\varrho_0^{\alpha}|\frac{v_y}J|^{q}Jdy\right)^{\frac{2q}{q-1}}
+\eta\int\varrho_0^\alpha|\frac{v_y}{J}|^qJdy\nonumber\\
&&+\eta\int\varrho_0^\alpha|\frac{v_y}{J}|^{q-2}\frac {v_{yy}^2}Jdy+C(\eta)\overline\varrho^{-\frac{2\alpha}{q-1}}\Big(\frac1{\underline J}\Big)^{2}{\overline J}^{2q}\|(\varrho_0^\alpha)_y\|_{L^{\frac{2p}{2-p}}}^{2q}+C\overline{\varrho}^{-\frac{2\alpha+1}{2(q+1)}}
    \end{eqnarray}
holds for any fixed $\eta\in(0,1).$

{\sl\textit{Step 3.}} Differentiating (\ref{1dlg-E14}) with respect to $y,$ multiplying both sides of  the resultant by $\varrho_0^{\alpha}v_y$ and integrating  over $\mathbb R$, one finds that
\begin{eqnarray}\label{3dlg-E12}
&&\frac1 2\frac d {dt}\int\varrho_0^{\alpha+1}v_y^2dy+(q-1)\int\varrho_0^\alpha|\frac{v_y}{J}|^{q-2}\frac {v_{yy}^2}Jdy\nonumber\\
&&=-\int\varrho_{0y}\varrho_0^{\alpha}v_tv_ydy-(q-1)\int|\frac{v_y}{J}|^{q-2}\frac {v_{yy}} {J}(\varrho_0^\alpha)_yv_ydy+(q-1)\int|\frac{v_y}{J}|^{q}J_y (\varrho_0^\alpha)_ydy\nonumber\\
&&+(q-1)\int\varrho_0^{\alpha}|\frac{v_y}{J}|^{q-2}\frac {{v_y}{J_y}} {J^2}v_{yy}dy+\int R(\varrho\Theta)_y(\varrho_0^\alpha)_{y}v_ydy+\int R\varrho_0^\alpha(\varrho\Theta)_yv_{yy}dy\nonumber\\
&&:=\sum_{i=1}^6{\mathbb B}_i.
\end{eqnarray}
First, ${\mathbb B}_1$-${\mathbb B}_4$ are estimated as follows
\begin{eqnarray*}
&|{\mathbb B}_1|&=\Big|-\int\varrho_{0y}\varrho_0^{\alpha}v_tv_ydy\Big|\\
&&\leqslant C\overline\varrho^{-\alpha}\|(\varrho_0^{\alpha})_y\|_{L^\infty}\Big(\int\varrho_0^{\alpha+1}v_t^2dy\Big)^{\frac1 2}\Big(\int\varrho_0^{\alpha+1}v_y^2dy\Big)^{\frac1 2}\\
&&\leqslant \Big(\int\varrho_0^{\alpha+1}v_y^2dy\Big)^{\frac{2q}{q-1}}
+\eta\int\varrho_0^{\alpha+1}v_t^2dy+C(\eta)\overline\varrho^{-\frac{4q\alpha}{1+q}}\|(\varrho_0^\alpha)_y\|_{L^\infty}^{\frac{4q}{1+q}},\\
&|{\mathbb B}_2|&=\Big|-\int(q-1)|\frac{v_y}{J}|^{q-2}\frac {v_{yy}} {J}(\varrho_0^\alpha)_yv_ydy\Big|\\
&&\leqslant C\overline\varrho^{-\frac{\alpha}2}\Big(\frac1 {\underline J}\Big)^{\frac {q-1} {2}}\|(\varrho_0^{\alpha})_y\|_{L^2}\Big(\int\varrho_0^\alpha|\frac{v_y}{J}|^{q-2}\frac {v_{yy}^2}Jdy\Big)^{\frac1 2}\|v_y\|_{L^\infty}^{\frac q 2}\\
&&\leqslant \Big(\int\varrho_0^{\alpha+1}v_y^2dy\Big)^{\frac{2q}{q-1}}+\epsilon\int\varrho_0^\alpha|\frac{v_y}{J}|^{q-2}\frac {v_{yy}^2}Jdy\\
&&+C(\epsilon)\overline\varrho^{-\frac{16\alpha}{5-q}}\Big(\frac1 {\underline J}\Big)^{\frac {2(q-1)(q+2)}{5-q}}{\overline J}^{\frac {4(q-1)}{5-q}}\|(\varrho_0^\alpha)_y\|_{L^2}^{\frac{2(q+2)} {5-q}},\\
&|{\mathbb B}_3|&=\Big|\int(q-1)|\frac{v_y}{J}|^{q}J_y (\varrho_0^\alpha)_ydy\Big|\\
&&\leqslant C\overline\varrho^{-\frac{2\alpha+1}2}\left(\frac{1}{\underline{J}}
\right)^{q}\|(\varrho_0^{\alpha})_y\|_{L^\infty}\Big(\int\varrho_0^\alpha J_y^2dy\Big)^{\frac1 2}\Big(\int\varrho_0^{\alpha+1} v_y^2dy\Big)^{\frac1 2}\|v_y\|_{L^\infty}^{q-1}\\
&&\leqslant \Big(\int\varrho_0^{\alpha} J_y^2dy+\int\varrho_0^{\alpha+1}v_y^2dy\Big)^{\frac{2q} {q-1}}+\epsilon\int\varrho_0^\alpha|\frac{v_y}{J}|^{q-2}\frac {v_{yy}^2}Jdy\\
&&+C(\epsilon)\left(\overline\varrho^{-\frac{2\alpha+1}2}\left(\frac{1}{\underline{J}}
\right)^{q}{\overline J}^{\frac{(q-1)^2}{q+2}}\|(\varrho_0^{\alpha})_y\|_{L^\infty}\right)^{\frac{4q(q+2)}{17q-q^2+2}},\\
&|{\mathbb B}_4|&=\Big|\int\varrho_0^{\alpha}(q-1)|\frac{v_y}{J}|^{q-2}\frac {{v_y}{J_y}} {J^2}v_{yy}dy\Big|\\
&&\leqslant C\Big(\frac1 {\underline J}\Big)^{\frac {q+1}{2}}\Big(\int\varrho_0^\alpha|\frac{v_y}{J}|^{q-2}\frac {v_{yy}^2}Jdy\Big)^{\frac1 2}\Big(\int\varrho_0^\alpha J_y^2dy\Big)^{\frac1 2}\|v_y\|_{L^\infty}^{\frac q 2}\\
&&\leqslant \Big(\int\varrho_0^{\alpha} J_y^2dy+\int\varrho_0^{\alpha+1}v_y^2dy\Big)^{\frac{2q} {q-1}}+\epsilon\int\varrho_0^\alpha|\frac{v_y}{J}|^{q-2}\frac {v_{yy}^2}Jdy\\
&&+C(\epsilon)\left(\Big(\frac1 {\underline J}\Big)^{\frac {q+1}{2}}\overline{\varrho}^{-\frac{q(2\alpha+1)}{2(q+2)}}{\overline J}^{\frac{q(q-1)}{2(q+2)}}\right)^{\frac{2q(q+2)}{2q-q^2+1}}
    \end{eqnarray*}
hold for any fixed $\epsilon\in(0,1)$ and $\eta\in(0,1).$ For ${\mathbb B}_5,$ one obtains that
\begin{eqnarray*}
&|{\mathbb B}_5|&=\Big|\int R(\varrho\Theta)_y(\varrho_0^\alpha)_{y}v_ydy\Big|\\
&&\leqslant R\left(
\Big|\int\frac{ 1}{ J}\varrho_{0y}\Theta(\varrho_0^\alpha)_{y}v_ydy\Big|+\Big|\int \frac 1{J^2}\Theta \varrho_0J_y(\varrho_0^\alpha)_{y}v_ydy\Big|+\Big|\int \frac{ 1}{ J}\varrho_0\Theta_y(\varrho_0^\alpha)_{y}v_ydy\Big|\right)\\
&&:=\sum_{n=1}^3|N_n|,
\end{eqnarray*}
where
\begin{eqnarray*}
&|N_1|&=R\Big|\int\frac{ 1}{ J}\varrho_{0y}\Theta(\varrho_0^\alpha)_{y}v_ydy\Big|\\
&&\leqslant C\overline\varrho^{-\frac{4\alpha+1}2}\overline{J}\Big(\frac1 {\underline J}\Big)\|(\varrho_0^{\alpha})_y\|^2_{L^\infty}\Big(\int\varrho_0^{\alpha+2}\Theta^2dy\Big)^{\frac1 2}\Big(\int\varrho_0^{\alpha+1}v_y^2dy\Big)^{\frac1 2}\\
&&\leqslant \Big(\int\varrho_0^{\alpha+2}\Theta^2dy+\int\varrho_0^{\alpha+1}v_y^2dy\Big)^{\frac{2q}{q-1}}+C\left(\overline\varrho^{-\frac{4\alpha+1}2}\overline{J}\Big(\frac1 {\underline J}\Big)\|(\varrho_0^{\alpha})_y\|^2_{L^\infty}\right)^{\frac{2q}{q+1}},\\
&|N_2|&=R\Big|\int \frac 1{J^2}\Theta \varrho_0J_y(\varrho_0^\alpha)_{y}v_ydy\Big|\\
&&\leqslant C\overline\varrho^{-\alpha}\Big(\frac1 {\underline J}\Big)^2\overline{J}\|(\varrho_0^{\alpha})_y\|_{L^\infty}\Big(\int\varrho_0^{\alpha+2}\Theta^2dy\Big)^{\frac1 2}\Big(\int\varrho_0^{\alpha} J_y^2dy\Big)^{\frac1 2}\|v_y\|_{L^\infty}\\
\\&&\leqslant\Big(\int\varrho_0^{\alpha} J_y^2dy+\int\varrho_0^{\alpha+2}\Theta^2dy+\int\varrho_0^{\alpha+1}v_y^2dy\Big)^{\frac{2q} {q-1}}+\epsilon\int\varrho_0^\alpha|\frac{v_y}{J}|^{q-2}\frac {v_{yy}^2}Jdy\\
&&+C(\epsilon)\left(\overline\varrho^{-\frac{\alpha(q+4)+1}{q+2}}{\overline J}^{\frac{2q+1}{q+2}}\Big(\frac1 {\underline J}\Big)^2\|(\varrho_0^{\alpha})_y\|_{L^\infty}\right)^{\frac{2q(q+2)}{q^2+3}},\\
&|N_3|&=R\Big|\int \frac{ 1}{ J}\varrho_0\Theta_y(\varrho_0^\alpha)_{y}v_ydy\Big|\\
&&\leqslant C\overline{J}\overline\varrho^{-\alpha}\Big(\frac1 {\underline J}\Big)\|(\varrho_0^{\alpha})_y\|_{L^\infty}\Big(\int\varrho_0^{\alpha+1}\Theta_y^2dy\Big)^{\frac1 2}\Big(\int\varrho_0^{\alpha+1}v_y^2dy\Big)^{\frac1 2}\\
&&\leqslant \Big(\int\varrho_0^{\alpha+1}\Theta_y^2dy+\int\varrho_0^{\alpha+1}v_y^2dy\Big)^{\frac{2q} {q-1}}+C\left(\overline\varrho^{-\alpha}\overline{J}\Big(\frac1 {\underline J}\Big)\|(\varrho_0^\alpha)_y\|_{L^\infty}\right)^{\frac{2q} {q+1}}.
\end{eqnarray*}
For ${\mathbb B}_6,$ one estimates as follows
\begin{eqnarray*}
&|{\mathbb B}_6|&=\Big|\int R\varrho_0^\alpha(\varrho\Theta)_yv_{yy}dy\Big|\\
&&\leqslant R\left(
\Big|\int\frac{ 1}{ J}\varrho_{0y}\varrho_0^\alpha\Theta v_{yy}dy\Big|+\Big|\int \frac 1{J^2}\varrho_0J_y\varrho_0^\alpha\Theta v_{yy}dy\Big|+\Big|\int \frac{ 1}{ J}\varrho_0\varrho_0^\alpha \Theta_yv_{yy}dy\Big|\right)\\
&&:=\sum_{k=1}^3|{\widetilde N}_k|,
\end{eqnarray*}
where
\begin{eqnarray*}
&|{\widetilde N}_1|&=R\Big|\int\frac{ 1}{ J}\varrho_{0y}\varrho_0^\alpha\Theta v_{yy}dy\Big|\\
&&\leqslant C\overline{\varrho}^{-(\alpha+1)}\Big(\frac1 {\underline J}\Big)^{\frac{q-1}2}\|(\varrho_0^{\alpha})_y\|_{L^\infty}
\overline{J}\Big(\int\varrho_0^{\alpha+2}\Theta^2dy\Big)^{\frac12}\Big(\int\varrho_0^\alpha|\frac{v_y}{J}|^{q-2}\frac {v_{yy}^2}Jdy\Big)^{\frac1 2}\|v_y\|_{L^\infty}^{\frac{2-q}2}\\
&&\leqslant \left(\int\varrho_0^{\alpha+1}v_y^2dy+\int\varrho_0^{\alpha+2}\Theta^2dy\right)^{\frac{2q}{q-1}}+\epsilon\int\varrho_0^\alpha|\frac{v_y}{J}|^{q-2}\frac {v_{yy}^2}Jdy\\
&&+C(\epsilon)\left(\overline{\varrho}^{-\frac{2(2\alpha+1)}{q+2}}\overline{J}^{\frac{(q-1)(2-q)}{2(q+2)}}\Big(\frac1 {\underline J}\Big)^{\frac{q-1}2}\|(\varrho_0^{\alpha})_y\|_{L^\infty}\right)^{\frac{q(q+2)}{q^2-q+1}},\\
&|{\widetilde N}_2|&=\Big|\int \frac 1{J^2}\varrho_0J_y\varrho_0^\alpha\Theta v_{yy}dy\Big|\\
&&\leqslant C\overline\varrho\Big(\frac1 {\underline J}\Big)^{\frac{7-q}2}\Big(\int\varrho_0^{\alpha} J_y^2dy\Big)^{\frac1 2}\Big(\int\varrho_0^\alpha|\frac{v_y}{J}|^{q-2}\frac {v_{yy}^2}Jdy\Big)^{\frac1 2}\|v_y\|_{L^\infty}^{\frac{2-q}2}\|\Theta\|_{L^\infty}\\
&&\leqslant \left(\int\varrho_0^{\alpha+1}v_y^2dy+\int\varrho_0^{\alpha+2}\Theta^2dy+\int\varrho_0^{\alpha+1}\Theta_y^2dy+\int\varrho_0^{\alpha} J_y^2dy\right)^{\frac{2q}{q-1}}+\epsilon\int\varrho_0^\alpha|\frac{v_y}{J}|^{q-2}\frac {v_{yy}^2}Jdy\\
&&+C(\epsilon)\left(\overline{\varrho}^{-\frac{12\alpha-2q\alpha-3q+2}{4(q+2)}}\overline{J}^{\frac{(q-1)(2-q)}{2(q+2)}}\Big(\frac1 {\underline J}\Big)^{\frac{3-q}2}\right)^{\frac{4q(q+2)}{5q^2-7q+6}},\\
&|{\widetilde N}_3|&=\Big|\int \frac{ 1}{ J}\varrho_0\varrho_0^\alpha\Theta_y v_{yy}dy\Big|\\
&&\leqslant C\overline\varrho^{\frac1 2}\Big(\frac1 {\underline J}\Big)^{\frac{5-q}2}\Big(\int\varrho_0^{\alpha+1}\Theta_y^2dy\Big)^{\frac1 2}\Big(\int\varrho_0^\alpha|\frac{v_y}{J}|^{q-2}\frac {v_{yy}^2}Jdy\Big)^{\frac1 2}\|v_y\|_{L^\infty}^{\frac{2-q}2}\\
&&\leqslant \left(\int\varrho_0^{\alpha+1}v_y^2dy+\int\varrho_0^{\alpha+1}\Theta_y^2dy\right)^{\frac{2q}{q-1}}+\epsilon\int\varrho_0^\alpha|\frac{v_y}{J}|^{q-2}\frac {v_{yy}^2}Jdy\\
&&+C(\epsilon)\left(\overline{\varrho}^{-\frac{2(2-q)\alpha-2q}{q+2}}\overline{J}^{\frac{(q-1)(2-q)}{2(q+2)}}\Big(\frac1 {\underline J}\Big)^{\frac{5-q}2}\right)^{\frac{q(q+2)}{q^2-q+1}}
\end{eqnarray*}
hold for any fixed $\epsilon\in(0,1)$ and $\eta\in(0,1).$ Summary, one gets that
\begin{eqnarray}\label{3dlg-E13}
&&\frac d {dt}\int\varrho_0^{\alpha+1}v_y^2dy+\int\varrho_0^\alpha|\frac{v_y}{J}|^{q-2}\frac {v_{yy}^2}Jdy\nonumber\\
&&\leqslant \Big(\int\varrho_0^{\alpha+2}\Theta^2dy+\int\varrho_0^{\alpha} J_y^2dy+\int\varrho_0^{\alpha+1}\Theta_y^2dy+\int\varrho_0^{\alpha+1}v_y^2dy\Big)^{\frac{2q} {q-1}}+\eta\int\varrho_0^{\alpha+1}v_t^2dy\nonumber\\
&&+C(\eta)\overline\varrho^{-5\alpha}{\overline J}^{8}\|(\varrho_0^\alpha)_y\|_{L^\infty}^{4}\left(\frac{1}{\underline{J}}\right)^{12}.
\end{eqnarray}
holds for any given $\eta\in(0,1).$

{\sl\textit{Step 4.}} Differentiating (\ref{1dlg-E15}) with respect to $y,$  multiplying both sides of  the resultant by $\varrho_0^{\alpha}\Theta_y$ and  integrating it over $\mathbb R$, one gets that
\begin{eqnarray}\label{3dlg-E14}
&&\frac1 2\frac d {dt}\int\varrho_0^{\alpha+1}\Theta_y^2dy+(p-1)\int\varrho_0^\alpha|\frac{\Theta_y} J|^{p-2}\frac {\Theta_{yy}^2}Jdy\nonumber\\
&&=-\int\varrho_{0y}\varrho_0^{\alpha}\Theta_t\Theta_ydy-\int(p-1)|\frac{\Theta_y} J|^{p-2}\frac {\Theta_{yy}} {J}(\varrho_0^\alpha)_y\Theta_ydy+\int(p-1)|\frac{\Theta_y} J|^{p}J_y (\varrho_0^\alpha)_ydy\nonumber\\
&&+\int\varrho_0^{\alpha}(p-1)|\frac{\Theta_y} J|^{p-2}\frac {{\Theta_y}{J_y}} {J^2}\Theta_{yy}dy+\int R\varrho\Theta v_y(\varrho_0^\alpha)_{y}\Theta_ydy+\int R\varrho\Theta v_{y}\varrho_0^{\alpha}\Theta_{yy}dy\nonumber\\
&&-\int(\varrho_0^{\alpha})_y\Theta_y|\frac{v_y}J|^{q}Jdy
-\int\varrho_0^{\alpha}\Theta_{yy}|\frac{v_y}J|^{q}Jdy
:=\sum_{i=1}^8{\mathbb C}_i.
\end{eqnarray}
Due to Lemma \ref{2dlg-L2.1} and H\"{o}lder inequality, it is easy to get that
\begin{eqnarray*}
&|{\mathbb C}_1|&=\Big|-\int\varrho_{0y}\varrho_0^{\alpha}\Theta_t\Theta_ydy\Big|\nonumber\\
&&\leqslant C\overline\varrho^{-\alpha}\|(\varrho_0^{\alpha})_y\|_{L^\infty}\Big(\int\varrho_0^{\alpha+3}\Theta_t^2dy\Big)^{\frac1 2}\Big(\int\varrho_0^{\alpha+1}\Theta_y^2dy\Big)^{\frac1 2}\\
&&\leqslant \Big(\int\varrho_0^{\alpha+1}\Theta_y^2dy\Big)^{\frac{2q} {q-1}}+\eta\int\varrho_0^{\alpha+3}\Theta_t^2dy
+C(\eta)\overline\varrho^{-\frac{4q\alpha}{q+1}}\|(\varrho_0^\alpha)_y\|_{L^\infty}^{\frac{4q}{q+1}},\\
&|{\mathbb C}_2|&=\Big|-\int(p-1)|\frac{\Theta_y} J|^{p-2}\frac{\Theta_{yy}}{J}(\varrho_0^\alpha)_y\Theta_ydy\Big|\\
&&\leqslant C\overline\varrho^{-\frac{(p+2)\alpha+p}{4}}\Big(\frac{1}{\underline{J}}\Big)^{\frac{2p-1}{2}}\|(\varrho_0^{\alpha})_y\|_{L^\frac{4}{2-p}}\Big(\int\varrho_0^\alpha
|\frac{\Theta_y}{ J}|^{p-2}\frac{\Theta _{yy}^2}{J}dy\Big)^{\frac1 2}\Big(\int\varrho_0^{\alpha+1}\Theta_y^2dy\Big)^{\frac{p}{4}}\\
&&\leqslant \Big(\int\varrho_0^{\alpha+1}\Theta_y^2dy\Big)^{\frac{2q}{q-1}}+\epsilon\int\varrho_0^\alpha|\frac{\Theta _y} J|^{p-2}\frac {\Theta _{yy}^2}Jdy\\
&&+C(\epsilon)\left(\overline\varrho^{-\frac{(p+2)\alpha+p}{4}}\Big(\frac{1}{\underline{J}}\Big)^{\frac{2p-1}{2}}\|(\varrho_0^{\alpha})_y\|_{L^\frac{4}{2-p}}\right)^{\frac{8q}{4q+p-pq}},\\
&|{\mathbb C}_3|&=\Big|(p-1)\int|\frac{\Theta_y} J|^{p}J_y (\varrho_0^\alpha)_ydy\Big|\\
&&\leqslant C\overline\varrho^{\frac{-\alpha} {2}}\Big(\frac1 {\underline J}\Big)^p\|(\varrho_0^{\alpha})_y\|_{L^\infty}\Big(\int\varrho_0^{\alpha} J_y^2dy\Big)^{\frac1 2}\|\Theta_y\|_{L^\infty}^p\\
&&\leqslant \Big(\int\varrho_0^{\alpha} J_y^2dy+\int\varrho_0^{\alpha+1}\Theta_y^2dy\Big)^{\frac{2q} {q-1}}+\epsilon\int\varrho_0^\alpha|\frac{\Theta _y} J|^{p-2}\frac {\Theta _{yy}^2}Jdy\\
&&+C(\epsilon)\left(\overline\varrho^{-\frac{5p\alpha+2p+2\alpha}{2(p+2)}}\overline{ J}^{\frac{p-1}{p+2}}\|(\varrho_0^\alpha)_y\|_{L^\infty}\right)^{\frac{4q(p+2)}{6q-pq-p+2}},\\
&|{\mathbb C}_4|&=\Big|(p-1)\int\varrho_0^{\alpha}|\frac{\Theta_y} J|^{p-2}\frac {{\Theta_y}{J_y}} {J^2}\Theta_{yy}dy\Big|\\
&&\leqslant C\Big(\frac1 {\underline J}\Big)^{\frac {p+1} 2}\Big(\int\varrho_0^{\alpha} J_y^2dy\Big)^{\frac1 2}\Big(\int\varrho_0^\alpha|\frac{\Theta _y} J|^{p-2}\frac {\Theta _{yy}^2}Jdy\Big)^{\frac1 2}\|\Theta_y\|_{L^\infty}^{\frac p 2}\\
&&\leqslant \Big(\int\varrho_0^{\alpha} J_y^2dy+\int\varrho_0^{\alpha+1}\Theta_y^2dy\Big)^{\frac{2q}{q-1}}+\epsilon\int\varrho_0^\alpha|\frac{\Theta _y} J|^{p-2}\frac {\Theta _{yy}^2}Jdy\\
&&+C(\epsilon)\left(\overline\varrho^{-\frac {p(2\alpha+1)}{2(p+2)}}\Big(\frac1 {\underline J}\Big)^{\frac {p-1}{2}}{\overline J}^{\frac {p(p-1))}{2(p+2)}}\right)^\frac{2q(p+2)}{q+p+1-pq},\\
&|{\mathbb C}_5|&=\Big|\int R\varrho\Theta v_y(\varrho_0^\alpha)_{y}\Theta_ydy\Big|\\
&&\leqslant C\overline\varrho^{-\alpha}\Big(\frac1 {\underline J}\Big)\|(\varrho_0^{\alpha})_y\|_{L^\infty}\overline{J}\Big(\int\varrho_0^{\alpha+2}\Theta^2dy\Big)^{\frac1 2}\Big(\int\varrho_0^{\alpha+1}\Theta_y^2dy\Big)^{\frac1 2}\|v_y\|_{L^\infty}\\
&&\leqslant \Big(\int\varrho_0^{\alpha+2}\Theta^2dy+\int\varrho_0^{\alpha+1}\Theta_y^2dy	+\int\varrho_0^{\alpha+1}v_y^2dy\Big)^{\frac{2q}{q-1}}+\eta\int\varrho_0^\alpha|\frac{v_y}{J}|^{q-2}\frac {v_{yy}^2}Jdy \\
&&+C(\eta)\left(\overline\varrho^{-\frac {(q+4)\alpha+1}{q+2}}\Big(\frac1 {\underline J}\Big){\overline J}^\frac{2q+1}{q+2}\|(\varrho_0^{\alpha})_y\|_{L^\infty}\right)^{\frac {2q(q+2)}{q^2+3}},\\
&|{\mathbb C}_6|&=\Big|\int R\varrho\Theta v_{y}\varrho_0^{\alpha}\Theta_{yy}dy\Big|\\
&&\leqslant C\Big(\frac1 {\underline J}\Big)^{\frac{3-p} 2}\Big(\int\varrho_0^{\alpha+2}\Theta^2dy\Big)^{\frac1 2}\Big(\int\varrho_0^\alpha|\frac{\Theta _y} J|^{p-2}\frac {\Theta _{yy}^2}Jdy\Big)^{\frac1 2}\|v_y\|_{L^\infty}\|\Theta_y\|_{L^\infty}^{\frac{2-p} 2}\\
 &&\leqslant \Big(\int\varrho_0^{\alpha+1}\Theta_y^2dy+\int\varrho_0^{\alpha+1}v_y^2dy+\int\varrho_0^{\alpha+2}\Theta^2dy\Big)^{\frac{2q} {q-1}}\\
 && +\eta\int\varrho_0^\alpha|\frac{v_y}{J}|^{q-2}\frac {v_{yy}^2}Jdy+\epsilon\int\varrho_0^\alpha|\frac{\Theta _y} J|^{p-2}\frac {\Theta _{yy}^2}Jdy\\
 &&+C(\epsilon,\eta)\left(\overline\varrho^{-\left(\frac {2\alpha+1}{q+2}+\frac{(2-p)(2\alpha+1)}{2(p+2)}\right)}
 \Big(\frac1 {\underline J}\Big)^{\frac {q-1}{q+2}+\frac{(2-p)(q-1)}{2(p+2)}}\overline{J}^{\frac{3-p}{2}}\right)^{\frac{2q(q+2)(p+2)}{4pq^2+2pq+2p-16q+12}},\\
&|{\mathbb C}_7|&=\Big|-\int(\varrho_0^{\alpha})_y\Theta_y|\frac{v_y}J|^{q}Jdy\Big|\\
&&\leqslant C\Big(\frac1 {\underline J}\Big)^q\|(\varrho_0^{\alpha})_y\|_{L^\infty}\overline{J}\Big(\int\varrho_0^{\alpha+2}\Theta^2dy\Big)^{\frac1 2}\Big(\int\varrho_0^{\alpha+1}v_y^2dy\Big)^{\frac1 2}\|v_y\|_{L^\infty}^{q-1}\\
&& \leqslant\Big(\int\varrho_0^{\alpha+1}\Theta_y^2dy	+\int\varrho_0^{\alpha+1}v_y^2dy\Big)^{\frac{2q}{q-1}}+\eta\int\varrho_0^\alpha|\frac{v_y}{J}|^{q-2}\frac {v_{yy}^2}Jdy \\
&&+C(\eta)\left(\overline\varrho^{-\frac {(q-1)(2\alpha+1)}{q+2}}\Big(\frac1 {\underline J}\Big)^q{\overline J}^\frac{q-1}{q+2}\|(\varrho_0^{\alpha})_y\|_{L^\infty}\right)^{\frac {4q(q+2)}{3(q^2+3)}},\\
&|{\mathbb C}_8|&=\Big|-\int\varrho_0^{\alpha}\Theta_{yy}|\frac{v_y}J|^{q}Jdy\Big|\\
&&\leqslant C\Big(\frac1 {\underline J}\Big)^{\frac{1+q-p} 2}\Big(\int\varrho_0^{\alpha}\left|\frac{v_y}{J}\right|^qJdy\Big)^{\frac1 2}\|\Theta_y\|_{L^\infty}^{\frac{2-p}{2}}\|v_y\|_{L^\infty}^{\frac{q}{2}}\\
 &&\leqslant \Big(\int\varrho_0^{\alpha+1}\Theta_y^2dy+\int\varrho_0^{\alpha+1}v_y^2dy+\int\varrho_0^{\alpha}\left|\frac{v_y}{J}\right|^qJdy\Big)^{\frac{2q} {q-1}}
 +\eta\int\varrho_0^\alpha|\frac{v_y}{J}|^{q-2}\frac {v_{yy}^2}Jdy\\
 &&+\epsilon\int\varrho_0^\alpha|\frac{\Theta _y} J|^{p-2}\frac {\Theta _{yy}^2}Jdy+C(\epsilon,\eta)\left(\overline\varrho^{-\frac {(6-p)(2\alpha+1)}{4(p+2)}}
 \Big(\frac1 {\underline J}\Big)^{\frac{1+q-p}{2}}\overline{J}^{\frac{2(q-1)}{p+2}}\right)^{\frac{4q(p+2)(q+2)}{5pq^2+12pq-4q^2+4q-2p+8}}
\end{eqnarray*}
hold for any given $\epsilon\in(0,1)$ and $\eta\in(0,1).$ Thus,
\begin{eqnarray}\label{3dlg-E15}
&&\frac d {dt}\int\varrho_0^{\alpha+1}\Theta_y^2dy+\int\varrho_0^\alpha|\frac{\Theta_y} J|^{p-2}\frac {\Theta_{yy}^2}Jdy\nonumber\\
  &&\leqslant \left(\int\varrho_0^{\alpha+1}v_y^2dy+\int\varrho_0^\alpha J_y^2dy+\int\varrho_0^{\alpha+1}\Theta_y^2dy +\int\varrho_0^{\alpha+2}\Theta^2dy\right)^{\frac{2q} {q-1}}+\eta\int\varrho_0^{\alpha+3}\Theta_t^2dy \nonumber\\
&&+\eta\int\varrho_0^\alpha|\frac{v_y}{J}|^{q-2}\frac {v_{yy}^2}Jdy+C(\eta)\left(\overline\varrho^{-10\alpha} \Big(\frac1 {\underline J}\Big)^{4}\overline{J}^{24}\|(\varrho_0^\alpha)_y\|_{L^\infty}^{16}\|(\varrho_0^\alpha)_y\|_{L^p}^{16}\right)
 \end{eqnarray}
holds for any fixed $\eta\in (0,1).$ 

{\sl\textit{Step 5.}} Multiplying both sides of (\ref{1dlg-E14}) by $\varrho_0^{\alpha}v_t$ and integrating with respect to $y$, it runs that
\begin{equation}\label{3dlg-E15}
\begin{aligned}
&\frac1 q\frac d {dt}\int\varrho_0^\alpha|\frac{v_y}{J}|^qJdy+\int\varrho_0^{\alpha+1}v_t^2dy\\
&=-\int|\frac{v_y}{J}|^{q-2}\frac {v_y}J(\varrho_0^\alpha)_yv_tdy+\frac1 q\int \varrho_0^{\alpha}|\frac{v_y}{J}|^qJ_tdy-\int\varrho_0^\alpha|\frac{v_y}{J}|^{q-2}\frac {1}{J^2}J_t{v_y^2}dy\\
&-\int R(\varrho\Theta)_y\varrho_0^\alpha v_tdy=\sum_{i=1}^4{\mathbb F}_i.
\end{aligned}
    \end{equation}
Obviously,
\begin{eqnarray*}
&&|{\mathbb F}_1|=\Big|-\int|\frac{v_y}{J}|^{q-2}\frac {v_y}J(\varrho_0^\alpha)_yv_tdy\Big|\\
&&\leqslant C\overline\varrho^{-\frac{(3q-2)\alpha+q}{2q}}\Big(\frac1 {\underline J}\Big)^{q-1}\left(\int\varrho_0^\alpha|\frac{v_y}{J}|^qJdy\right)^{\frac{q-1}{q}}\Big(\int\varrho_0^{\alpha+1}v_t^2dy\Big)^{\frac1 2}\|(\varrho_0^{\alpha})_y\|_{L^{\frac{2q}{2-q}}}\\
&&\leqslant \epsilon\int\varrho_0^{\alpha+1}v_t^2dy+\eta\int\varrho_0^\alpha|\frac{v_y}{J}|^qJdy+
C(\epsilon,\eta)\left(\overline\varrho^{-\frac{(3q-2)\alpha+q}{2q}}
\Big(\frac{1}{\underline J}\Big)^{q-1}\|(\varrho_0^{\alpha})_y\|_{L^{\frac{2q}{2-q}}}\right)^{\frac{2q}{2-q}},\\
&&|{\mathbb F}_2|=\Big|\frac1 q\int \varrho_0^{\alpha}|\frac{v_y}{J}|^qJ_tdy\Big|\\
&&\leqslant C\left(\int\varrho_0^\alpha|\frac{v_y}{J}|^qJdy\right)\|v_y\|_{L^\infty}\\
&&\leqslant C\overline{\varrho}^{-\frac{2\alpha+1}{q+2}}\overline{J}^{\frac{q-1}{q+2}}\left(\int\varrho_0^\alpha|\frac{v_y}{J}|^qJdy\right)
\left(\int\varrho_0^{\alpha+1}v_y^2dy\right)^{\frac{1}{q+2}}\left(\int\varrho_0^\alpha|\frac{v_y}{J}|^{q-2}\frac{v_{yy}^2}{J}dy\right)^{\frac{1}{q+2}}\\
&&\leqslant \left(\int\varrho_0^{\alpha+1}v_y^2dy+\int\varrho_0^\alpha|\frac{v_y}{J}|^qJdy\right)^{\frac{2q}{q-1}}
+\eta\int\varrho_0^\alpha|\frac{v_y}{J}|^{q-2}\frac {v_{yy}^2}{J}dy+C(\eta)\left(\overline{\varrho}^{-\frac{2\alpha+1}{q+2}}\overline{J}^{\frac{q-1}{q+2}}\right)^{\frac{2q(q+2)}{q^2+3}},\\
&&|{\mathbb F}_3|=|-\int\varrho_0^\alpha|\frac{v_y}{J}|^{q-2}\frac {1}{J^2}J_t{v_y^2}dy|\\
&&\leqslant C\left(\int\varrho_0^\alpha|\frac{v_y}{J}|^qJdy\right)\|v_y\|_{L^\infty}\\
&&\leqslant \left(\int\varrho_0^{\alpha+1}v_y^2dy+\int\varrho_0^\alpha|\frac{v_y}{J}|^qJdy\right)^{\frac{2q}{q-1}}
+\eta\int\varrho_0^\alpha|\frac{v_y}{J}|^{q-2}\frac {v_{yy}^2}{J}dy+C(\eta)\left(\overline{\varrho}^{-\frac{2\alpha+1}{q+2}}\overline{J}^{\frac{q-1}{q+2}}\right)^{\frac{2q(q+2)}{q^2+3}},\\
&&|{\mathbb F}_4|=|-\int R(\varrho\Theta)_y\varrho_0^\alpha v_tdy|\\
&&\leqslant 
R|\int \frac{1}{J}\varrho_{0y}\varrho_0^{\alpha}\Theta v_tdy|
+R|\int \frac{1}{J^2}J_y\varrho_0^{\alpha+1}\Theta v_tdy|+R|\int \frac{1}{J}\varrho_0J_y\varrho_0^{\alpha}\Theta_y v_tdy|\\
&&:=\sum\limits_{i=1}^3|J_i|,
\end{eqnarray*}
where
\begin{eqnarray*}
&|J_1|&=\Big|\int \frac{1}{J}\varrho_{0y}\varrho_0^{\alpha}\Theta v_tdy\Big|\\
&&\leqslant \frac{1}{\underline{J}}\overline{J}\overline{\varrho}^{-(2\alpha+1)}\|\left(\varrho_0^\alpha\right)_y\|_{L^\infty}
\left(\int\rho_0^{\alpha+1}v_t^2dy\right)^{\frac{1}{2}}\left(\int\rho_0^{\alpha+2}\Theta^2dy\right)^{\frac{1}{2}}\\
&&\leqslant\left(\int\rho_0^{\alpha+2}\Theta^2dy\right)^{\frac{2q}{q-1}}+\epsilon\int\rho_0^{\alpha+1}v_t^2dy
+C(\epsilon)\left(\frac{1}{\underline{J}}\overline{J}\overline{\varrho}^{-(2\alpha+1)}\|\left(\varrho_0^\alpha\right)_y\|_{L^\infty}\right)^{\frac{4q}{q+1}},\\
&|J_2|&=\Big|-\int \frac{1}{J^2}J_y\varrho_0^{\alpha+1}\Theta v_tdy\Big|\\
&&\leqslant \left(\frac{1}{\underline{J}}\right)^2\overline{J}\overline{\varrho}^{-\frac{2\alpha+1}{4}}\left(\int\varrho_0^\alpha J_y^2dy\right)^\frac{1}{2}
\left(\int\rho_0^{\alpha+1}v_t^2dy\right)^{\frac{1}{2}}\left(\int\rho_0^{\alpha+2}\Theta^2dy\right)^{\frac{1}{4}}\left(\int\rho_0^{\alpha+1}\Theta_y^2dy\right)^{\frac{1}{4}}\\
&&\leqslant\left(\int\rho_0^{\alpha+2}\Theta^2dy+\int\rho_0^{\alpha+1}\Theta_y^2dy+\int\varrho_0^\alpha J_y^2dy\right)^{\frac{2q}{q-1}}+\epsilon\int\rho_0^{\alpha+1}v_t^2dy\\
&&+C(\epsilon)\left(\left(\frac{1}{\underline{J}}\right)^2\overline{J}\overline{\varrho}^{-\frac{2\alpha+1}{4}}\right)^{\frac{2q}{q+1}},\\
&|J_3|&=\Big|\int \frac{1}{J}\varrho_0J_y\varrho_0^{\alpha}\Theta_y v_tdy\Big|\\
&&\leqslant \left(\frac{1}{\underline{J}}\right)\overline{\varrho}\left(\int\rho_0^{\alpha+1}v_t^2dy\right)^{\frac{1}{2}}
\left(\int\rho_0^{\alpha+1}J_y^2dy\right)^{\frac{1}{2}}\|\Theta_y\|_{L^\infty}\\
&&\leqslant\left(\int\rho_0^{\alpha+1}\Theta_y^2dy+\int\varrho_0^\alpha J_y^2dy\right)^{\frac{2q}{q-1}}+\epsilon\int\rho_0^{\alpha+1}v_t^2dy
+\eta\int\varrho_0^\alpha|\frac{\Theta _y} J|^{p-2}\frac {\Theta _{yy}^2}Jdy\\
&&+C(\epsilon,\eta)\left(\left(\frac{1}{\underline{J}}\right)\overline{J}^{\frac{p}{2(p+2)}}
\overline{\varrho}^{-\frac{2\alpha-2p-3}{2(p+2)}}\right)^{\frac{4q(p+2)}{2pq+q+2p+5}}.
    \end{eqnarray*}
hold for any given $\epsilon\in(0,1)$ and $\eta\in(0,1).$ So,
\begin{eqnarray}\label{3dlg-E16}
&&\frac d {dt}\int\varrho_0^\alpha|\frac{v_y}{J}|^qJdy+\int\varrho_0^{\alpha+1}v_t^2dy\nonumber\\
&&\leqslant \left(\int\varrho_0^{\alpha+1}v_y^2dy+\int\rho_0^{\alpha+2}\Theta^2dy+\int\rho_0^{\alpha+1}\Theta_y^2dy+\int\varrho_0^\alpha J_y^2dy+\int\varrho_0^\alpha|\frac{v_y}{J}|^qJdy\right)^{\frac{2q}{q-1}}\nonumber\\
&&+\eta\int\varrho_0^\alpha|\frac{v_y}{J}|^qJdy+\eta\int\varrho_0^\alpha|\frac{v_y}{J}|^{q-2}\frac {v_{yy}^2}Jdy\nonumber\\
&&+C(\eta)\left(\|(\varrho_0^\alpha)_y\|_{L^q}\|(\varrho_0^\alpha)_y\|_{L^\infty}\right)^{\frac{4q}{2-q}}\overline{J}^{4}
\left(\frac{1}{\underline{J}}\right)^{\max\{\frac{2q}{2-q},4q\}}\overline{\varrho}^{-\min\{\frac{(3q-2)\alpha+q}{2-q},8\alpha\}}
\end{eqnarray}
hold for any given $\eta\in(0,1).$

{\sl\textit{Step 6.}} Multiplying equation (\ref{1dlg-E15}) by $\varrho_0^{\alpha+2}\Theta_t$ and integrating the resultant over $\mathbb R$, one obtains that
\begin{eqnarray}\label{3dlg-E17}
&&\frac1 p\frac d {dt}\int\varrho_0^{\alpha+2}|\frac{\Theta_y} J|^pJdy+\int\varrho_0^{\alpha+3}\Theta_t^2dy\nonumber\\
&&=-\int|\frac{\Theta_y} J|^{p-2}\frac {\Theta_y}J(\varrho_0^{\alpha+2})_y\Theta_tdy+\frac1 p\int \varrho_0^{\alpha+2}|\frac{\Theta_y} J|^pJ_tdy-\int\varrho_0^{\alpha+2}|\frac{\Theta_y} J|^{p-2}\frac {1}{J^2}J_t{\Theta_y^2}dy\nonumber\\
&&-\int R\varrho\varrho_0^{\alpha+2}v_y\Theta\Theta_tdy+\int|\frac{v_y}J|^{q}J\varrho_0^{\alpha+2}\Theta_tdy
:=\sum_{i=1}^5{\mathbb H}_i.
\end{eqnarray}
Each term on the right hand of (\ref{3dlg-E17}) can be estimated as follows
\begin{eqnarray*}
&|{\mathbb H}_1|&=\Big|-\int|\frac{\Theta_y} J|^{p-2}\frac {\Theta_y}J(\varrho_0^{\alpha+2})_y\Theta_tdy\Big|\\
&&\leqslant C\overline\varrho^{-\frac{(3p-2)(\alpha+1)}{2p}}\|(\varrho_0^{\alpha})_y\|_{L^\infty}\Big(\int\varrho_0^{\alpha+1}|\frac{\Theta_y} J|^pJdy\Big)^{\frac{1}{2}}\Big(\int\varrho_0^{\alpha+3}\Theta_t^2dy\Big)^{\frac12}\\
&&\leqslant\Big(\int\varrho_0^{\alpha+1}|\frac{\Theta_y} J|^pJdy\Big)^{\frac{2q}{q-1}} +\epsilon\int\varrho_0^{\alpha+3}\Theta_t^2dy
+C(\epsilon)\left(\overline\varrho^{-\frac{(3p-2)(\alpha+1)}{2p}}\|(\varrho_0^{\alpha})_y\|_{L^\infty}\right)^{\frac{4q}{q+1}},\\
&|{\mathbb H}_2|&=\Big|\frac1 p\int \varrho_0^{\alpha+2}|\frac{\Theta_y} J|^pJ_tdy\Big|\\
&&\leqslant C\Big(\frac1 {\underline J}\Big)\int\varrho_0^{\alpha+2}|\frac{\Theta_y} J|^pJdy\|v_y\|_{L^\infty}\\
&&\leqslant C\overline{\varrho}^{-\frac{2\alpha+1}{q+2}}\overline{J}^{\frac{q-1}{q+2}}\left(\frac{1}{\underline{J}}\right)\left(\int\varrho_0^{\alpha+2}|\frac{\Theta_y} J|^pJdy\right)\left(\int\varrho_0^{\alpha+1}v_y^2dy\right)^{\frac{1}{q+2}}
\left(\int\varrho_0^\alpha|\frac{v_y}{J}|^{q-2}\frac {v_{yy}^2}Jdy\right)^{\frac{1}{q+2}}\\
&&\leqslant\left(\int\varrho_0^{\alpha+2}|\frac{\Theta_y} J|^pJdy+\int\varrho_0^{\alpha+1}v_y^2dy\right)^{\frac{2q}{q-1}}
+\eta\int\varrho_0^\alpha|\frac{v_y}{J}|^{q-2}\frac {v_{yy}^2}Jdy\\
&&+C(\eta)\left(C\overline{\varrho}^{-\frac{2\alpha+1}{q+2}}\overline{J}^{\frac{q-1}{q+2}}\left(\frac{1}{\underline{J}}\right)\right)^{\frac{2q(q+2)}{q^2+3}},\\
&|{\mathbb H}_3|&=\Big|-\int\varrho_0^{\alpha+2}|\frac{\Theta_y} J|^{p-2}\frac {1}{J^2}J_t{\Theta_y^2}dy\Big|\\
&&\leqslant C\Big(\frac1 {\underline J}\Big)\int\varrho_0^{\alpha+2}|\frac{\Theta_y} J|^pJdy\|v_y\|_{L^\infty}\\
&&\leqslant\left(\int\varrho_0^{\alpha+2}|\frac{\Theta_y} J|^pJdy+\int\varrho_0^{\alpha+1}v_y^2dy\right)^{\frac{2q}{q-1}}
+\eta\int\varrho_0^\alpha|\frac{v_y}{J}|^{q-2}\frac {v_{yy}^2}Jdy\\
&&+C(\eta)\left(\overline{\varrho}^{-\frac{2\alpha+1}{q+2}}\overline{J}^{\frac{q-1}{q+2}}\left(\frac{1}{\underline{J}}\right)\right)^{\frac{2q(q+2)}{q^2+3}},\\
&|{\mathbb H}_4|&=\Big|-\int R\varrho\varrho_0^{\alpha+2}v_y\Theta\Theta_tdy\Big|\\
&&\leqslant C\overline\varrho^{\frac1 2}\Big(\frac1 {\underline J}\Big)\overline{J}\Big(\int\varrho_0^{\alpha+2}\Theta^2dy\Big)^{\frac1 2}\Big(\int\varrho_0^{\alpha+3}\Theta_t^2dy\Big)^{\frac1 2}\|v_y\|_{L^\infty}\\
&&\leqslant\left(\int\varrho_0^{\alpha+2}\Theta^2dy+\int\varrho_0^{\alpha+1}v_y^2dy\right)^{\frac{2q}{q-1}}
+\eta\int\varrho_0^\alpha|\frac{v_y}{J}|^{q-2}\frac {v_{yy}^2}Jdy+\epsilon\int\varrho_0^{\alpha+3}\Theta_t^2dy\\
&&+C(\epsilon,\eta)\left(C\overline\varrho^{-\frac{4\alpha-q}{q+2}}\Big(\frac1 {\underline J}\Big)\overline{J}^{\frac{2q+1}{q+2}}\right)^{\frac{4q(q+2)}{q^2-2q+3}},\\
&|{\mathbb H}_5|&=\Big|\int|\frac{v_y}J|^{q}J\varrho_0^{\alpha+2}\Theta_tdy\Big|\\
&&\leqslant C\overline{J}^{2-q}\|v_y\|_{L^\infty}^{q-1}\Big(\int\varrho_0^{\alpha+1}v_y^2dy\Big)^{\frac1 2}\Big(\int\varrho_0^{\alpha+3}\Theta_t^2dy\Big)^{\frac1 2}\\
&&\leqslant \left( \int\varrho_0^{\alpha+1}v_y^2dy\right)^{\frac{2q}{q-1}}+\eta\int\varrho_0^\alpha|\frac{v_y}{J}|^{q-2}\frac {v_{yy}^2}Jdy+\epsilon\int\varrho_0^{\alpha+3}\Theta_t^2dy\\
&&+C(\epsilon,\eta)\left(\Big(\frac1 {\underline J}\Big)^{\frac{5-q}{q+2}}\overline{\varrho}^{-\frac{(q-1)(2\alpha+1)}{q+2}}\right)^{\frac{4(q+2)}{11-5q}}\\
    \end{eqnarray*}
hold for any given $\epsilon\in(0,1)$ and $\eta\in(0,1).$ Thus, one obtains that
\begin{eqnarray}\label{3dlg-E18}
&&\frac d {dt}\int\varrho_0^{\alpha+2}|\frac{\Theta_y} J|^pJdy+\int\varrho_0^{\alpha+3}\Theta_t^2dy\nonumber\\
&&\leqslant \Big(\int\varrho_0^{\alpha+2}\Theta^2dy+\int\varrho_0^{\alpha+2}|\frac{\Theta_y} J|^pJdy +\int\varrho_0^{\alpha+1}v_y^2dy\Big)^{\frac{2q}{q-1}}
+\eta\int\varrho_0^\alpha|\frac{v_y}{J}|^{q-2}\frac {v_{yy}^2}Jdy\nonumber\\
&&+C(\eta)\overline{\varrho}^{-4\alpha}\overline{J}^{4}\left(\frac{1}{\underline{J}}\right)^{16}\|(\varrho_0^{\alpha})_y\|_{L^\infty}^{4}
\end{eqnarray}
for any fixed $\eta\in (0,1).$

{\sl\textit{Step 7.}} Differentiating (\ref{1dlg-E13}) with respect to $y,$ multiplying it by $\varrho_0^{\alpha}J_y$ and integrating the resultant over $\mathbb R$, one deduces that
\begin{eqnarray}\label{3dlg-E19}
&&\frac1 2\frac d {dt}\int\varrho_0^{\alpha} J_y^2dy=\int\varrho_0^{\alpha}{J_y}v_{yy}dy\nonumber\\
&&\leqslant C{\overline J}^{\frac {3-q} {2}}\|v_y\|_{L^\infty}^\frac{2-q}{2}\Big(\int\varrho_0^{\alpha} J_y^2dy\Big)^{\frac1 2}\Big(\int\varrho_0^\alpha|\frac{v_y}{J}|^{q-2}\frac {v_{yy}^2}Jdy\Big)^{\frac1 2}\nonumber\\
&&\leqslant C\overline{\varrho}^{-\frac{(3-q)(2\alpha+1)}{2(q+2)}}{\overline J}^{\frac{5-q}{2(q+2)}}\Big(\int\varrho_0^{\alpha} J_y^2dy\Big)^{\frac1 2}\left(\int\varrho_0^\alpha|\frac{v_y}{J}|^{q-2}\frac {v_{yy}^2}Jdy\right)^{\frac{2}{q+2}}\Big(\int\varrho_0^{\alpha+1}v_y^2dy\Big)^{\frac{2-q}{2(q+2)}}\nonumber\\
&&\leqslant \Big(\int\varrho_0^{\alpha}J_y^2dy+\int\varrho_0^{\alpha+1}v_y^2dy\Big)^{\frac{2q}{q-1}}+\eta\int\varrho_0^\alpha|\frac{v_y}{J}|^{q-2}\frac {v_{yy}^2}Jdy\nonumber\\
&&+C(\eta)\left(\overline{\varrho}^{-\frac{(3-q)(2\alpha+1)}{2(q+2)}}{\overline J}^{\frac{5-q}{2(q+2)}}\right)^{\frac{q(q+2)}{q^2-q+1}}
\end{eqnarray}
holds for any given $\eta\in (0,1).$

Now,  one adds up the estimates (\ref{3dlg-E9})-(\ref{3dlg-E19}) to arrive at
\begin{eqnarray}\label{3dlg-E20}
&&\frac d {dt}\int\varrho_0^{\alpha}\Big({\varrho_0}v^2+{\varrho_0}\Theta+\varrho_0^2\Theta^2+{\varrho_0}v_y^2
+{\varrho_0}\Theta_y^2+|\frac{v_y}{J}|^qJ+\varrho_0^2|\frac{\Theta_y} J|^pJ+ J_y^2\Big)dy\nonumber\\
&&+\int\varrho_0^\alpha\Big(|\frac{v_y}{J}|^qJ+{\varrho_0}|\frac{\Theta_y} J|^pJ+|\frac{v_y}{J}|^{q-2}\frac{ v_{yy}^2}J+|\frac{\Theta_y} J|^{p-2}\frac {\Theta_{yy}^2}J+\varrho_0^3\Theta_t^2+{\varrho_0}v_t^2\Big)dy\nonumber\\
&&\leqslant\bigg(\int\varrho_0^{\alpha}\Big({\varrho_0}v^2+{\varrho_0}\Theta+\varrho_0^2\Theta^2+{\varrho_0}v_y^2
+{\varrho_0}\Theta_y^2+|\frac{v_y}{J}|^qJ+\varrho_0^2|\frac{\Theta_y} J|^pJ+ J_y^2\Big)dy\bigg)^{\frac{2q}{q-1}}\nonumber\\
&&+C{\overline J}^{12}\Big(\frac1{\underline J}\Big)^{\frac{12q}{2-q}}{\overline \varrho}^{-\min\{\frac{(3q-2)\alpha+q}{2-q},\frac{10\alpha}{q-1}\}}
\left(\|(\varrho_0^\alpha)_y\|_{L^p}\|(\varrho_0^\alpha)_y\|_{L^\infty}\right)^{\frac{16q}{2-q}}+C\overline{\varrho}^{-\frac{2\alpha+1}{2(q+1)}}.
\end{eqnarray}
Set
\begin{equation*}
G_0=C{\overline J}^{12}\Big(\frac1{\underline J}\Big)^{\frac{12q}{2-q}}{\overline \varrho}^{-\min\{\frac{(3q-2)\alpha+q}{2-q},\frac{10\alpha}{q-1}\}}
\left(\|(\varrho_0^\alpha)_y\|_{L^p}\|(\varrho_0^\alpha)_y\|_{L^\infty}\right)^{\frac{16q}{2-q}}+C\overline{\varrho}^{-\frac{2\alpha+1}{2(q+1)}}
 \end{equation*}
and
\begin{eqnarray*}
&H_0(t)&=\|\varrho_0^{\frac {\alpha+2} 2}v_0\|_{L^2}^2+\|\varrho_0^{\alpha+1}\Theta_0\|_{L^1}+\|\varrho_0^{\frac {\alpha+2} 2}\Theta_0\|_{L^2}^2+\|\varrho_0^{\frac {\alpha+1} 2} v_0^{\prime}\|_{L^2}^2+\|\varrho_0^{\frac {\alpha+1} 2}\Theta_0^{\prime}\|_{L^2}^2\\
&&+\|\varrho_0^{\frac \alpha q}|\frac { v_0^{\prime}} {J_0}|{J_0}^{\frac 1 q}\|_{L^q}^q+\|\varrho_0^{\frac {\alpha+2} p}|\frac {\Theta_0^{\prime}} {J_0}|{J_0}^{\frac 1 p}\|_{L^p}^p+\|\varrho_0^{\frac \alpha 2} J_0^{\prime}\|_{L^2}^2+G_0t.\notag
 \end{eqnarray*}
According to Lemma \ref{2dlg-L2.2}, one gets that
\begin{equation}\label{3dlg-E21}
\begin{aligned}
&\int\varrho_0^{\alpha}\Big({\varrho_0}v^2+{\varrho_0}\Theta+\varrho_0^2\Theta^2+{\varrho_0}v_y^2
+{\varrho_0}\Theta_y^2+|\frac{v_y}{J}|^qJ+\varrho_0^2|\frac{\Theta_y} J|^pJ+ J_y^2\Big)dy\\
&\leqslant H_0(t)\bigg(1-\frac{q+1}{q-1}H_0^{\frac{q+1}{q-1}}(t)t\bigg)^{-\frac{q-1}{q+1}}
\end{aligned}
    \end{equation}
holds for all $t\in[0,\widetilde{t_0}]$, where  $\widetilde{t_0} < T$ and
\begin{equation}\label{2dlg-E22}
\frac{q+1}{q-1}H_0^{\frac{q+1}{q-1}}(\widetilde{t_0})\widetilde{t_0}<1.
    \end{equation}
Moreover,
\begin{eqnarray}\label{3dlg-E23}
&&\bigg(\int\varrho_0^{\alpha}\Big({\varrho_0}v^2+{\varrho_0}\Theta+\varrho_0^2\Theta^2+{\varrho_0}v_y^2
+{\varrho_0}\Theta_y^2+|\frac{v_y}{J}|^qJ+\varrho_0^2|\frac{\Theta_y} J|^pJ+ J_y^2\Big)dy\bigg)(t)\nonumber\\
&&+\int_0^t\int\varrho_0^\alpha\Big(|\frac{v_y}{J}|^qJ+{\varrho_0}|\frac{\Theta_y} J|^pJ+|\frac{v_y}{J}|^{q-2}\frac { v_{yy}^2}J+|\frac{\Theta_y} J|^{p-2}\frac {\Theta_{yy}^2}J+\varrho_0^3\Theta_t^2+{\varrho_0}v_t^2\Big)dyds\nonumber\\
&&\leqslant H_0(t)+\int_0^tH_0^{\frac{2q} {q-1}}(s)\bigg(1-\frac{q+1}{q-1}H_0^{\frac{q+1}{q-1}}(s)s\bigg)^{-\frac{2q}{q+1}}ds
\end{eqnarray}
holds on $[0,\widetilde{t_0}]$.
\end{proof}

\begin{proposition}\label{3dlg-P3.3}
Let $p,q\in(1,2)$ and $\alpha<\min\{-\frac{q}{2(q-1)},-\frac{4-p}{2-p}\}.$ Suppose that the quaternion $(J,\varrho, v,\Theta)$ is a strong solution to the problem (\ref{1dlg-E16})-(\ref{1dlg-E17}) on ${\mathbb R}\times [0, T].$  Then, there exists a constant $\widetilde{t_1}>0$ with $\widetilde{t_1}<T$, which is independent of $\underline \varrho$, such that
\begin{equation}\label{3dlg-E24}
{\frac 3 4}\underline J\leqslant J\leqslant{\frac 5 4}\overline J\ on\  \mathbb R \times [0,\widetilde{t_1}].
    \end{equation}
Moreover,
\begin{equation}\label{3dlg-E25}
\sup\limits_{y\in\mathbb R}J(y,t)\leqslant2M_0(t)\exp\bigg(\overline\varrho^{-\frac{\alpha}{q}}\frac{q-1}{q}t\bigg),
    \end{equation}
\begin{equation}\label{3dlg-E26}
\inf\limits_{y\in\mathbb R}J(y,t)\geqslant(\inf\limits_{y\in\mathbb R}J_0)\bigg(1+C(\inf\limits_{y\in\mathbb R}J_0)\Big[q(M_0(t)-\sup\limits_{y\in\mathbb R}J_0)\Big]^{\frac1 q}\Big(\int_0^tF_0(s)ds\Big)^{\frac{q-1} q}\bigg)^{-1}
    \end{equation}
on $[0,\widetilde{t_1}]$, where 
\begin{eqnarray*}
&&M_0(t)=\frac{C}{q}\bigg(q\sup\limits_{y\in\mathbb R}J_0+H_0(t)+\int_0^tH_0^{\frac{2q}{q-1}}(s)\Big(1-\frac{q+1}{q-1}H_0^{\frac{q+1}{q-1}}(s)s\Big)^{-\frac{2q}{q-1}}ds\bigg),\\
&&H_0(t)=\Big(\|\varrho_0^{\frac {\alpha+2} 2}v_0\|_{L^2}^2+\|\varrho_0^{\alpha+1}\Theta_0\|_{L^1}+\|\varrho_0^{\frac {\alpha+2} 2}\Theta_0\|_{L^2}^2+\|\varrho_0^{\frac {\alpha+1} 2} v_0^{\prime}\|_{L^2}^2+\|\varrho_0^{\frac {\alpha+1} 2}\Theta_0^{\prime}\|_{L^2}^2\\
&&~~~~~~~~~~+\|\varrho_0^{\frac \alpha q}|\frac { v_0^{\prime}} {J_0}|{J_0}^{\frac 1 q}\|_{L^q}^q+\|\varrho_0^{\frac {\alpha+2} p}|\frac {\Theta_0^{\prime}} {J_0}|{J_0}^{\frac 1 p}\|_{L^p}^p+\|\varrho_0^{\frac \alpha 2} J_0^{\prime}\|_{L^2}^2\Big)
+G_0t,\\
&&G_0=C{\overline J}^{12}\Big(\frac1{\underline J}\Big)^{\frac{12q}{2-q}}{\overline \varrho}^{-\min\{\frac{(3q-2)\alpha+q}{2-q},\frac{10\alpha}{q-1}\}}
\left(\|(\varrho_0^\alpha)_y\|_{L^p}\|(\varrho_0^\alpha)_y\|_{L^\infty}\right)^{\frac{16q}{2-q}}+C\overline{\varrho}^{-\frac{2\alpha+1}{2(q+1)}},\\
&&F_0(t)=\overline\varrho^{-\frac{\alpha}{q-1}}M_0^{-\frac{q+1}{q-1}}(t)\exp\bigg(-\frac{q+1}{q}\overline\varrho^{-\frac{\alpha}{q}}t\bigg)
\end{eqnarray*}
on $[0,\widetilde{t_1}].$
\end{proposition}

\begin{proof}
According to Lemma \ref{2dlg-L2.1} and  Proposition \ref{3dlg-P3.1}, one deduces that
\begin{equation*}\label{3dlg-E27}
\begin{aligned}
\int_0^t\|v_y\|_{L^\infty}ds&\leqslant C\int_0^t\Big(\int\varrho_0^\alpha|\frac{v_y}{J}|^qJdy\Big)^{\frac1{2q}}\Big(\int\varrho_0^\alpha|\frac{v_y}{J}|^{q-2}\frac {v_{yy}^2}Jdy\Big)^{\frac1{2q}}{\overline J}^{\frac{q-1}q}{\overline\varrho}^{\frac{-\alpha}q}ds\\
&\leqslant C\Bigg(H_0(t)+\int_0^tH_0^{\frac{2q} {q-1}}(s)\bigg(1-\frac{q+1}{q-1}H_0^{\frac{q+1}{q-1}}(s)s\bigg)^{-\frac{2q}{q+1}}ds\Bigg)^{\frac1 q}\Big(\int_0^t{\overline J}{\overline\varrho}^{-\frac{\alpha}{q-1}}ds\Big)^{\frac{q-1}{q}}.
\end{aligned}
    \end{equation*}
So, it is easy to get that
\begin{equation*}\label{3dlg-E28}
\|J-J_0\|_{L^\infty}(t)\leqslant\underline J/4
    \end{equation*}
on $[0,t_1^*]$ for some $t_1^*$. Then, ${\frac 3 4}\underline J\leqslant J\leqslant{\frac 5 4}\overline J\mbox{ on }{ \mathbb R } \times [0,t_1^*].$ Moreover,  (\ref{1dlg-E13}) implies that
\begin{eqnarray*}
&\sup\limits_{y\in\mathbb R}J(y,t)&\leqslant\sup\limits_{y\in\mathbb R}J_0+\int_0^t\|v_y\|_{L^\infty}ds\\
&&\leqslant\sup\limits_{y\in\mathbb R}J_0+ C\Bigg(H_0(t)+\int_0^tH_0^{\frac{2q} {q-1}}(s)\bigg(1-\frac{q+1}{q-1}H_0^{\frac{q+1}{q-1}}(s)s\bigg)^{-\frac{2q}{q+1}}ds\Bigg)^{\frac1 q}\\
&&~~~~~~~~~~~~~~\cdot\Big(\int_0^t\sup\limits_{y\in\mathbb R}J(y,s){\overline\varrho}^{-\frac{\alpha}{q-1}}ds\Big)^{\frac{q-1}{q}}\\
&&\leqslant \frac{C}{q}\bigg(q\sup\limits_{y\in\mathbb R}J_0+H_0(t)+\int_0^tH_0^{\frac{2q} {q-1}}(s)\bigg(1-\frac{q+1}{q-1}H_0^{\frac{q+1}{q-1}}(s)s\bigg)^{-\frac{2q}{q+1}}ds\bigg)\\
&&~~~~~~+\frac{q-1}{q}{\overline\varrho}^{-\frac{\alpha}{q}}\int_0^t\sup\limits_{y\in\mathbb R}J(y,s)ds.
\end{eqnarray*}
By Gronwall's inequality, one gets that
\begin{equation*}\label{3dlg-E29}
\sup\limits_{y\in\mathbb R}J(y,t)\leqslant2M_0(t)\exp\bigg(\overline\varrho^{-\frac{\alpha}{q}}\frac{q-1}{q}t\bigg)
    \end{equation*}
on $[0,\widetilde{t_1}]$, where $\widetilde{t_1}=\min\{T,t_1^*,\widetilde{t_0}\}$ with $\widetilde{t_0}$ stated in Proposition \ref{3dlg-P3.1} and
\begin{equation*}
M_0(t)=\frac C q\bigg(q\sup\limits_{y\in\mathbb R}J_0+H_0(t)+\int_0^tH_0^{\frac{2q} {q-1}}(s)\bigg(1-\frac{q+1}{q-1}H_0^{\frac{q+1}{q-1}}(s)s\bigg)^{-\frac{2q}{q+1}}ds\bigg).
    \end{equation*}

On the other hand, one finds that
\begin{equation*}
\Big(\frac1 J\Big)_t=-\frac{1} {J^2}v_y\leqslant\Big(\sup\limits_{y\in\mathbb R}J(y,t)\Big)^{-2}\|v_y\|_{L^\infty}\leqslant CM_0^{-2}(t)\exp\bigg(\frac{-2\overline\varrho^{-\frac{\alpha}{q}}(q-1)t}{q}\bigg)\|v_y\|_{L^\infty}.
\notag
    \end{equation*}
Thus,
\begin{eqnarray*}\label{3dlg-E30}
&&\frac1 J\leqslant\frac{1}{J_0}+C\int_0^tM_0^{-2}(s)\exp\bigg(\frac{-2\overline\varrho^{-\frac{\alpha}{q}}(q-1)s}{q}\bigg)\|v_y\|_{L^\infty}ds\nonumber\\
&&~~~~~\leqslant\frac{1}{J_0}+ C\Bigg(H_0(t)+\int_0^tH_0^{\frac{2q} {q-1}}(s)\bigg(1-\frac{q+1}{q-1}H_0^{\frac{q+1}{q-1}}(s)s\bigg)^{-\frac{2q}{q+1}}ds\Bigg)^{\frac1 q}\nonumber\\
&&~~~~~\cdot\Bigg(\int_0^t\sup\limits_{y\in\mathbb R}J(y,s){\overline\varrho}^{-\frac{\alpha}{q-1}}M_0^{\frac{-2q}{q-1}}(s)\exp\bigg(\frac{-2\overline\varrho^{-\frac{\alpha}{q}}s}{q}\bigg)ds\Bigg)^{\frac{q-1}{q}}\\
&&~~~~~\leqslant\frac{1}{J_0}+C\Big(q(M_0(t)-\sup\limits_{y\in\mathbb R}J_0)\Big)^{\frac1 q}\Bigg(\int_0^t{\overline\varrho}^{-\frac{\alpha}{q-1}}M_0^{-\frac{q+1}{q-1}}(s)\exp\bigg(\frac{-(q+1)
\overline\varrho^{-\frac{\alpha}{q}}s}{q}\bigg)ds\Bigg)^{\frac{q-1}{q}}\nonumber
    \end{eqnarray*}
holds on  $[0,\widetilde{t_1}]$.
\end{proof}

Denote
$${\mathcal L}(t):=\Big((\inf\limits_{y\in\mathbb R}J)^{-1}+\|J\|_{L^\infty}+\|J_y\|_{L^2}+\|\varrho_0^{\frac{\alpha+1}{2}}v\|_{L^2}+\|v\|_{W^{1,q}}+\|\varrho_0^{\frac{\alpha+2}{2}}\Theta\|_{L^2}
+\|\Theta\|_{W^{1,p}}\Big)(t).$$
According to Proposition \ref{3dlg-P3.1}-Proposition \ref{3dlg-P3.3}, it is easy to get that
\begin{equation}
\mathcal {L}(t)\leqslant C(\inf\limits_{y\in\mathbb R}J_0)^{-1}\bigg(1+C(\inf\limits_{y\in\mathbb R}J_0)\Big[q(M_0(t)-\sup\limits_{y\in\mathbb R}J_0)\Big]^{\frac1 q}\Big(\int_0^tF_0(s)ds\Big)^{\frac{q-1} q}\bigg)
\notag
    \end{equation}
for any $t\in[0,T_0]$ with $T_0 = \min\{\widetilde{t_0},\widetilde{t_1}\},$ where $G_0(t)$ and $F_0(t)$ are mentioned as in
Proposition \ref{3dlg-P3.3} there. Starting from a time $t_0$ with $t_0<T_0$, one can find that the quaternion $(J,\varrho, v, \Theta)|_{t=t_0}$ satisfies the conditions on the initial data stated in Proposition \ref{3dlg-P3.1}. Thus, the solution $(J,\varrho, v, \Theta)$ can be extended forward in
time to another time $t_1 =t _0 +\overline{t_1}$, for some positive time $\overline{t_1}$ depending only on $\overline\varrho$ and the upper bound of $\mathcal {L}(t_0).$ Besides, the extended solution $(J,\varrho, v, \Theta)$ has the same regularities as expressed in Proposition \ref{3dlg-P3.1}-Proposition \ref{3dlg-P3.3} on $[0,t_1]$. Particularly, the time $t_1$ is independent of $\underline\varrho$ by Proposition \ref{3dlg-P3.2} and Proposition \ref{3dlg-P3.3}, but is controlled by certain conditions (see remark \ref{3dlf-R1}).

\begin{remark}\label{3dlf-R1}
Let $p,q\in(1,2),$ and ${\overline J}, {\overline\varrho}, {\underline J}$ and $H_0(t)$ stated in Proposition \ref{3dlg-P3.3}. Let the quaternion $(J,\varrho, v,\Theta)$ be a strong solution to the problem (\ref{1dlg-E16})-(\ref{1dlg-E17}) on ${\mathbb R}\times [0, t_1].$ Then the time $t_1$ is determined by
\begin{eqnarray*}
&&\frac{q+1}{q-1}H_0^{\frac{q+1}{q-1}}(t_1)t_1<1,\label{3dlg-E31}\\
&&C\Bigg(H_0(t)+\int_0^tH_0^{\frac{2q} {q-1}}(s)\bigg(1-\frac{q+1}{q-1}H_0^{\frac{q+1}{q-1}}(s)s\bigg)^{-\frac{2q}{q+1}}ds\Bigg)^{\frac1 q}\Big(\int_0^t{\overline J}{\overline\varrho}^{-\frac{\alpha}{q-1}}ds\Big)^{\frac{q-1}{q}}\leqslant\underline J/4,\nonumber\\\label{3dlg-E31}
    \end{eqnarray*}
where $\alpha<\min\{-\frac{q}{2(q-1)},-\frac{4-p}{2-p}\}.$
\end{remark}

According to Proposition \ref{3dlg-P3.1}-Proposition \ref{3dlg-P3.3} and the argument above, the following corollary is deduced directly. The idea to prove the following corollary is similar to the Corollary 3.1 in \cite{Fang-Zang-2023}.  In order to avoid repetition and ensure the paper being of reasonable length, the proof of Corollary \ref{3dlg-C1} is deleted.

\begin{corollary}\label{3dlg-C1}
Let $p,q\in(1,2)$ and $\alpha<\min\{-\frac{q}{2(q-1)},-\frac{4-p}{2-p}\}.$  Assume that the initial data $(J_0,\varrho_0, v_0, \Theta_0)$ satisfies
\begin{eqnarray}
&&0<\underline\varrho\leqslant\varrho_0\leqslant\overline\varrho\ \mbox{ on }\ {\mathbb R}\mbox{ with }\overline\varrho\geqslant1,\quad (\varrho_0^\alpha)_y\in L^p\cap L^\infty,\label{3dlg-E33}\\
&&0<\underline J\leqslant J_0\leqslant\overline J\ on\ \mathbb R,\quad\varrho_0^{\frac{\alpha-1}{2}}J_0^{\prime}\in L^2,\quad\varrho_0^{{\alpha-1}}J_0^{\prime}\in L^1,\label{3dlg-E34}\\
&&\varrho_0^{\frac{\alpha+1}{2}}v_0\in L^2_{loc},\quad\varrho_0^{\frac{\alpha+1}{2}}v_0^{\prime}\in L^2,\quad\varrho_0^{\frac{\alpha}{q}}\big|\frac{v_0^{\prime}}{J_0}\big|J_0^{\frac1 q}\in L^q,\label{3dlg-E35}\\
&&\varrho_0^{\frac{\alpha+2}{2}}\Theta_0\in L^2,\quad\varrho_0^{\frac{\alpha+1}{2}}\Theta_0^{\prime}\in L^2,\quad\varrho_0^{\alpha+1}\Theta_0\in L^1,\quad\varrho_0^{\frac{\alpha+2}{p}}\big|\frac{\Theta_0^{\prime}}{J_0}\big|J_0^{\frac1 p}\in L^p\label{3dlg-E36}
    \end{eqnarray}
for some positive constants $\underline\varrho, \overline\varrho,\underline J$ and $\overline J$.

Then, there exists a positive time $T_0$ such that the problem (\ref{1dlg-E16})-(\ref{1dlg-E17}) admits a unique strong solution $(J,\varrho, v, \Theta)$ on $(-r,r) \times [0,T_0].$
Moreover, the solution $(J,\varrho, v, \Theta)$ has the following properties that
\begin{eqnarray}
&&{\frac 3 4}\underline J\leqslant J\leqslant{\frac 5 4}\overline J\mbox{ and } 0<{\frac {4\underline J}{5\overline J}}\underline \varrho\leqslant \varrho\leqslant{\frac {4\underline J}{3\overline J}}\overline \varrho\ on\  (-r,r)  \times [0,T_0],\nonumber\\
&&\sup\limits_{t\in[0,T_0]}\Big(\left\|\varrho_0^{\frac{\alpha+1}{2}}v\right\|_{L^2(-r,r)}^2+\left\|\varrho_0^{\alpha+1}\Theta\right\|_{L^1(-r,r)}
+\left\|\varrho_0^{\frac{\alpha+2}{2}}\Theta\right\|_{L^2(-r,r)}^2+\left\|\varrho_0^{\frac{\alpha+1}{2}}v_y\right\|_{L^2(-r,r)}^2\nonumber\\
&&+\left\|\varrho_0^{\frac{\alpha+1}{2}}\Theta_y\right\|_{L^2(-r,r)}^2+\left\|\varrho_0^{\frac{\alpha}{q}}|\frac{v_y}{J}|J^{\frac1q}\right\|_{L^q(-r,r)}^q
+\left\|\varrho_0^{\frac{\alpha+2}{p}}|\frac{\Theta_y}{J}|J^{\frac1p}\right\|_{L^p(-r,r)}^p+\left\|\varrho_0^{\frac{\alpha}{2}}J_y\right\|_{L^2(-r,r)}^2\Big)\nonumber\\
&&+\int_0^{T_0}\bigg(\left\|\varrho_0^{\frac{\alpha}{q}}|\frac{v_y}{J}|J^{\frac1q}\right\|_{L^q(-r,r)}^q
+\left\|\varrho_0^{\frac{\alpha+1}{p}}|\frac{\Theta_y}{J}|J^{\frac1p}\right\|_{L^p(-r,r)}^p
+\left\|\varrho_0^{\frac{\alpha}{2}}|\frac{v_y}{J}|^{\frac{q-2}{2}}\frac{v_{yy}}{\sqrt{J}}\right\|_{L^2(-r,r)}^2\nonumber\\
&&+\left\|\varrho_0^{\frac{\alpha}{2}}|\frac{\Theta_y}{J}|^{\frac{p-2}{2}}\frac{\Theta_{yy}}{\sqrt{J}}\right\|_{L^2(-r,r)}^2
+\left\|\varrho_0^{\frac{\alpha+1}{2}}v_t\right\|_{L^2(-r,r)}^2+\left\|\varrho_0^{\frac{\alpha+3}{2}}\Theta_t\right\|_{L^2(-r,r)}^2\bigg)ds\leqslant C\label{3dlg-Eq37}
\end{eqnarray}
hold for any $0 < r < \infty.$ Here, the time $T_0$ and the constant $C>0$ on the right hand of (\ref{3dlg-Eq37}) depend on  $p, q, \underline J, \overline J, \overline\varrho, \|(\varrho_0^\alpha)_y\|_{{L^2}\cap{L^\infty}},  \|\varrho_0^{\frac{\alpha}{2}}J_0^{\prime}\|_{L^2}, \|\varrho_0^{\frac{\alpha+1}{2}}v_0\|_{L^2_{loc}}, \|\varrho_0^{\frac{\alpha+1}{2}}v_0^{\prime}\|_{L^2}, \Big\|\varrho_0^{\frac{\alpha}{q}}\big|\frac{v_0^{\prime}}{J_0}\big|J_0^{\frac1 q}\Big\|_{L^q},\|\varrho_0^{\frac{\alpha+2}{2}}\Theta_0\|_{L^2},\\
\|\varrho_0^{\frac{\alpha+1}{2}}\Theta_0^{\prime}\|_{L^2}$ and $\Big\|\varrho_0^{\frac{\alpha+2}{p}}\big|\frac{\Theta_0^{\prime}}{J_0}\big|J_0^{\frac1 p}\Big\|_{L^p}$, but $\underline\varrho$ and $r$.
\end{corollary}

\section{Local existence in the presence of far field vacuum}\label{4dlg-S4}

In this section, we intend to prove the local existence of strong solutions to the problem (\ref{1dlg-E16})-(\ref{1dlg-E17}), in the presence of vacuum at far fields. We start with the uniqueness of the solutions. For convince, it is also assumed that $J_0\equiv1$ throughout this section.

\begin{proposition}\label{4dlg-P1}
Let $p,q\in(1,2)$ and $\alpha<\min\{-\frac{q}{2(q-1)},-\frac{4-p}{2-p}\}.$ Let $\varrho_0$ satisfy $\inf\limits_{y\in(-r,r)}\varrho_0(y)>0$ for any $r\in(0,\infty)$ and $\varrho_0\leqslant\overline\varrho$ on $\mathbb R$ for a positive constant $\overline\varrho\geqslant1$. Assume that two quaternions  $(J_1,\varrho_1,v_1,\Theta_1)$ and $(J_2,\varrho_2,v_2,\Theta_2)$ are two solutions to the problem (\ref{1dlg-E16})-(\ref{1dlg-E17}) on $\mathbb R \times [0,T]$ meeting the same initial conditions and satisfying
\begin{eqnarray*}
&& c_0\leqslant J_i\leqslant C_0\mbox{ for some positive constants $c_0$ and $C_0$},\quad \varrho_i\geqslant0\ \mbox{ on}\ {\mathbb R} \times [0,T],  \\
&&J_i-J_0\in C([0,T];L^2),\quad J_{it}\in L^\infty(0,T;L^q)\cap L^\infty(0,T;L^2), \\
&&\varrho_0^{\frac\alpha 2}J_{iy}\in L^\infty(0,T;L^2),\quad\varrho_0^{\frac{\alpha+1} 2}v_i\in C([0,T];L^2),\quad\varrho_0^{\frac{\alpha+2} 2}\Theta_i\in C([0,T];L^2),\\
&&\varrho_0^{\frac{\alpha+1} 2}v_{iy}\in L^\infty(0,T;L^2),\quad\varrho_0^{\frac{\alpha+1} 2}\Theta_{iy}\in L^\infty(0,T;L^2), \\
&&\varrho_0^{\frac{\alpha+1} 2}v_{it}\in L^2(0,T;L^2),\quad\varrho_0^{\frac{\alpha+3} 2}\Theta_{it}\in L^2(0,T;L^2), \\
&&\varrho_0^{\frac\alpha q}|\frac {v_{iy}} {J_i}|J_i^{\frac 1 q}\in L^\infty(0,T;L^q)\cap L^q(0,T;L^q),\quad\varrho_0^{\frac\alpha p}|\frac {\Theta_{iy}} {J_i}|J_i^{\frac 1 p}\in L^\infty(0,T;L^p)\cap L^p(0,T;L^p),\\
&&\varrho_0^{\frac\alpha 2}|\frac {v_{iy}} J_i|^{\frac{q-2}2}\frac {v_{iyy}}{\sqrt {J_i}}\in L^2(0,T;L^2),\quad\varrho_0^{\frac\alpha 2}|\frac {\Theta_{iy}}{ J_i}|^{\frac{p-2}2}\frac {\Theta_{iyy}}{\sqrt {J_i}}\in L^2(0,T;L^2);
    \end{eqnarray*}
for $i = 1, 2.$ Then $(J_1,\varrho_1,v_1,\Theta_1)\equiv(J_2,\varrho_2,v_2,\Theta_2).$
\end{proposition}

The proof of Proposition \ref{4dlg-P1} is standard and similar to the proof of Proposition 4.1 in \cite{Fang-Zang-2023}. Here, we deleted it in order to shorten the length of the paper. Now, we focus on the proof of Theorem \ref{1dlg-thm1}.

\textbf{Proof of Theorem \ref{1dlg-thm1}.} The uniqueness is a direct corollary of Proposition \ref{4dlg-P1}, and it remains to prove the existence. For any $\varepsilon\in(0,1)$, set $\varrho_{0\varepsilon}(y)=\varrho_{0}(y)+\varepsilon$ for $y\in\mathbb R.$ Clearly,  $\varepsilon\leqslant\varrho_{0\varepsilon}(y)\leqslant\overline\varrho+1$ for all $y\in{\mathbb R}.$ Consider the
following approximate system of (\ref{1dlg-E16})-(\ref{1dlg-E15})
\begin{equation}\label{4dlg-E13}
    \left\{
    \begin{array}{ll}
J_t=v_y, \\
J\varrho=J_0\varrho_{0\varepsilon},\\
\varrho_{0\varepsilon}v_t-\frac{1}{J_0}\left(|\frac{v_y}J|^{q-2}\frac{v_y}J\right)_y+R\frac{1}{J_0}(\varrho\Theta)_y=0, \\
\varrho_{0\varepsilon}\Theta_t-\frac{1}{J_0}\left(|\frac{\Theta_y}J|^{p-2}\frac{\Theta_y}J\right)_y+R\frac{1}{J_0}\varrho\Theta v_y=\frac{J}{J_0}\left|\frac{v_y}{J}\right|^q.
\end{array}
    \right.
    \end{equation}
According to Corollary \ref{3dlg-C1}, the system (\ref{4dlg-E13}) admits a strong solution $(J_\varepsilon,\varrho_\varepsilon,v_\varepsilon,\Theta_\varepsilon)$ on $(-r,r)\times[0,T_0]$ for some $T_0>0,$ where $T_0$ depends on $p,$ $q,$ $\underline J,$ $\overline J$ and certain norms of initial data, 
but independent of $\varepsilon$ and $R.$ Moreover, the solution $(J_\varepsilon,\varrho_\varepsilon,v_\varepsilon,\Theta_\varepsilon)$ has the following properties
\begin{eqnarray}
&&{\frac 3 4}\underline J\leqslant J_\varepsilon\leqslant{\frac 5 4}\overline J\mbox{ and }0<{\frac {4\underline J}{5\overline J}}\underline \varrho\leqslant \varrho_\varepsilon\leqslant{\frac {4\underline J}{3\overline J}}\overline J\mbox{ on }  (-r,r) \times [0,T_0],\label{4dlg-E14}\\
&&\sup\limits_{t\in[0,T_0]}\Big(\left\|\varrho_{0\varepsilon}^{\frac{\alpha+1}{2}}v\right\|_{L^2(-r,r)}^2+\left\|\varrho_{0\varepsilon}^{\alpha+1}\Theta\right\|_{L^1(-r,r)}
+\left\|\varrho_{0\varepsilon}^{\frac{\alpha+2}{2}}\Theta\right\|_{L^2(-r,r)}^2+\left\|\varrho_{0\varepsilon}^{\frac{\alpha+1}{2}}v_y\right\|_{L^2(-r,r)}^2\nonumber\\
&&+\left\|\varrho_{0\varepsilon}^{\frac{\alpha+1}{2}}\Theta_y\right\|_{L^2(-r,r)}^2
+\left\|\varrho_{0\varepsilon}^{\frac{\alpha}{q}}\left|\frac{v_y}{J}\right|J^{\frac1q}\right\|_{L^q(-r,r)}^q
+\left\|\varrho_{0\varepsilon}^{\frac{\alpha+2}{p}}\left|\frac{\Theta_y}{J}\right|J^{\frac1p}\right\|_{L^p(-r,r)}^p
+\left\|\varrho_{0\varepsilon}^{\frac{\alpha}{2}}J_y\right\|_{L^2(-r,r)}^2\Big)\nonumber\\
&&+\int_0^{T_0}\bigg(\left\|\varrho_{0\varepsilon}^{\frac{\alpha}{q}}\left|\frac{v_y}{J}\right|J^{\frac1q}\right\|_{L^q(-r,r)}^q
+\left\|\varrho_{0\varepsilon}^{\frac{\alpha+1}{p}}\left|\frac{\Theta_y}{J}\right|J^{\frac1p}\right\|_{L^p(-r,r)}^p
+\left\|\varrho_{0\varepsilon}^{\frac{\alpha}{2}}|\frac{v_y}{J}|^{\frac{q-2}{2}}\frac{v_{yy}}{\sqrt{J}}\right\|_{L^2(-r,r)}^2\nonumber\\
&&+\left\|\varrho_{0\varepsilon}^{\frac{\alpha}{2}}|\frac{\Theta_y}{J}|^{\frac{p-2}{2}}\frac{\Theta_{yy}}{\sqrt{J}}\right\|_{L^2(-r,r)}^2
+\left\|\varrho_{0\varepsilon}^{\frac{\alpha+1}{2}}v_t\right\|_{L^2(-r,r)}^2+\left\|\varrho_{0\varepsilon}^{\frac{\alpha+3}{2}}\Theta_t\right\|_{L^2(-r,r)}^2\bigg)ds
\leqslant C,\label{4dlg-E15}
    \end{eqnarray}
for any $0 < r < \infty$, where
\begin{equation}
C=\widetilde H(T_0)+\int_0^{T_0}\widetilde H^{\frac{2q} {q-1}}(s)\bigg(1-\frac{q+1}{q-1}\widetilde H^{\frac{q+1}{q-1}}(s)s\bigg)^{-\frac{2q}{q+1}}ds,
\notag
\end{equation}
with
\begin{equation*}
\begin{aligned}
\widetilde H(t)&=\|\varrho_0^{\frac {\alpha+2} 2}v_0\|_{L^2}^2+\|\varrho_0^{\alpha+1}\Theta_0\|_{L^1}+\|\varrho_0^{\frac {\alpha+2} 2}\Theta_0\|_{L^2}^2+\|\varrho_0^{\frac {\alpha+1} 2} v_0^{\prime}\|_{L^2}^2+\|\varrho_0^{\frac {\alpha+1} 2}\Theta_0^{\prime}\|_{L^2}^2\\
&+\|\varrho_0^{\frac \alpha q}|\frac { v_0^{\prime}} {J_0}|{J_0}^{\frac 1 q}\|_{L^q}^q+\|\varrho_0^{\frac {\alpha+2} p}|\frac {\Theta_0^{\prime}} {J_0}|{J_0}^{\frac 1 p}\|_{L^p}^p+\|\varrho_0^{\frac \alpha 2} J_0^{\prime}\|_{L^2}^2+\widetilde{G_0}t
\end{aligned}
 \end{equation*}
and
\begin{equation*}
\widetilde{G_0}=C{\overline J}^{12}\Big(\frac1{\underline J}\Big)^{\frac{12q}{2-q}}({\overline \varrho}+1)^{-\min\{\frac{(3q-2)\alpha+q}{2-q},\frac{8\alpha}{q-1}\}}
\left(\|(\varrho_0^\alpha)_y\|_{L^p}\|(\varrho_0^\alpha)_y\|_{L^\infty}\right)^{\frac{4q}{2-q}}+C({\overline \varrho}+1)^{-\frac{2\alpha+1}{2(q+1)}}.
 \end{equation*}
Using the Banach-Alaoglu theorem and the Cantors diagonal argument, we deduce
from the estimates (\ref{4dlg-E14})-(\ref{4dlg-E15}) that there is a subsequence $(J_\varepsilon,\varrho_\varepsilon,v_\varepsilon,\Theta_\varepsilon),$ still denoted
$(J_\varepsilon,\varrho_\varepsilon,v_\varepsilon,\Theta_\varepsilon),$ and $(J,\varrho,v,\Theta)$ such that
\begin{eqnarray*}
&J_\varepsilon-J_0\to J-J_0&weak\mbox-^*\ in\ L^{\infty}(0,T_0;W^{1,q}(-r,r))\cap L^{\infty}(0,T_0;W^{1,2}(-r,r)),\label{4dlg-E16}\\
&\partial_tJ_\varepsilon\to J_t& weak\mbox-^*\ in\ L^{\infty}(0,T_0;L^{q}(-r,r))\cap L^{\infty}(0,T_0;L^{2}(-r,r)),\label{4dlg-E17}\\
&\partial_tJ_\varepsilon\to J_t& weakly\ in\ L^{q}(0,T_0;W^{1,q}(-r,r)),\label{4dlg-E18}\\
&\partial_yv_\varepsilon\to v_y& weak\mbox-^*\ in\ L^{\infty}(0,T_0;L^{q}(-r,r))\cap L^{\infty}(0,T_0;L^{2}(-r,r)),\label{4dlg-E19}\\
&\partial_yv_\varepsilon\to v_y& weakly\ in\ L^{q}(0,T_0;W^{1,q}(-r,r)),\label{4dlg-E20}\\
&\partial_y\Theta_\varepsilon\to \Theta_y&weak\mbox-^*\ in\ L^{\infty}(0,T_0;L^{p}(-r,r))\cap L^{\infty}(0,T_0;L^{2}(-r,r)),\label{4dlg-E21}\\
&\partial_y\Theta_\varepsilon\to \Theta_y&weakly\ in\ L^{p}(0,T_0;W^{1,p}(-r,r)),\label{4dlg-E22}\\
&\partial_tv_\varepsilon\to v_t&weakly\ in\ L^{2}(0,T_0;L^{2}(-r,r)),\label{4dlg-E23}\\
&\partial_t\Theta_\varepsilon\to \Theta_t& weakly\ in\ L^{2}(0,T_0;L^{2}(-r,r))\label{4dlg-E24}
\end{eqnarray*}
and
\begin{eqnarray*}
&J_\varepsilon-J_0\to J-J_0&strongly\ in\ C([0,T_0]; L^{\infty}(-r,r)),\label{4dlg-E25}\\
&v_\varepsilon\to v&strongly\ in\ C([0,T_0]; L^{q}(-r,r))\cap L^{q}(0,T_0;W^{1,q}(-r,r)),\label{4dlg-E26}\\
&\Theta_\varepsilon\to \Theta& strongly\ in\ C([0,T_0]; L^{p}(-r,r))\cap L^{p}(0,T_0;W^{1,p}(-r,r))\label{4dlg-E27}
\end{eqnarray*}
for any $r\in(0,\infty)$. Thus, 
${\frac 3 4}\underline J\leqslant J\leqslant{\frac 5 4}\overline J$ and $\varrho\geqslant0$ on $(-r,r) \times [0,T_0].$ Since
\begin{equation}
(J_\varepsilon,\varrho_\varepsilon,v_\varepsilon,\Theta_\varepsilon)|_{t=0}=(J_0,\varrho_{0\varepsilon},v_0,\Theta_0)
\notag
\end{equation}
and $\varrho_{0\varepsilon}\to\varrho_0$ in $L^\infty(-r,r)$ for any $r\in(0,\infty)$, one obtains that
\begin{equation*}
(J,\varrho,v,\Theta)|_{t=0}=(J_0,\varrho_0,v_0,\Theta_0).
\end{equation*}
Due to the fact that ${\frac 3 4}\underline J\leqslant J_\varepsilon,J\leqslant{\frac 5 4}\overline J$ on $(-r,r) \times [0,T_0]$, one gets 
that $\frac{1}{J_\varepsilon}\to\frac 1 J$ strongly in $C([0,T_0]; L^{\infty}(-r,r))$ for any $R>0$. 
Moreover, all the nonlinear terms in (\ref{4dlg-E13}), that is, $\frac{\partial_yv_\varepsilon}{J_\varepsilon}$, $\frac{\partial_y\Theta_\varepsilon}{J_\varepsilon}$, $|\frac{\partial_yv_\varepsilon}{J_\varepsilon}|^{q-2}\frac{\partial_yv_\varepsilon}{J_\varepsilon}$, $|\frac{\partial_y\Theta_\varepsilon}{J_\varepsilon}|^{p-2}\frac{\partial_y\Theta_\varepsilon}{J_\varepsilon}$ and $\frac{(\partial_yv_\varepsilon)^2}{J_\varepsilon}$, converge strongly in $L^{q}((-r,r)\times(0,T_0))$, $L^{p}((-r,r)\times(0,T_0))$, $L^{\frac{q}{q-1}}((-r,r)\times(0,T_0))$, $L^{\frac{p}{p-1}}((-r,r)\times(0,T_0))$ and $L^{\frac{q}{2}}((-r,r)\times(0,T_0))$ to $\frac{\partial_yv}{J}$, $\frac{\partial_y\Theta}{J}$, $|\frac{\partial_yv}{J}|^{q-2}\frac{\partial_yv}{J}$, $|\frac{\partial_y\Theta}{J}|^{p-2}\frac{\partial_y\Theta}{J}$ and $\frac{(\partial_yv)^2}{J}$, respectively, for any fixed $r>0$. So, the quaternion $(J,\varrho,v,\Theta)$ satisfies the system (\ref{1dlg-E16})-(\ref{1dlg-E17}) in the sense of distribution. Furthermore, the quaternion $(J,\varrho,v,\Theta)$ satisfies the system (\ref{1dlg-E16})-(\ref{1dlg-E17}) a.e. on $(-r,r)\times [0,T_0]$. Thus, the quaternion $(J,\varrho,v,\Theta)$ is a solution to the problem (\ref{1dlg-E16})-(\ref{1dlg-E17}).

Now, we prove that the quaternion $(J,\varrho,v,\Theta)$ possesses the required regularities in Definition \ref{1dlg-D1}. Besides the regularities obtained by convergence, the other desired regularities of $(J,\varrho,v,\Theta)$ can be verified as follows. First, 
$(J-J_0,v,\Theta)\in C([0,T_0];L^{2}(-r,r))$ for any fixed $r\in(0,\infty).$ Next, according to the convergence $\sqrt\varrho_{0\varepsilon}\to\sqrt{\varrho_0}$ and $\frac{1}{\sqrt\varrho_{0\varepsilon}}\to\frac{1}{\sqrt{\varrho_0}}$ in $L^\infty(-r,r)$ for any fixed $r\in(0,\infty)$, we derive 
that
\begin{eqnarray*}
&\partial_yJ_\varepsilon\to J_y& weak\mbox-^*\ in\ L^{\infty}(0,T_0;L^{2}(-r,r)),\\
&\varrho_{0\varepsilon}^{\frac{\alpha}{2}}\partial_yJ_\varepsilon\to \varrho_{0}^{\frac{\alpha}{2}}J_y& weak\mbox-^*\ in\ L^{\infty}(0,T_0;L^{2}(-r,r)),\\
&\varrho_{0\varepsilon}^{\frac{\alpha+1}{2}}\partial_yv_\varepsilon\to \varrho_{0}^{\frac{\alpha+1}{2}}v_y& weak\mbox-^*\ in\ L^{\infty}(0,T_0;L^{2}(-r,r)),\\
&\varrho_{0\varepsilon}^{\frac{\alpha+1}{2}}\partial_y\Theta_\varepsilon\to \varrho_{0}^{\frac{\alpha+1}{2}}\Theta_y&weak\mbox-^*\ in\ L^{\infty}(0,T_0;L^{2}(-r,r)),\\
&\varrho_{0\varepsilon}^{\frac{\alpha+1}{2}}\partial_tv_\varepsilon\to \varrho_{0}^{\frac{\alpha+1}{2}}v_t& weakly\ in\ L^{2}(0,T_0;L^{2}(-r,r)),\\
&\varrho_{0\varepsilon}^{\frac{\alpha+3}{2}}\partial_t\Theta_\varepsilon\to \varrho_{0}^{\frac{\alpha+3}{2}}\Theta_t& weakly\ in\ L^{2}(0,T_0;L^{2}(-r,r)),\\
&\varrho_{0\varepsilon}^{\frac{\alpha}{q}}\Big|\frac{\partial_yv_\varepsilon}{J_\varepsilon}\Big|J_\varepsilon^{\frac1 q}\to \varrho_{0}^{\frac{\alpha}{q}}\Big|\frac{v_y}{J}\Big|J^{\frac1 q}& weak\mbox-^*\ in\ L^{\infty}(0,T_0;L^{q}(-r,r)),\\
&\varrho_{0\varepsilon}^{\frac{\alpha+2}{p}}\Big|\frac{\partial_y\Theta_\varepsilon}{J_\varepsilon}\Big|J_\varepsilon^{\frac1 p}\to \varrho_{0}^{\frac{\alpha+2}{p}}\Big|\frac{\Theta_y}{J}\Big|J^{\frac1 p}& weak\mbox-^*\ in\ L^{\infty}(0,T_0;L^{p}(-r,r)),\\
&\varrho_{0\varepsilon}^{\frac{\alpha}{2}}\Big|\frac{\partial_y v_\varepsilon}{J_\varepsilon}\Big|^{\frac{q-2}{2}}\frac{\partial_{yy}v_\varepsilon}{\sqrt{J_\varepsilon}}\to \varrho_{0}^{\frac{\alpha}{2}}\Big|\frac{v_y}{J}\Big|^{\frac{q-2}{2}}\frac{v_{yy}}{\sqrt J}& weakly\ in\ L^{2}(0,T_0;L^{2}(-r,r)),\\
&\varrho_{0\varepsilon}^{\frac{\alpha}{2}}\Big|\frac{\partial_y\Theta_\varepsilon}{J_\varepsilon}\Big|^{\frac{p-2}{2}}\frac{\partial_{yy}\Theta_\varepsilon}{\sqrt{J_\varepsilon}}\to \varrho_{0}^{\frac{\alpha}{2}}\Big|\frac{\Theta_y}{J}\Big|^{\frac{p-2}{2}}\frac{\Theta_{yy}}{\sqrt J}& weakly\ in\ L^{2}(0,T_0;L^{2}(-r,r))
   \end{eqnarray*}
for any $r\in(0,\infty)$.  Consequently, 
\begin{eqnarray*}
&\|\varrho_{0}^{\frac{\alpha}{2}}J_y\|_{L^{\infty}(0,T_0;L^{2}(-r,r))}\leqslant C,&\|\varrho_{0}^{\frac{\alpha+1}{2}}v_y\|_{L^{\infty}(0,T_0;L^{2}(-r,r))}\leqslant C,\\
&\|\varrho_{0}^{\frac{\alpha+1}{2}}\Theta_y\|_{L^{\infty}(0,T_0;L^{2}(-r,r))}\leqslant C,&\|\varrho_{0}^{\frac{\alpha+1}{2}}v_t\|_{L^{2}(0,T_0;L^{2}(-r,r))}\leqslant C,\\
&\|\varrho_{0}^{\frac{\alpha+3}{2}}\Theta_t\|_{L^{2}(0,T_0;L^{2}(-r,r))}\leqslant C,&\Big\|\varrho_{0}^{\frac{\alpha}{q}}\Big|\frac{v_y}{J}\Big|J^{\frac1 q}\Big\|_{L^{\infty}(0,T_0;L^{q}(-r,r))}\leqslant C,\\
&\Big\|\varrho_{0}^{\frac{\alpha+2}{p}}\Big|\frac{\Theta_y}{J}\Big|J^{\frac1 p}\Big\|_{L^{\infty}(0,T_0;L^{p}(-r,r))}\leqslant C,&\Big\|\varrho_{0}^{\frac{\alpha}{2}}\Big|\frac{v_y}{J}\Big|^{\frac{q-2}{2}}\frac{v_{yy}}{\sqrt J}\Big\|_{L^{2}(0,T_0;L^{2}(-r,r))}\leqslant C,\\
&\Big\|\varrho_{0}^{\frac{\alpha}{2}}\Big|\frac{\Theta_y}{J}\Big|^{\frac{p-2}{2}}\frac{\Theta_{yy}}{\sqrt J}\Big\|_{L^{2}(0,T_0;L^{2}(-r,r))}\leqslant C
    \end{eqnarray*}
for a positive constant $C$ independent of $r.$ Thus,
\begin{eqnarray*}
&\varrho_{0}^{\frac{\alpha}{2}}J_y\in L^{\infty}(0,T_0;L^{2}),\quad\varrho_{0}^{\frac{\alpha+1}{2}}v_y\in L^{\infty}(0,T_0;L^{2}),\quad\varrho_{0}^{\frac{\alpha+1}{2}}\Theta_y\in L^{\infty}(0,T_0;L^{2}),\\
&\varrho_{0}^{\frac{\alpha+1}{2}}v_t\in L^{2}(0,T_0;L^{2}),
\quad\varrho_{0}^{\frac{\alpha+3}{2}}\Theta_t\in L^{2}(0,T_0;L^{2}),\\
&\varrho_{0}^{\frac{\alpha}{q}}\Big|\frac{v_y}{J}\Big|J^{\frac1 q}\in L^{\infty}(0,T_0;L^{q}),\quad\varrho_{0}^{\frac{\alpha+2}{p}}\Big|\frac{\Theta_y}{J}\Big|J^{\frac1 p}\in L^{\infty}(0,T_0;L^{p}),\\
&\varrho_{0}^{\frac{\alpha}{2}}\Big|\frac{v_y}{J}\Big|^{\frac{q-2}{2}}\frac{v_{yy}}{\sqrt J}\in L^{2}(0,T_0;L^{2}),\quad\varrho_{0}^{\frac{\alpha}{2}}\Big|\frac{\Theta_y}{J}\Big|^{\frac{p-2}{2}}\frac{\Theta_{yy}}{\sqrt J}\in L^{2}(0,T_0;L^{2}).
\end{eqnarray*}
Furthermore,
\begin{equation}
\varrho_{0}^{\frac{\alpha+1}{2}}v\in C([0,T_0];L^{2}),\quad\varrho_{0}^{\frac{\alpha+2}{2}}\Theta\in C([0,T_0];L^{2}).
\notag
    \end{equation}
Based on the above discussion, the quaternion $(J,\varrho,v,\Theta)$ meets all regularities mentioned in Definition \ref{1dlg-thm1}. The proof is completed.

\section{ Global existence in the presence of far field vacuum}\label{5dlg-S5}

In this section, we concentrate on proving global existence to the problem (\ref{1dlg-E16})-(\ref{1dlg-E17}) in the presence of vacuum at far field. We start with the condition [H1]-[H5] on initial data $(J_0,\varrho_0,v_0,\Theta_0),$ that is,
\begin{eqnarray}
&&0\leqslant\varrho_0\leqslant\overline\varrho\ on\ {\mathbb R}\ with \ \overline{\varrho}\geqslant1,\label{5dlg-E1}\\
&&0<\underline J\leqslant J_0\leqslant\overline J\ on\ \mathbb R,\varrho_0^{\frac{\alpha}{ 2}}J_0^{\prime}\in L^2,\quad \varrho_0^{\alpha}J_0^{\prime}\in L^1,\label{5dlg-E2}\\
&&\varrho_0^{\frac{\alpha+1}{2}}v_0\in L^2,\quad\varrho_0^{\frac{\alpha+1}{2}}v_0^{\prime}\in L^2,\quad\varrho_0^{\frac{\alpha}{q}}\big|\frac{v_0^{\prime}}{J_0}\big|\in L^q,\label{5dlg-E3}\\
&&\varrho_0^{\frac{\alpha+2}{2}}\Theta_0\in L^2,\quad\varrho_0^{\frac{\alpha+1}{2}}\Theta_0^{\prime}\in L^2,\quad\varrho_0^{\alpha+1}\Theta_0\in L^1,\quad\varrho_0^{\frac{\alpha+2}{p}}\big|\frac{\Theta_0^{\prime}}{J_0}\big|\in L^p,\label{5dlg-E4}\\
&&(\varrho_0^\alpha)_y\in L^p\cap L^\infty ,\quad\varrho_0(y)\geqslant \frac {A_0} {(1+|y|)^{l}}\quad (\forall y \in \mathbb R),\label{5dlg-E5}
    \end{eqnarray}
where $\alpha<\min\{-\frac{q}{2(q-1)},-\frac{4-p}{2-p}\}$ and $l\in(0,\min\{1,-\frac{3p-2}{(2-p)\alpha+3},-\frac{p-1}{\alpha p}\}).$
Although the initial $J_0$ is not the identical one, the initial $J_0$ is bounded from blow. Taking similar calculation in Proposition \ref{3dlg-P3.2} and Proposition \ref{3dlg-P3.3}, one drive the following proposition.

\begin{proposition}\label{5dlg-P5.1}
Let $p,q\in(1,2)$ and $\alpha<\min\{-\frac{q}{2(q-1)},-\frac{4-p}{2-p}\}.$ For a given positive time $T,$ let the quaternion
$(J,\varrho,v,\Theta)$ be a strong solution to the problem (\ref{1dlg-E16})-(\ref{1dlg-E17}) on $\mathbb R \times [0,T]$ with $(J_0,\varrho_0,v_0,\Theta_0)$ satisfying (\ref{5dlg-E1})-(\ref{5dlg-E5}). Then
\begin{eqnarray}
&&\int\varrho_0^{\alpha}\Big({\varrho_0}v^2+{\varrho_0}\Theta+\varrho_0^2\Theta^2+{\varrho_0}v_y^2
+{\varrho_0}\Theta_y^2+\frac{1}{J_0}|\frac{v_y}{J}|^qJ+\frac{\varrho_0^2}{J_0}|\frac{\Theta_y} J|^pJ+ J_y^2\Big)dy\nonumber\\
&&\leqslant H_0(t)\bigg(1-\frac{q+1}{q-1}H_0^{\frac{q+1}{q-1}}(t)t\bigg)^{-\frac{q-1}{q+1}},\label{5dlg-E6}\\
&&\sup\limits_{y\in\mathbb R}J(y,t)\leqslant2M_0(t)\exp\bigg((\sup\limits_{y\in\mathbb R}\varrho_0)^{-\frac{\alpha}{q}}\frac{(q-1)t}{q}\bigg),\label{5dlg-E7}\\
&&\inf\limits_{y\in\mathbb R}J(y,t)\geqslant(\inf\limits_{y\in\mathbb R}J_0)\bigg(1+C(\inf\limits_{y\in\mathbb R}J_0)\Big[q(M_0(t)-\sup\limits_{y\in\mathbb R}J_0)\Big]^{\frac1 q}\Big(\int_0^tF_0(s)ds\Big)^{\frac{q-1} q}\bigg)^{-1},\label{5dlg-E8}\\
&&\varrho(y,t)\leqslant\varrho_0(\sup\limits_{y\in\mathbb R}J_0)(\inf\limits_{y\in\mathbb R}J_0)^{-1}(1+C(\inf\limits_{y\in\mathbb R}J_0)[q(M_0(t)-\sup\limits_{y\in\mathbb R}J_0)]^{\frac1 q}(\int_0^tF_0(s)ds)^{\frac{q-1} q}),\label{5dlg-E9}\\
&&\varrho(y,t)\geqslant\frac{1}{2}\varrho_0(\inf\limits_{y\in\mathbb R}J_0)M^{-1}_0(t)\exp\bigg(-(\sup\limits_{y\in\mathbb R}\varrho_0)^{-\frac{\alpha}{q}}\frac{(q-1)t}{q}\bigg)\label{5dlg-E10}
    \end{eqnarray}
for all $t\in[0,T]$, where
\begin{eqnarray*}
&&H_0=\Big(\|\varrho_0^{\frac {\alpha+2} 2}v_0\|_{L^2}^2+\|\varrho_0^{\alpha+1}\Theta_0\|_{L^1}+\|\varrho_0^{\frac {\alpha+2} 2}\Theta_0\|_{L^2}^2+\|\varrho_0^{\frac {\alpha+1} 2} v_0^{\prime}\|_{L^2}^2+\|\varrho_0^{\frac {\alpha+1} 2}\Theta_0^{\prime}\|_{L^2}^2\\
&&~~~~+\|\varrho_0^{\frac \alpha q}|\frac { v_0^{\prime}} {J_0}|{J_0}^{\frac 1 q}\|_{L^q}^q+\|\varrho_0^{\frac {\alpha+2} p}|\frac {\Theta_0^{\prime}} {J_0}|{J_0}^{\frac 1 p}\|_{L^p}^p+\|\varrho_0^{\frac \alpha 2} J_0^{\prime}\|_{L^2}^2\Big)+G_0t\\
&&G_0=C{\overline J}^{12}\Big(\frac1{\underline J}\Big)^{\frac{12q}{2-q}}{\overline \varrho}^{-\min\{\frac{(3q-2)\alpha+q}{2-q},\frac{10\alpha}{q-1}\}}
\left(\|(\varrho_0^\alpha)_y\|_{L^p}\|(\varrho_0^\alpha)_y\|_{L^\infty}\right)^{\frac{16q}{2-q}}+C\overline{\varrho}^{-\frac{2\alpha+1}{2(q+1)}},\\
&&M_0(t)=\frac{C}{q}\bigg(q\sup\limits_{y\in\mathbb R}J_0+H_0(t)+\int_0^tH_0^{\frac{2q}{q-1}}(s)\Big(1-\frac{q+1}{q-1}H_0^{\frac{q+1}{q-1}}(s)s\Big)^{-\frac{2q}{q-1}}ds\bigg),\\
&&F_0(t)=({\sup\limits_{y\in\mathbb R}\varrho_0})^{-\frac{\alpha}{q-1}}M_0^{-\frac{q+1}{q-1}}(t)\exp\bigg(-\frac{q+1}{q}({\sup\limits_{y\in\mathbb R}\varrho_0})^{-\frac{\alpha}{q}}t\bigg)
 \end{eqnarray*}
on $[0,T].$
\end{proposition}

Next, the energy inequality of strong solution to the problem (\ref{1dlg-E16})-(\ref{1dlg-E17}) is derived in the following proposition. The proof of Proposition \ref{5dlg-P5.2} focuses on the initial date $J_0$ but not on the lower bound of $J_0.$ This requires more detailed calculations, compared to the proof of Proposition 5.2 in \cite{Fang-Zang-2023}.

\begin{proposition}\label{5dlg-P5.2}
Let $p,q\in(1,2)$ and $\alpha<\min\{-\frac{q}{2(q-1)},-\frac{4-p}{2-p}\}.$ For a given positive time $T,$ let the quaternion
$(J,\varrho,v,\Theta)$ be a strong solution to the problem (\ref{1dlg-E16})-(\ref{1dlg-E17}) on $\mathbb R \times [0,T]$ with $(J_0,\varrho_0,v_0,\Theta_0)$ satisfying (\ref{5dlg-E1})-(\ref{5dlg-E5}) and set $\delta(t)=\inf\limits_{y\in\mathbb R}J(y,t)$ on $[0,T].$ Then
\begin{eqnarray}\label{5dlg-E11}
&&\frac d {dt}\bigg(\int\varrho_0^{\alpha}\Big({\varrho_0}v^2+{\varrho_0}\Theta+\varrho_0^2\Theta^2+{\varrho_0}v_y^2
+{\varrho_0}\Theta_y^2+\frac{1}{J_0}|\frac{v_y}{J}|^qJ+\frac{\varrho_0^2}{J_0}|\frac{\Theta_y} J|^pJ+ J_y^2\Big)dy\bigg)(t)\nonumber\\
&&+\int\varrho_0^\alpha\Big(\frac{1}{J_0}|\frac{v_y}{J}|^qJ+\frac{\varrho_0}{J_0}|\frac{\Theta_y} J|^pJ+\frac{1}{J_0}|\frac{v_y}{J}|^{q-2}\frac{ v_{yy}^2}J+\frac{1}{J_0}|\frac{\Theta_y} J|^{p-2}\frac {\Theta_{yy}^2}J+\varrho_0^3\Theta_t^2+{\varrho_0}v_t^2\Big)dy\nonumber\\
&&\leqslant\bigg(\int\varrho_0^{\alpha}\Big({\varrho_0}v^2+{\varrho_0}\Theta+\varrho_0^2\Theta^2+{\varrho_0}v_y^2
+{\varrho_0}\Theta_y^2+\frac{1}{J_0}|\frac{v_y}{J}|^qJ+\frac{\varrho_0^2}{J_0}|\frac{\Theta_y} J|^pJ+ J_y^2\Big)dy\bigg)^{\frac{2q}{q-1}}\nonumber\\
&&+C\left(\delta(t)\right)^{-\frac{24}{2-q}}({\sup\limits_{y\in\mathbb R}\varrho_0})^{-\min\{\frac{(3q-2)\alpha+q}{2-q},\frac{10\alpha}{q-1}\}}
\left(\|(\varrho_0^\alpha)_y\|_{L^p}\|(\varrho_0^\alpha)_y\|_{L^\infty}\right)^{\frac{16q}{2-q}}
\end{eqnarray}
holds for any $t\in[0,T]$, where the positive constant $C$ just depends on $p$, $q$ and $A_0$.
\end{proposition}

\begin{proof}
Choose a function $\xi\in C_c^{\infty}$((-2,2)) such that
$$\xi\equiv1\mbox{ on }(-1,1)\mbox{ and }0\leqslant\xi\leqslant1\mbox{ on }(-2,2).$$
For each $r\geqslant1$, we set $\xi_r(y)=\xi(\frac{y}{r})$ for $y\in\mathbb R$.

To derive $(\ref{5dlg-E11})$, we need to complete the following several steps, using similar estimate in Proposition \ref{3dlg-P3.2}.

{\sl\textit{Step 1.}} Multiplying equation (\ref{1dlg-E14}) and (\ref{1dlg-E15}) by $\varrho_0^\alpha v\xi_r^3$ and $\frac12\varrho_0^\alpha\xi_r^3$ respectively, adding the two equations and integrating the resultant over $\mathbb R$, one obtains that
\begin{eqnarray}\label{5dlg-E12}
&&\frac1 2\frac d {dt}\int\Big(\varrho_0^{\alpha+1}v^2+\varrho_0^{\alpha+1}\Theta\Big)\xi_r^3dy+\int\frac{\varrho_0^\alpha}{J_0}|\frac{v_y}{J}|^qJ\xi_r^3dy\nonumber\\
&&=-\int|\frac{v_y}{J}|^{q-2}\frac {v_y}J(\frac{\varrho_0^\alpha}{J_0})_yv\xi_r^3dy-3\int\frac{\varrho_0^\alpha}{J_0}|\frac{v_y}{J}|^{q-2}\frac {v_y}Jv\xi_r^2\xi_r^{\prime}dy+\int R\varrho\Theta(\frac{\varrho_0^\alpha}{J_0})_yv\xi_r^3dy\nonumber\\
&&+3\int R\varrho\Theta\frac{\varrho_0^\alpha}{J_0} v\xi_r^2\xi_r^{\prime}dy-\frac12\int|\frac{\Theta_y} J|^{p-2}\frac {\Theta_y}J(\frac{\varrho_0^\alpha}{J_0})_y\xi_r^3dy-\frac{3}{2}\int\frac{\varrho_0^\alpha}{J_0}|\frac{\Theta_y} J|^{p-2}\frac {\Theta_y}J\xi_r^2\xi_r^{\prime}dy
\nonumber\\
&&+\frac12\int R\frac{\varrho_0^\alpha}{J_0}\varrho\Theta v_y\xi_r^3dy:=\sum_{i=1}^7 \Pi_i.
\end{eqnarray}
The term $\Pi_1$-$\Pi_7$ on the right hand of (\ref{5dlg-E12}) is estimated as follows
\begin{eqnarray*}
&|\Pi_1|&=\Big|-\int|\frac{v_y}{J}|^{q-2}\frac {v_y}J(\frac{\varrho_0^\alpha}{J_0})_yv\xi_r^3dy\Big|\\
&&\leqslant \Big(\int\varrho_0^{\alpha+1}v^2dy+\int\varrho_0^{\alpha+1}v_y^2dy\Big)^{\frac{2q}{q-1}}
+\epsilon\int\frac{\varrho_0^\alpha}{J_0}|\frac{v_y}{J}|^qJ\xi_r^3dy\\
&&+C(\epsilon)\left(\|(\varrho_0^\alpha)_y\|_{L^p}\|(\varrho_0^\alpha)_y\|_{L^\infty}(\frac{1}{\underline{J}})^{\frac{1}{q}} (\sup\limits_{y\in\mathbb R}\varrho_0)^{-\frac{(3q-2)\alpha+q}{2q}}\right)^\frac{4q}{5-q}\\
&& +C(\epsilon)(\|\varrho_0^{\frac{\alpha}{2}}J_{0y}\|_{L^2}
\|(\varrho_0^\alpha)_y\|_{L^\infty}\sup\limits_{y\in\mathbb R}\varrho_0)^{\frac{4q}{q+1}}\\
&&\leqslant \Big(\int\varrho_0^{\alpha+1}v^2dy+\int\varrho_0^{\alpha+1}v_y^2dy\Big)^{\frac{2q}{q-1}}
+\epsilon\int\frac{\varrho_0^\alpha}{J_0}|\frac{v_y}{J}|^qJ\xi_r^3dy\\
&&+C(\epsilon)\left(\|\varrho_0^{\frac{\alpha}{2}}J_{0y}\|_{L^2}\|(\varrho_0^\alpha)_y\|_{L^p}\|(\varrho_0^\alpha)_y\|_{L^\infty}(\frac{1}{\underline{J}})^{\frac{1}{q}} (\sup\limits_{y\in\mathbb R}\varrho_0)^{-\frac{(3q-2)\alpha+q}{2q}}\right)^\frac{4q}{5-q},\\
&|\Pi_2|&=\Big|-3\int\frac{\varrho_0^\alpha}{J_0}|\frac{v_y}{J}|^{q-2}\frac {v_y}Jv\xi_r^2\xi_r^{\prime}dy\Big|\\
&&\leqslant \Big(\int\varrho_0^{\alpha+1}v^2dy\Big)^{\frac{2q}{q-1}}+\epsilon\int\frac{\varrho_0^\alpha}{J_0}|\frac{v_y}{J}|^qJ\xi_r^3dy
+C(\epsilon)\left(\int\varrho_0^{-\frac{q+(q-2)\alpha}{2(2-q)}}|\xi_r^\prime|^{\frac{q}{2-q}}dy\right)^{\frac{4(2-q)}{5-q}},\\
&|\Pi_3|&=\Big|\int R\varrho\Theta(\frac{\varrho_0^\alpha}{J_0})_yv\xi_r^3dy\Big|\\
&&\leqslant \Big(\int\varrho_0^{\alpha+1}v^2dy+\int\varrho_0^{\alpha+2}\Theta^2dy\Big)^{\frac{2q}{q-1}}+C\left(\overline{J}\frac{1}{\underline{J}}
(\sup\limits_{y\in\mathbb R}\varrho_0)^{-\frac{2\alpha+1}{2}}\|\varrho_0^{\frac{\alpha}{2}}J_{0y}\|_{L^2}\right)^{\frac{q+1}{2q}},\\
&|\Pi_4|&=\Big|3\int R\varrho\Theta\frac{\varrho_0^\alpha}{J_0} v\xi_r^2\xi_r^{\prime}dy\Big|\\
&&\leqslant C\Big(\int\varrho_0^{\alpha+1}\Theta\xi_r^3dy\Big)\|v\|_{L^\infty}\|\xi_r^\prime\|_{L^\infty}(\frac{1}{\underline{J}})\\
&&\leqslant\Big(\int\varrho_0^{\alpha+1}\Theta dy+\int\varrho_0^{\alpha+1}v^2dy+\int\varrho_0^{\alpha+1}v_y^2dy\Big)^\frac{2q}{q-1}+C\left((\sup\limits_{y\in\mathbb R}\varrho_0)^{-\frac{\alpha+1}{2}}\|\xi_r^\prime\|_{L^\infty}\frac{1}{\underline{J}}\right)^\frac{3+q}{4q},\\
&|\Pi_5|&=\Big|-\int|\frac{\Theta_y} J|^{p-2}\frac {\Theta_y}J(\frac{\varrho_0^\alpha}{J_0})_y\xi_r^3dy\Big|\\
&&\leqslant \eta\int\frac{\varrho_0^{\alpha+1}}{J_0}|\frac{\Theta_y}{J}|^pJdy+C(\eta)\left(\frac{1}{\underline{J}}\right)^{p+1}
(\sup\limits_{y\in\mathbb R}\varrho_0)^{-(2\alpha-p-1)}\|\varrho_0^{\frac{\alpha}{p}}J_{0y}\|_{L^p}^p,\\
&|\Pi_6|&=\Big|-3\int\frac{\varrho_0^\alpha}{J_0}|\frac{\Theta_y} J|^{p-2}\frac {\Theta_y}J\xi_r^2\xi_r^{\prime}dy\Big|\\
&&\leqslant \eta\int\frac{\varrho_0^{\alpha+1}}{J_0}|\frac{\Theta_y}{J}|^pJdy+C(\eta)\left(\frac{1}{\underline{J}}\right)^{(p-1)}\int\varrho_0^{\alpha-p+1}\xi_r^{3-p}|\xi_r^\prime|^pdy,\\
&|\Pi_7|&=\Big|\frac12\int R\frac{\varrho_0^\alpha}{J_0}\varrho\Theta v_y\xi_r^3dy\Big|\\
&&\leqslant\Big(\int\varrho_0^{\alpha+1}\Theta dy+\int\varrho_0^{\alpha+1}v_y^2dy\Big)^{\frac{2q} {q-1}}
+\eta\int\frac{\varrho_0^{\alpha}}{J_0}\left|\frac{v_y}{J}\right|^q\frac{v_{yy}^2}{J} dy+
C(\eta)\left(\Big(\frac1 {\underline J}\Big)\overline{J}\right)^{\frac{2q(q+2)}{q^2-2q+7}}
\end{eqnarray*}
hold for any given $\eta\in(0,1)$ and $\epsilon\in(0,1)$. So, one obtains that
\begin{eqnarray}\label{5dlg-E13}
&&\frac d {dt}\int\Big(\varrho_0^{\alpha+1}v^2+\varrho_0^{\alpha+1}\Theta\Big)\xi_r^3dy+\int\frac{\varrho_0^\alpha}{J_0}|\frac{v_y}{J}|^qJ\xi_r^3dy\\
&&\leqslant \left(\int\varrho_0^{\alpha+1}v^2dy+\int\varrho_0^{\alpha+1}\Theta dy+\int\varrho_0^{\alpha+1}v_y^2dy
\right)^{\frac{2q}{q-1}}\nonumber\\
&&+\eta\int\frac{\varrho_0^{\alpha+1}}{J_0}|\frac{\Theta_y}{J}|^pJ\xi_r^3dy+\eta\int\frac{\varrho_0^{\alpha}}{J_0}\left|\frac{v_y}{J}\right|^q\frac{v_{yy}^2}{J} dy\nonumber\\
&&+C(\eta)(\sup\limits_{y\in\mathbb R}\varrho_0)^{-(2\alpha+1)}\|\xi_r^\prime\|_{L^\infty}\left(\frac{1}{\underline{J}}\right)^{q+2}
\left(\int(\varrho_0^\alpha)_y^{\frac{2q}{2-q}}dy\right)^{\frac{2(2-q)}{5-q}}\overline{J}^{q+2}\nonumber\\
&&+C\left(\|\varrho_0^{\frac{\alpha}{2}}J_{0y}\|_{L^2}\|(\varrho_0^\alpha)_y\|_{L^p}\|(\varrho_0^\alpha)_y\|_{L^\infty}(\frac{1}{\underline{J}})^{\frac{1}{q}} (\sup\limits_{y\in\mathbb R}\varrho_0)^{-\frac{(3q-2)\alpha+q}{2q}}\right)^\frac{4q}{5-q}\nonumber\\
&&+C(\eta)\left(\frac{1}{\underline{J}}\right)^{(p-1)}\int\varrho_0^{\alpha-p+1}|\xi_r^\prime|^pdy
+C(\eta)\left(\int\varrho_0^{-\frac{q+(q-2)\alpha}{2(2-q)}}|\xi_r^\prime|^{\frac{q}{2-q}}dy\right)^{\frac{4(2-q)}{5-q}}\nonumber
\end{eqnarray}
holds for any given $\eta\in(0,1).$

{\sl\textit{Step 2.}} Multiplying both sides of equation (\ref{1dlg-E15}) by $\varrho_0^{\alpha+1}\Theta\xi_r^3$, and integrating the results over $\mathbb R$, one finds that
\begin{eqnarray}\label{5dlg-E15}
&&\frac1 2\frac d {dt}\int\varrho_0^{\alpha+2}\Theta^2\xi_r^3dy+\int\frac{\varrho_0^{\alpha+1}}{J_0}|\frac{\Theta_y} J|^pJ\xi_r^3dy\nonumber\\
&&=-\int|\frac{\Theta_y} J|^{p-2}\frac {\Theta_y}J(\frac{\varrho_0^{\alpha+1}}{J_0})_y\Theta \xi_r^3dy-\int\frac{\varrho_0^{\alpha+1}}{J_0}R\varrho\Theta^2{v_y}\xi_r^3dy+\int\frac{\varrho_0^{\alpha+1}}{J_0}|\frac{v_y}J|^{q}J\Theta \xi_r^3dy\nonumber\\
&&-3\int|\frac{\Theta_y} J|^{p-2}\frac {\Theta_y}J\frac{\varrho_0^{\alpha+1}}{J_0}\Theta \xi_r^2\xi_r^{\prime}dy:=\sum_{i=1}^4{\mathbb G}_i.
    \end{eqnarray}
Each term on the right hand of (\ref{5dlg-E15}) is estimated as follows
\begin{eqnarray*}
&&|{\mathbb G}_1|=\Big|-\int|\frac{\Theta_y} J|^{p-2}\frac {\Theta_y}J(\frac{\varrho_0^{\alpha+1}}{J_0})_y\Theta \xi_r^3dy\Big|\\
&&\leqslant \Big(\int\varrho_0^{\alpha+2}\Theta^2dy+\int\varrho_0^{\alpha+1}\Theta_y^2dy\Big)^{\frac{2q}{q-1}}
+\epsilon\int\frac{\varrho_0^{\alpha+1}}{J_0}|\frac{\Theta_y}{J}|^pJ\xi_r^3dy\\
&&+C(\epsilon)\left(\|\varrho_0^{\frac{\alpha}{2}}J_{0y}\|_{L^2}\|(\varrho_0^\alpha)_y\|_{L^p}
\|(\varrho_0^\alpha)_y\|_{L^\infty}(\frac{1}{\underline{J}})^{\frac{1}{p}} (\sup\limits_{y\in\mathbb R}\varrho_0)^{-\frac{(3p-2)\alpha+p}{2p}}\right)^\frac{4q}{5-q},\\
&&|{\mathbb G}_2|=\Big|-\int\frac{\varrho_0^{\alpha+1}}{J_0}R\varrho\Theta^2{v_y}\xi_r^3dy\Big|\\
&&\leqslant\Big(\int\varrho_0^{\alpha+2}\Theta^2 dy+\int\varrho_0^{\alpha+1}v^2dy\Big)^{\frac{2q} {q-1}}
+\eta\int\frac{\varrho_0^{\alpha}}{J_0}\left|\frac{v_y}{J}\right|^q\frac{v_{yy}^2}{J} dy\\
&&+C(\eta)\left({\overline J}\Big(\frac{1 }{\underline J}\Big)(\sup\limits_{y\in\mathbb R}\varrho_0)^{-\frac{2\alpha+1}{q+2}}\right)^{\frac{2q(q+2)}{q^2-2q+7}},\\
&&|{\mathbb G}_3|=\Big|\int\frac{\varrho_0^{\alpha+1}}{J_0}|\frac{v_y}J|^{q}J\Theta \xi_r^3dy\Big|\\
&&\leqslant
\left(\int\varrho_0^{\alpha+1}\Theta_y^2dy+\int\varrho_0^{\alpha+2}\Theta^2dy+\int\frac{\varrho_0^{\alpha}}{J_0}|\frac{v_y}J|^{q}Jdy\right)^{\frac{2q}{q-1}}
+C\left(\overline{\varrho}^{-\frac{2\alpha+1}{4}}\right)^{\frac{2q}{q+1}}\\
&&|{\mathbb G}_4|=\Big|-3\int|\frac{\Theta_y} J|^{p-2}\frac {\Theta_y}J\frac{\varrho_0^{\alpha+1}}{J_0}\Theta \xi_r^2\xi_r^{\prime}dy\Big|\\
&&\leqslant \Big(\int\frac{\varrho_0^{\alpha+2}}{J_0}|\frac{\Theta_y} J|^{p}J\xi_r^3dy+\int\varrho_0^{\alpha+2}\Theta^2 \xi_r^3dy\Big)^\frac{2q}{q-1}\\
&&+C\left(\int\varrho_0^{\frac{(2-p)\alpha-2(p-1)}{2-p}}\xi_r^\frac{p(3-p)}{2-p}|\xi_r^\prime|^\frac{2p}{2-p}dy\right)^\frac{2q(2-p)}{pq+3p+2q-2}
    \end{eqnarray*}
hold for any given $\eta\in(0,1)$. So, it is easy to get that
\begin{eqnarray}\label{5dlg-E16}
&&\frac1 2\frac d {dt}\int\varrho_0^{\alpha+2}\Theta^2\xi_r^3dy+\int\frac{\varrho_0^{\alpha+1}}{J_0}|\frac{\Theta_y} J|^pJ\xi_r^3dy\nonumber\\
&&\leqslant \Big(\int\varrho_0^{\alpha+2}\Theta^2dy+\int\varrho_0^{\alpha+1}\Theta_y^2dy+\int\varrho_0^{\alpha+1}v^2dy+\int\frac{\varrho_0^{\alpha+2}}{J_0}|\frac{\Theta_y} J|^{p}Jdy\Big)^{\frac{2q} {q-1}}\nonumber\\
&&+\eta\int\frac{\varrho_0^{\alpha}}{J_0}\left|\frac{v_y}{J}\right|^q\frac{v_{yy}^2}{J} dy
+C\left(\|\varrho_0^{\frac{\alpha}{2}}J_{0y}\|_{L^2}\|(\varrho_0^\alpha)_y\|_{L^p}
\|(\varrho_0^\alpha)_y\|_{L^\infty}(\frac{1}{\underline{J}})^{\frac{1}{p}} (\sup\limits_{y\in\mathbb R}\varrho_0)^{-\frac{(3p-2)\alpha+p}{2p}}\right)^\frac{4q}{5-q}\nonumber\\
&&+C\left(\int\varrho_0^{\frac{(2-p)\alpha-2(p-1)}{2-p}}\xi_r^\frac{p(3-p)}{2-p}|\xi_r^\prime|^\frac{2p}{2-p}dy\right)^\frac{2q(2-p)}{pq+3p+2q-2}
    \end{eqnarray}
holds for any given $\eta\in(0,1)$.

{\sl\textit{Step 3.}} Differentiating (\ref{1dlg-E14}) with respect to $y,$ multiplying the resultant by $\varrho_0^{\alpha}v_y\xi_r^2$ and integrating the results over $\mathbb R$, one gets that
\begin{eqnarray}\label{5dlg-E17}
&&\frac1 2\frac d {dt}\int\varrho_0^{\alpha+1}v_y^2\xi_r^2dy+(q-1)\int\frac{\varrho_0^\alpha}{J_0}|\frac{v_y}{J}|^{q-2}\frac {v_{yy}^2}J\xi_r^2dy\nonumber\\
&&=-\int\varrho_{0y}\frac{\varrho_0^\alpha}{J_0}v_tv_y\xi_r^2dy-(q-1)\int|\frac{v_y}{J}|^{q-2}\frac {v_{yy}} {J}(\frac{\varrho_0^\alpha}{J_0})_yv_y\xi_r^2dy+\int(q-1)|\frac{v_y}{J}|^{q}J_y (\frac{\varrho_0^\alpha}{J_0})_y\xi_r^2dy\nonumber\\
&&+\int\frac{\varrho_0^\alpha}{J_0}(q-1)|\frac{v_y}{J}|^{q-2}\frac {{v_y}{J_y}} {J^2}v_{yy}\xi_r^2dy+\int R(\varrho\Theta)_y(\frac{\varrho_0^\alpha}{J_0})_{y}v_y\xi_r^2dy+\int R\frac{\varrho_0^\alpha}{J_0}(\varrho\Theta)_yv_{yy}\xi_r^2dy,\nonumber\\
&&-2(q-1)\int\frac{\varrho_0^\alpha}{J_0}|\frac{v_y}{J}|^{q-2}\frac {v_{yy}} {J}v_y\xi_r\xi_r^{\prime}dy+2(q-1)\int\frac{\varrho_0^\alpha}{J_0}|\frac{v_y}{J}|^{q}J_y \xi_r\xi_r^{\prime}dy\nonumber\\
&&+2\int R(\varrho\Theta)_y\frac{\varrho_0^\alpha}{J_0} v_y\xi_r\xi_r^{\prime}dy:=\sum_{i=1}^9{\mathbb J}_i.
\end{eqnarray}
Now we estimate ${\mathbb J}_1$-${\mathbb J}_5$ on the right hand of (\ref{5dlg-E17}) as follows
\begin{eqnarray*}
&|{\mathbb J}_1|&=\Big|-\int\varrho_{0y}\frac{\varrho_0^\alpha}{J_0}v_tv_y\xi_r^2dy\Big|\\
&&\leqslant \Big(\int\varrho_0^{\alpha+1}v_y^2dy\Big)^{\frac{2q}{q-1}}
+\eta\int\varrho_0^{\alpha+1}v_t^2dy+C(\eta)(\sup\limits_{y\in{\mathbb R}}\varrho_0)^{-\frac{4q\alpha}{q+1}}\|(\varrho_0^\alpha)_y\|_{L^\infty}^{\frac{4q}{q+1}},\\
&|{\mathbb J}_2|&=\Big|-\int(q-1)|\frac{v_y}{J}|^{q-2}\frac {v_{yy}} {J}(\frac{\varrho_0^\alpha}{J_0})_yv_y\xi_r^2dy\Big|\\
&&\leqslant \Big(\int\varrho_0^{\alpha+1}v_y^2dy\Big)^{\frac{2q}{q-1}}+\eta\int\frac{\varrho_0^\alpha}{J_0}|\frac{v_y}{J}|^{q-2}\frac {v_{yy}^2}Jdy\\
&&+C(\eta)(\sup\limits_{y\in{\mathbb R}}\varrho_0)^{-\frac{16\alpha}{5-q}}\Big(\frac1 {\underline J}\Big)^{\frac {2(q-1)(q+2)}{5-q}}{\overline J}^{\frac {5(q-1)}{5-q}}\|(\varrho_0^\alpha)_y\|_{L^2}^{\frac{2(q+2)} {5-q}}\|\varrho_0^{\frac{\alpha}{2}}J_{0y}\|_{L^2}^{\frac{2(q+2)} {5-q}},\\
&|{\mathbb J}_3|&=\Big|\int(q-1)|\frac{v_y}{J}|^{q}J_y (\frac{\varrho_0^\alpha}{J_0})_y\xi_r^2dy\Big|\\
&&\leqslant \Big(\int\varrho_0^{\alpha} J_y^2dy+\int\varrho_0^{\alpha+1}v_y^2dy\Big)^{\frac{2q} {q-1}}+\eta\int\frac{\varrho_0^\alpha}{J_0}|\frac{v_y}{J}|^{q-2}\frac {v_{yy}^2}Jdy\\
&&+C(\eta)\left((\sup\limits_{y\in{\mathbb R}}\varrho_0)^{-\frac{2\alpha+1}2}\left(\frac{1}{\underline{J}}
\right)^{q}{\overline J}^{\frac{(q-1)^2}{q+2}}\|(\varrho_0^{\alpha})_y\|_{L^\infty}\|\varrho_0^{\frac{\alpha}{2}}J_{0y}\|_{L^2}\right)^{\frac{2q(q+2)}{17q-q^2+2}},\\
&|{\mathbb J}_4|&=(q-1)\Big|\int\frac{\varrho_0^\alpha}{J_0}|\frac{v_y}{J}|^{q-2}\frac {{v_y}{J_y}} {J^2}v_{yy}\xi_r^2dy\Big|\\
&&\leqslant \Big(\int\varrho_0^{\alpha} J_y^2dy+\int\varrho_0^{\alpha+1}v_y^2dy\Big)^{\frac{2q} {q-1}}+\eta\int\frac{\varrho_0^\alpha}{J_0}|\frac{v_y}{J}|^{q-2}\frac {v_{yy}^2}Jdy\\
&&+C(\eta)\left(\Big(\frac1 {\underline J}\Big)^{\frac {q+1}{2}}(\sup\limits_{y\in{\mathbb R}}\varrho_0)^{-\frac{q(2\alpha+1)}{2(q+2)}}
{\overline J}^{\frac{q^2-1}{2(q+2)}}\right)^{\frac{2q(q+2)}{2q-q^2+1}},\\
&|{\mathbb J}_5|&=\Big|\int R(\varrho\Theta)_y(\frac{\varrho_0^\alpha}{J_0})_{y}v_y\xi_r^2dy\Big|\\
&&\leqslant R\left(\Big|\int \frac{ 1}{ J}J_{0y}\varrho_{0}\Theta(\varrho_0^\alpha)_{y}v_y\xi_r^2dy\Big|
+\Big|\int\frac{ 1}{ J}\varrho_{0y}J_0\Theta(\frac{\varrho_0^\alpha}{J_0})_{y}v_y\xi_r^2dy\Big|\right.\\
&&~~~~+\left.\Big|\int \frac 1{J^2}\Theta J_0\varrho_0J_y(\frac{\varrho_0^\alpha}{J_0})_{y}v_y\xi_r^2 dy\Big|+\Big|\int \frac{ 1}{ J}J_0\varrho_0\Theta_y(\frac{\varrho_0^\alpha}{J_0})_{y}v_y\xi_r^2dy\Big|\right)\\
&&:=\sum_{n=1}^4|J_n|,
    \end{eqnarray*}
where
\begin{eqnarray*}
&|J_1|&=R\Big|\int \frac{ 1}{ J}J_{0y}\varrho_{0}\Theta(\frac{\varrho_0^\alpha}{J_0})_{y}v_ydy\xi_r^2\Big|\\
&&\leqslant \Big(\int\varrho_0^{\alpha+2}\Theta^2dy+\int\varrho_0^{\alpha+1}v_y^2dy\Big)^{\frac{2q}{q-1}}+\eta\int\frac{\varrho_0^\alpha}{J_0}|\frac{v_y}{J}|^{q-2}\frac {v_{yy}^2}Jdy\\
&&+C(\eta)\left((\sup\limits_{y\in{\mathbb R}}\varrho_0)^{-\frac{2\alpha+1}{q+2}}\overline{J}^{\frac{q-1}{q+2}}\Big(\frac1 {\underline J}\Big)\left(\int\varrho_0^{\alpha} J_{0y}^2dy\right)^{\frac12}\right)^{\frac{4q(q+2)}{3q^2+q+4}},\\
&|J_2|&=R\Big|\int\frac{ 1}{ J}J_0\varrho_{0y}\Theta(\frac{\varrho_0^\alpha}{J_0})_{y}v_y\xi_r^2dy\Big|\\
&&\leqslant \Big(\int\varrho_0^{\alpha+2}\Theta^2dy+\int\varrho_0^{\alpha+1}v_y^2dy\Big)^{\frac{2q}{q-1}}\\
&&+C\left((\sup\limits_{y\in{\mathbb R}}\varrho_0)^{-\frac{4\alpha+1}2}\overline{J}\Big(\frac1 {\underline J}\Big)\|(\varrho_0^{\alpha})_y\|^2_{L^\infty}\left(\int\varrho_0^{\alpha} J_{0y}^2dy\right)\right)^{\frac{2q}{q+1}},\\
&|J_3|&=R\Big|\int \frac 1{J^2}\Theta J_0\varrho_0J_y(\frac{\varrho_0^\alpha}{J_0})_{y}v_y\xi_r^2dy\Big|\\
&&\leqslant\Big(\int\varrho_0^{\alpha} J_y^2dy+\int\varrho_0^{\alpha+2}\Theta^2dy+\int\varrho_0^{\alpha+1}v_y^2dy\Big)^{\frac{2q} {q-1}}+\eta\int\frac{\varrho_0^\alpha}{J_0}|\frac{v_y}{J}|^{q-2}\frac {v_{yy}^2}Jdy\\
&&+C\left((\sup\limits_{y\in{\mathbb R}}\varrho_0)^{-\frac{\alpha(q+4)+1}{q+2}}{\overline J}^{\frac{2q+1}{q+2}}\Big(\frac1 {\underline J}\Big)^2\|(\varrho_0^{\alpha})_y\|_{L^\infty}\left(\int\varrho_0^{\alpha} J_{0y}^2dy\right)\right)^{\frac{2q(q+2)}{q^2+2q+3}},\\
&|J_4|&=R\Big|\int \frac{ 1}{ J}J_0\varrho_0\Theta_y(\frac{\varrho_0^\alpha}{J_0})_{y}v_y\xi_r^2dy\Big|\\
&&\leqslant \Big(\int\varrho_0^{\alpha+1}\Theta_y^2dy+\int\varrho_0^{\alpha+1}v_y^2dy\Big)^{\frac{2q} {q-1}}\\
&&+C\left((\sup\limits_{y\in{\mathbb R}}\varrho_0)^{-2\alpha}\overline{J}\Big(\frac1 {\underline J}\Big)\|(\varrho_0^\alpha)_y\|_{L^\infty}\left(\int\varrho_0^{\alpha} J_{0y}^2dy\right)\right)^{\frac{2q} {q+1}}
\end{eqnarray*}
hold for any given $\eta\in(0,1)$ and $\epsilon\in(0,1).$ For ${\mathbb J}_6$, one finds that
\begin{eqnarray*}
&|{\mathbb J}_6|&=\Big|\int R\frac{\varrho_0^\alpha}{J_0}(\varrho\Theta)_yv_{yy}\xi_r^2dy\Big|\\
&&\leqslant R\left(\Big|\int \frac{ 1}{ J}J_{0y}\varrho_0^{\alpha+1} \Theta v_{yy}\xi_r^2dy\Big|
+\Big|\int\frac{ 1}{ J}\varrho_{0y}\varrho_0^\alpha\Theta v_{yy}\xi_r^2dy\Big|\right.\\
&&~~~~+\left.\Big|\int \frac 1{J^2}J_y\varrho_0^{\alpha+1}\Theta v_{yy}\xi_r^2dy\Big|+\Big|\int \frac{ 1}{ J}\varrho_0^{\alpha+1}\Theta_yv_{yy}\xi_r^2 dy\Big|\right)\\
&&:=\sum_{n=1}^4|{\widetilde J}_n|,
\end{eqnarray*}
where
\begin{eqnarray*}
&|{\widetilde J}_1|&=R\Big|\int \frac{ 1}{ J}J_{0y}\varrho_0^{\alpha+1}\Theta v_{yy} \xi_r^2dy\Big|\\
&&\leqslant \left(\int\varrho_0^{\alpha+1}v_y^2dy+\int\varrho_0^{\alpha+2}\Theta^2dy\right)^{\frac{2q}{q-1}}
+\eta\int\frac{\varrho_0^\alpha}{J_0}|\frac{v_y}{J}|^{q-2}\frac {v_{yy}^2}Jdy\\
&&+C(\eta)\left((\sup\limits_{y\in{\mathbb R}}\varrho_0)^{-\frac{(10-2q)\alpha+(2-3q)}{4(q+2)}}\overline{J}^{\frac{(q-1)(2-q)}{2(q+2)}}\Big(\frac1 {\underline J}\Big)^{\frac{5-q}2}\|(\varrho_0^{\alpha})_y\|_{L^\infty}\Big(\int\varrho_0^{\alpha}J_{0y}\xi_r^2dy\Big)^{\frac12}\right)^{\frac{q(q+2)}{q^2+1}},\\
&|{\widetilde J}_2|&=R\Big|\int\frac{ 1}{ J}\varrho_{0y}\varrho_0^\alpha\Theta v_{yy} \xi_r^2dy\Big|\\
&&\leqslant \left(\int\varrho_0^{\alpha+1}v_y^2dy+\int\varrho_0^{\alpha+2}\Theta^2dy\right)^{\frac{2q}{q-1}}+\eta\int\frac{\varrho_0^\alpha}{J_0}|\frac{v_y}{J}|^{q-2}\frac {v_{yy}^2}Jdy\\
&&+C(\eta)\left((\sup\limits_{y\in{\mathbb R}}\varrho_0)^{-\frac{(6-q)\alpha+4}{q+2}}\overline{J}^{\frac{2-q}{2(q+2)}}\Big(\frac1 {\underline J}\Big)^{\frac{q-1}2}\|(\varrho_0^{\alpha})_y\|_{L^\infty}\right)^{\frac{q(q+2)}{q^2-q+1}},\\
&|{\widetilde J}_3|&=\Big|\int \frac{1}{J^2}J_y\varrho_0^{\alpha+1}\Theta v_{yy}\xi_r^2dy\Big|\\
&&\leqslant \left(\int\varrho_0^{\alpha} J_y^2dy+\int\varrho_0^{\alpha+1}v_y^2dy+\int\varrho_0^{\alpha+2}\Theta^2dy+\int\varrho_0^{\alpha+1}\Theta_y^2dy\right)^{\frac{2q}{q-1}}
+\eta\int\frac{\varrho_0^\alpha}{J_0}|\frac{v_y}{J}|^{q-2}\frac {v_{yy}^2}Jdy\\
&&+C(\eta)\left((\sup\limits_{y\in{\mathbb R}}\varrho_0)^{-\frac{(10-2q)\alpha+(2-3q)}{4(q+2)}}\overline{J}^{\frac{2-q}{2(q+2)}}\Big(\frac1 {\underline J}\Big)^{\frac{3-q}2}\|(\varrho_0^{\alpha})_y\|_{L^\infty}\right)^{\frac{4q(q+2)}{5q^2-7q+6}}\\
&|{\widetilde J}_4|&=\Big|\int \frac{ 1}{ J}\varrho_0^{\alpha+1}\Theta_y v_{yy}\xi_r^2dy\Big|\\
&&\leqslant \left(\int\varrho_0^{\alpha+1}v_y^2dy+\int\varrho_0^{\alpha+1}\Theta_y^2dy\right)^{\frac{2q}{q-1}}
+\eta\int\frac{\varrho_0^\alpha}{J_0}|\frac{v_y}{J}|^{q-2}\frac {v_{yy}^2}Jdy\\
&&+C(\eta)\left((\sup\limits_{y\in{\mathbb R}}\varrho_0)^{-\frac{2\alpha(2-q)-(q-1)}{q+2}}\overline{J}^{\frac{2-q}{2(q+2)}}\Big(\frac1 {\underline J}\Big)^{\frac{5-q}2}\right)^{\frac{q(q+2)}{q^2-q+1}}
\end{eqnarray*}
hold for any given $\eta\in(0.1).$  For ${\mathbb J}_7$-${\mathbb J}_9,$ the estimate are given as follows
\begin{eqnarray*}
&|{\mathbb J}_7|&=\Big|-2\int(q-1)\frac{\varrho_0^\alpha}{J_0}|\frac{v_y}{J}|^{q-2}\frac {v_{yy}} {J}v_y\xi_r\xi_r^{\prime}dy\Big|\\
&&\leqslant \Big(\int\frac{\varrho_0^\alpha}{J_0}|\frac{v_y}{J}|^qJdy\Big)^{\frac{2q}{q-1}}+\eta\int\frac{\varrho_0^\alpha}{J_0}|\frac{v_y}{J}|^{q-2}\frac {v_{yy}^2}Jdy+
C(\eta)\left(\|\xi_r^\prime\|_{L^\infty}\Big(\frac1 {\underline J}\Big)\right)^\frac{4q}{q+1},\\
&|{\mathbb J}_8|&=\Big|2\int(q-1)\frac{\varrho_0^\alpha}{J_0}|\frac{v_y}{J}|^{q}J_y \xi_r\xi_r^{\prime}dy\Big|\\
&&\leqslant \Big(\int\frac{\varrho_0^\alpha}{J_0}|\frac{v_y}{J}|^qJdy\Big)^{\frac{1}{2}}\Big(\int\varrho_0^{\alpha} J_y^2\xi_r^2dy\Big)^{\frac{1}{2}}\|\xi_r^\prime\|_{L^\infty}\|v_y\|_{L^\infty}^\frac{q}{2}\left(\frac{1}{\underline{J}}\right)^\frac{q+1}{2}\\
&&\leqslant \Big(\int\frac{\varrho_0^\alpha}{J_0}|\frac{v_y}{J}|^qJdy+\int\varrho_0^{\alpha} J_y^2dy\Big)^{\frac{2q}{q-1}}+\eta\int\frac{\varrho_0^\alpha}{J_0}|\frac{v_y}{J}|^{q-2}\frac {v_{yy}^2}Jdy\\
&&+C(\eta)\left((\sup\limits_{y\in\mathbb R}\varrho_0)^{-\frac{2\alpha+1}{2(q+2)}}\overline{J}^\frac{q}{2(q+2)}\|\xi_r^\prime\|_{L^\infty}
\left(\frac{1}{\underline{J}}\right)^\frac{q+1}{2}\right)^\frac{4q(q+2)}{7q-q^2+2}
,\\
&|{\mathbb J}_9|&=\Big|2\int R(\varrho\Theta)_y\frac{\varrho_0^\alpha}{J_0} v_y\xi_r\xi_r^{\prime}dy\Big|\\
&&\leqslant\Big|2\int \frac{ 1}{ J}J_{0y}\varrho_0^{\alpha+1} v_y\xi_r\xi_r^{\prime}dy\Big|+\Big|2\int R\frac{ 1 }{ J}\varrho_{0y}\Theta\varrho_0^{\alpha+1} v_y\xi_r\xi_r^{\prime}dy\Big|\\
&&~~+\Big|2\int R\frac{ 1}{J^2}\Theta J_y\varrho_0^{\alpha+1} v_y\xi_r\xi_r^{\prime}dy\Big|+\Big|2\int R\frac{ 1 }{ J}\Theta_y\varrho_0^{\alpha+1} v_y\xi_r\xi_r^{\prime}dy\Big|\\
&&:=\sum_{k=1}^4|J_k|,
   \end{eqnarray*}
where
\begin{eqnarray*}
&|J_1|&=\Big|2\int \frac{ 1}{ J}J_{0y}\varrho_0^{\alpha+1} v_y\xi_r\xi_r^{\prime}dy\Big|\\
&&\leqslant \left(\int\varrho_0^{\alpha+1}v_y^2dy\right)^\frac{2q}{q-1}+C\left(\int\varrho_0^{\alpha} J_{0y}^2dy\right)^\frac{2q}{3q+1}(\sup\limits_{y\in\mathbb R}\varrho_0)^\frac{2q}{3q+1}\Big(\frac1 {\underline J}\Big)^\frac{2q}{3q+1}\|\xi_r^\prime\|_{L^\infty}^\frac{2q}{3q+1}\overline{J_0}^\frac{4q}{q+1},\\
&|J_2|&=\Big|2\int R\frac{ 1 }{ J}\varrho_{0y}\Theta\varrho_0^{\alpha+1} v_y\xi_r\xi_r^{\prime}dy\Big|\\
&&\leqslant \left(\int\varrho_0^{\alpha+1}v_y^2dy+\int\varrho_0^{\alpha+2}\Theta^2dy\right)^{\frac{2q}{q-1}}+C(\sup\limits_{y\in\mathbb R}\varrho_0)^{-\frac{q(2\alpha+1)}{q+1}}\|\xi_r^\prime\|_{L^\infty}^\frac{2q}{q+1}\|(\varrho_0^\alpha)_y\|_{\infty}^\frac{2q}{q+1}\overline{J_0}^\frac{2q}{q+1},\\
&|J_3|&=\Big|2\int R\frac{ 1}{J^2}\Theta J_y\varrho_0^{\alpha+1} v_y\xi_r\xi_r^{\prime}dy\Big|\\
&&\leqslant \left(\int\varrho_0^{\alpha+1}v_y^2dy+\int\varrho_0^{\alpha+2}\Theta^2\xi_r^2dy+\int\varrho_0^{\alpha}J_y^2dy\right)^{\frac{2q}{q-1}}
+\eta\int\frac{\varrho_0^\alpha}{J_0}|\frac{v_y}{J}|^{q-2}\frac {v_{yy}^2}Jdy\\
&&+C\left((\sup\limits_{y\in\mathbb R}\varrho_0)^{-\frac{q(2\alpha+1)}{q+2}}\|\xi_r^\prime\|_{L^\infty}\overline{J}^\frac{q}{q+2}\right)^\frac{4q(q+2)}{q^2+7},\\
&|J_4|&=\Big|2\int R\frac{ 1 }{ J}\Theta_y\varrho_0^{\alpha+1} v_y\xi_r\xi_r^{\prime}dy\Big|\\
&&\leqslant \left(\int\varrho_0^{\alpha+1}v_y^2dy+\int\varrho_0^{\alpha+2}\Theta_y^2dy\right)^{\frac{2q}{q-1}}
+C\left(\|\xi_r^\prime\|_{L^\infty}{\overline J}\frac{1}{\underline{J}}\right)^\frac{2q}{q+1}
    \end{eqnarray*}
hold for any given $\eta\in(0,1).$  In summary, one can get that
\begin{eqnarray}\label{5dlg-E18}
&&\frac d {dt}\int\varrho_0^{\alpha+1}v_y^2\xi_r^2dy+\int\frac{\varrho_0^\alpha}{J_0}|\frac{v_y}{J}|^{q-2}\frac {v_{yy}^2}J\xi_r^2dy\\
&&\leqslant\Big(\int\varrho_0^{\alpha+2}\Theta^2dy+\int\varrho_0^{\alpha} J_y^2dy+\int\varrho_0^{\alpha+1}\Theta_y^2dy+\int\varrho_0^{\alpha+1}v_y^2dy+\int\frac{\varrho_0^\alpha}{J_0}|\frac{v_y}{J}|^qJdy\Big)^{\frac{2q} {q-1}}\nonumber\\
&&+\eta\int\varrho_0^{\alpha+1}v_t^2dy+\eta\int\frac{\varrho_0^\alpha}{J_0}|\frac{v_y}{J}|^{q-2}\frac {v_{yy}^2}Jdy+C(\eta)(\sup\limits_{y\in{\mathbb R}}\varrho_0)^{-\frac{24\alpha}{5-q}}{\overline J}^{\frac {2q} {5-q}}\|(\varrho_0^\alpha)_y\|_{L^\infty}^{\frac{8q}{5-q}}\left(\frac{1}{\underline{J}}\right)^{\frac{2(q+2)}{5-q}}\nonumber\\
&&+C(\int\varrho_0^\alpha|\xi_r^{\prime}|^2dy)^\frac{2q(q+2)}{5-q}(\int\varrho_0^{\alpha} J_{0y}^2dy)^\frac{2q}{3q+1}((\sup\limits_{y\in\mathbb R}\varrho_0)^{-\frac{2\alpha+1}{2}}\|\xi_r^\prime\|_{L^\infty}\overline{J}^\frac{q+1}{q+2}
\frac{1}{\underline{J}}\|(\varrho_0^\alpha)_y\|_{L^\infty})^\frac{4q(q+2)}{5-q}.\nonumber
   \end{eqnarray}
holds for any given $\eta\in(0,1).$

{\sl\textit{Step 4.}} Differentiating (\ref{1dlg-E15}) with respect to $y,$ multiplying both sides of  the above differential equation by $\varrho_0^{\alpha}\Theta_y\xi_r^2$ and integrating the results over $\mathbb R$, one gets that
\begin{eqnarray}\label{5dlg-E19}
&&\frac1 2\frac d {dt}\int\varrho_0^{\alpha+1}\Theta_y^2\xi_r^2dy+(p-1)\int\frac{\varrho_0^\alpha}{J_0}|\frac{\Theta_y} J|^{p-2}\frac {\Theta_{yy}^2}J\xi_r^2dy\nonumber\\
&&=-\int\varrho_{0y}\varrho_0^{\alpha}\Theta_t\Theta_y\xi_r^2dy-(p-1)\int|\frac{\Theta_y} J|^{p-2}\frac {\Theta_{yy}} {J}(\frac{\varrho_0^\alpha}{J_0})_y\Theta_y\xi_r^2dy\nonumber\\
&&+(p-1)\int|\frac{\Theta_y} J|^{p}J_y (\frac{\varrho_0^\alpha}{J_0})_y\xi_r^2dy+(p-1)\int\varrho_0^{\alpha}|\frac{\Theta_y} J|^{p-2}\frac {{\Theta_y}{J_y}} {J^2}\Theta_{yy}\xi_r^2dy+\int R\varrho\Theta v_y(\frac{\varrho_0^\alpha}{J_0})_{y}\Theta_y\xi_r^2dy\nonumber\\
&&+\int R\varrho\Theta v_{y}\varrho_0^{\alpha}\Theta_{yy}\xi_r^2dy-\int|\frac{v_y}J|^{q}J(\frac{\varrho_0^\alpha}{J_0})_{y}\Theta_y\xi_r^2dy-\int|\frac{v_y}J|^{q}J\frac{\varrho_0^\alpha}{J_0}\Theta_{yy}\xi_r^2dy\nonumber\\
&&-2(p-1)\int|\frac{\Theta_y} J|^{p-2}\frac {\Theta_{yy}} {J}\frac{\varrho_0^\alpha}{J_0}\Theta_y\xi_r\xi_r^{\prime}dy+2(p-1)\int|\frac{\Theta_y} J|^{p}J_y \frac{\varrho_0^\alpha}{J_0}\xi_r\xi_r^{\prime}dy\nonumber\\
&&+2\int R\varrho\Theta v_y\frac{\varrho_0^\alpha}{J_0}\Theta_y\xi_r\xi_r^{\prime}dy-2\int|\frac{v_y}J|^{q}J\frac{\varrho_0^\alpha}{J_0}\Theta_y\xi_r\xi_r^{\prime}dy
:=\sum_{i=1}^{12}{\mathbb W}_i.
    \end{eqnarray}
Each term on the right hand of (\ref{5dlg-E19}) is estimated as follows
\begin{eqnarray*}
&|{\mathbb W}_1|&=\Big|-\int\varrho_{0y}\frac{\varrho_0^{\alpha}}{J_0}\Theta_t\Theta_y\xi_r^2dy\Big|\\
&&\leqslant \Big(\int\varrho_0^{\alpha+1}\Theta_y^2dy\Big)^{\frac{2q} {q-1}}+\eta\int\varrho_0^{\alpha+3}\Theta_t^2dy
+C(\eta)(\sup\limits_{y\in{\mathbb R}}\varrho_0)^{-\frac{4q(\alpha+1)}{3q+1}}\|(\varrho_0^\alpha)_y\|_{L^\infty}^{\frac{4q}{3q+1}}\left(\frac{1}{\underline J}\right)^{\frac{4q}{3q+1}},\\
&|{\mathbb W}_2|&=\Big|-(p-1)\int|\frac{\Theta_y} J|^{p-2}\frac {\Theta_{yy}} {J}(\frac{\varrho_0^{\alpha}}{J_0})_y\Theta_y\xi_r^2dy\Big|\\
&&\leqslant \Big(\int\varrho_0^{\alpha+1}\Theta_y^2dy\Big)^{\frac{2q}{q-1}}+\eta\int\frac{\varrho_0^\alpha}{J_0}|\frac{\Theta _y} J|^{p-2}\frac {\Theta _{yy}^2}Jdy\\
&&+C(\eta)\left((\sup\limits_{y\in{\mathbb R}}\varrho_0)^{-\frac{(p+2)\alpha+p}{4}}\Big(\frac{1}{\underline{J}}\Big)^{\frac{p-1}{2}}{\overline J}^{\frac12}\|(\varrho_0^{\alpha})_y\|_{L^\frac{4}{2-p}}\right)^{\frac{8q}{4+p-pq}}\\
&&+C(\eta)\left((\sup\limits_{y\in{\mathbb R}}\varrho_0)^{-\frac{p(2\alpha+1)\alpha+p}{2(p+2)}}{\overline J}^{\frac{p^2-1}{2(p+2)}}\Big(\frac{1}{\underline{J}}\Big)^{\frac{p}{2}}
\left(\int\varrho_0^{\alpha} J_{0y}^2dy\right)^{\frac12}\right)^{\frac{4q(p+2)}{3pq+4q+p}},\\
&|{\mathbb W}_3|&=\Big|(p-1)\int|\frac{\Theta_y} J|^{p}J_y (\frac{\varrho_0^{\alpha}}{J_0})_y\xi_r^2dy\Big|\\
&&\leqslant \Big(\int\varrho_0^{\alpha} J_y^2dy+\int\varrho_0^{\alpha+1}\Theta_y^2dy\Big)^{\frac{2q} {q-1}}+\eta\int\frac{\varrho_0^\alpha}{J_0}|\frac{\Theta _y} J|^{p-2}\frac {\Theta _{yy}^2}Jdy\\
&&+C(\eta)\left((\sup\limits_{y\in{\mathbb R}}\varrho_0)^{-\frac{(p+6)\alpha+2}{2(p+2)}}\overline{ J}^{\frac{2(p-1)}{p+2}}\left(\int\varrho_0^{\alpha} J_{0y}^2dy\right)^{\frac12}\right)^{\frac{2q(p+2)}{4q-2pq+p}}\\
&&+C(\eta)\left((\sup\limits_{y\in{\mathbb R}}\varrho_0)^{-\frac{p(2\alpha+1)}{p+2}}\overline{ J}^{\frac{p^2-1}{p+2}}\|(\varrho_0^\alpha)_y\|_{L^\infty}\right)^{\frac{4q(p+2)}{3pq+q+p}},\\
&|{\mathbb W}_4|&=\Big|(p-1)\int\frac{\varrho_0^{\alpha}}{J_0}|\frac{\Theta_y} J|^{p-2}\frac {{\Theta_y}{J_y}} {J^2}\Theta_{yy}\xi_r^2dy\Big|\\
&&\leqslant \Big(\int\varrho_0^{\alpha} J_y^2dy+\int\varrho_0^{\alpha+1}\Theta_y^2dy\Big)^{\frac{2q}{q-1}}+\eta\int\frac{\varrho_0^\alpha}{J_0}|\frac{\Theta _y} J|^{p-2}\frac {\Theta _{yy}^2}Jdy\\
&&+C(\eta)\left((\sup\limits_{y\in{\mathbb R}}\varrho_0)^{-\frac{p(2\alpha+1)}{2(p+2)}}\Big(\frac{1}{\underline J}\Big)^{\frac {p-1}{2}}{\overline J}^{\frac {p^2-1}{2(p+2)}}\right)^\frac{2q(p+2)}{q+p+1-pq},\\
&|{\mathbb W}_5|&=\Big|\int R\varrho\Theta v_y(\frac{\varrho_0^{\alpha}}{J_0})_{y}\Theta_y\xi_r^2dy\Big|\\
&&\leqslant \Big(\int\varrho_0^{\alpha+2}\Theta^2dy+\int\varrho_0^{\alpha+1}\Theta_y^2dy	+\int\varrho_0^{\alpha+1}v_y^2dy\Big)^{\frac{2q}{q-1}}+\eta\int\frac{\varrho_0^\alpha}{J_0}|\frac{v_y}{J}|^{q-2}\frac {v_{yy}^2}Jdy \\
&&+\eta\int\frac{\varrho_0^\alpha}{J_0}|\frac{\Theta _y} J|^{p-2}\frac {\Theta _{yy}^2}Jdy+C(\eta)\left((\sup\limits_{y\in{\mathbb R}}\varrho_0)^{-\frac {(q+4)\alpha}{q+2}}\Big(\frac1 {\underline J}\Big){\overline J}^\frac{2(q+1)}{q+2}\|(\varrho_0^{\alpha})_y\|_{L^\infty}\right)^{\frac {2q(q+2)}{q^2+3}}\\
&&+C(\eta)\left((\sup\limits_{y\in{\mathbb R}}\varrho_0)^{-\left(\frac {2\alpha+1}{q+2}+\frac {2\alpha+1}{p+2}\right)}\Big(\frac1 {\underline J}\Big){\overline J}^{\frac {2(q-1)}{q+2}+\frac {p-1}{p+2}}\left(\int\varrho_0^{\alpha} J_{0y}^2dy\right)^{\frac12}\right)^{\frac {4q(q+2)(p+2)}{3pq^2-8q+pq+4p+12}},\\
&|{\mathbb W}_6|&=\Big|\int R\varrho\Theta v_{y}\frac{\varrho_0^{\alpha}}{J_0}\Theta_{yy}\xi_r^2dy\Big|\\
 &&\leqslant \Big(\int\varrho_0^{\alpha+1}\Theta_y^2dy+\int\varrho_0^{\alpha+1}v_y^2dy+\int\varrho_0^{\alpha+2}\Theta^2dy\Big)^{\frac{2q} {q-1}}+\eta\int\frac{\varrho_0^{\alpha}}{J_0}|\frac{\Theta _y} J|^{p-2}\frac {\Theta _{yy}^2}Jdy\\
 &&+C(\eta)\left((\sup\limits_{y\in{\mathbb R}}\varrho_0)^{-\left(\frac {2\alpha+1}{q+2}+\frac{(2-p)(2\alpha+1)}{2(p+2)}\right)}
 \Big(\frac1 {\underline J}\Big)^{\frac {q-1}{q+2}+\frac{(2-p)(q-1)}{2(p+2)}}\overline{J}^{\frac{3-p}{2}}\right)^{\frac{2q(q+2)(p+2)}{4pq^2+2pq+2p-16q+12}},\\
&|{\mathbb W}_7|&=\Big|-\int(\frac{\varrho_0^{\alpha}}{J_0})_y\Theta_y|\frac{v_y}J|^{q}J\xi_r^2dy\Big|\\
&& \leqslant\Big(\int\varrho_0^{\alpha+1}\Theta_y^2dy	+\int\varrho_0^{\alpha+1}v_y^2Jdy\Big)^{\frac{2q}{q-1}}+\eta\int\frac{\varrho_0^{\alpha}}{J_0}|\frac{v_y}{J}|^{q-2}\frac {v_{yy}^2}{J}dy \\
&&+C(\eta)\left((\sup\limits_{y\in{\mathbb R}}\varrho_0)^{-\frac {(q-1)(2\alpha+1)}{q+2}}\Big(\frac1 {\underline J}\Big)^q{\overline J}^\frac{2(q-1)}{q+2}\|(\varrho_0^{\alpha})_y\|_{L^\infty}\right)^{\frac {4q(q+2)}{3(q^2+3)}}\\
&&+C(\eta)\left((\sup\limits_{y\in{\mathbb R}}\varrho_0)^{-\left(\frac {q(2\alpha+1)}{2(q+2)}+\frac{2\alpha+1}{p+2}\right)}
 \Big(\frac1 {\underline J}\Big)^{\frac{q}{2}}\overline{J}^{\frac {q^2-1}{2(q+2)}+\frac{p-1}{p+2}}\right)^{\frac{2q(q+2)(p+2)}{4pq-2q^2+4q-p+12}},\\
&|{\mathbb W}_8|&=\Big|-\int\frac{\varrho_0^{\alpha}}{J_0}\Theta_{yy}|\frac{v_y}J|^{q}J\xi_r^2dy\Big|\\
 &&\leqslant \Big(\int\varrho_0^{\alpha+1}\Theta_y^2dy+\int\frac{\varrho_0^{\alpha}}{J_0}\left|\frac{v_y}{J}\right|^qJdy\Big)^{\frac{2q} {q-1}}
 +\eta\int\frac{\varrho_0^{\alpha}}{J_0}|\frac{v_y}{J}|^{q-2}\frac {v_{yy}^2}Jdy+\eta\int\frac{\varrho_0^{\alpha}}{J_0}|\frac{\Theta _y} J|^{p-2}\frac {\Theta _{yy}^2}Jdy\\
 &&+C(\eta)\left((\sup\limits_{y\in{\mathbb R}}\varrho_0)^{-\frac {(6-p)(2\alpha+1)}{4(p+2)}}
 \Big(\frac1 {\underline J}\Big)^{\frac{1+q-p}{2}}\overline{J}^{\frac{2(q-1)}{p+2}}\right)^{\frac{4q(p+2)(q+2)}{5pq^2+12pq-4q^2+4q-2p+8}},\\
&|{\mathbb W}_9|&=\Big|-2(p-1)\int|\frac{\Theta_y} J|^{p-2}\frac {\Theta_{yy}} {J}\frac{\varrho_0^{\alpha}}{J_0}\Theta_y\xi_r\xi_r^{\prime}dy\Big|\\
&&\leqslant \Big(\int\frac{\varrho_0^{\alpha+2}}{J_0}|\frac{\Theta_y} J|^pJdy\Big)^{\frac{2q}{q-1}}+\eta\int\frac{\varrho_0^\alpha}{J_0}|\frac{\Theta _y} J|^{p-2}\frac {\Theta _{yy}^2}Jdy+C(\eta)\|\varrho_0^{-1}\xi_r^\prime\|_{L^\infty}^\frac{4q}{q+1},\\
&|{\mathbb W}_{10}|&=\Big|2(p-1)\int|\frac{\Theta_y} J|^{p}J_y \frac{\varrho_0^{\alpha}}{J_0}\xi_r\xi_r^{\prime}dy\Big|\\
&& \leqslant\Big(\int\varrho_0^{\alpha+1}\Theta_y^2dy	+\int\varrho_0^{\alpha}J_y^2dy\Big)^{\frac{2q}{q-1}}+\eta\int\frac{\varrho_0^{\alpha}}{J_0}|\frac{v_y}{J}|^{q-2}\frac {v_{yy}^2}{J}dy \\
&&+C(\eta)\left(\|\varrho_0^{-\frac12}\xi_r^\prime\|_{L^\infty}(\sup\limits_{y\in{\mathbb R}}\varrho_0)^{-\frac {(p-1)(2\alpha+1)}{p+2}}\Big({\overline J}\Big)^\frac {(p-1)^2}{p+2}\right)^{\frac {4q(q+2)}{9q-4pq+4p+3}},\\
&|{\mathbb W}_{11}|&=\Big|2\int R\varrho\Theta v_y\frac{\varrho_0^{\alpha}}{J_0}\Theta_y\xi_r\xi_r^{\prime}dy\Big|\\
&& \leqslant\Big(\int\varrho_0^{\alpha+1}\Theta_y^2dy	+\int\varrho_0^{\alpha+1}v_y^2dy+\int\varrho_0^{\alpha+2}\Theta^2dy\Big)^{\frac{2q} {q-1}}+\eta\int\frac{\varrho_0^{\alpha}}{J_0}|\frac{\Theta _y} J|^{p-2}\frac {\Theta _{yy}^2}Jdy\\
&&+C(\eta)\left(\|\xi_r^\prime\|_{L^\infty}(\sup\limits_{y\in{\mathbb R}}\varrho_0)^{-\frac {2\alpha+3}{4}}\Big(\frac{1}{\underline J}\Big)\right)^{\frac {4q}{q+3}},\\
&|{\mathbb W}_{12}|&=\Big|-2\int|\frac{v_y}J|^{q}J\frac{\varrho_0^{\alpha}}{J_0}\Theta_y\xi_r\xi_r^{\prime}dy\Big|\\
&&\leqslant \Big(\int\frac{\varrho_0^{\alpha}}{J_0}|\frac{v_y} J|^qJdy+\int\varrho_0^{\alpha+1}\Theta_y^2dy\Big)^{\frac{2q}{q-1}}+\eta\int\frac{\varrho_0^{\alpha}}{J_0}|\frac{v_y}{J}|^{q-2}\frac {v_{yy}^2}Jdy\\
&&+C(\eta)\left({\overline J}^\frac{p-1}{p+2}(\sup\limits_{y\in{\mathbb R}}\varrho_0)^{-\frac {2\alpha+1}{p+2}}\right)^{\frac {4q(q+2)}{3pq+q+p+3}}
   \end{eqnarray*}
hold for any given $\eta\in(0,1).$ Thus, one arrives that
{\small\begin{eqnarray}\label{5dlg-E20}
&&\frac d {dt}\int\varrho_0^{\alpha+1}\Theta_y^2\xi_r^2dy+\int\frac{\varrho_0^{\alpha}}{J_0}|\frac{\Theta_y} J|^{p-2}\frac {\Theta_{yy}^2}J\xi_r^2dy\\
&&\leqslant \left(\int\varrho_0^{\alpha+1}v_y^2dy+\int\frac{\varrho_0^\alpha}{J_0}|\frac{v_y}{J}|^qJdy+\int\frac{\varrho_0^{\alpha+2}}{J_0}|\frac{\Theta_y} J|^pJdy+\int\varrho_0^{\alpha+1}\Theta_y^2dy +\int\varrho_0^{\alpha}J_y^2dy+\int\varrho_0^{\alpha+2}\Theta^2dy\right)^{\frac{2q} {q-1}}\nonumber\\
&&+\eta\int\varrho_0^{\alpha+3}\Theta_t^2dy +\eta\int\frac{\varrho_0^{\alpha}}{J_0}|\frac{v_y}{J}|^{q-2}\frac {v_{yy}^2}Jdy+\eta\int\frac{\varrho_0^{\alpha}}{J_0}|\frac{\Theta_y} J|^{p-2}\frac {\Theta_{yy}^2}Jdy\nonumber\\
&&+C(\eta)\|\varrho_0^{-1}\xi_r^\prime\|^\frac{4q(q+2)}{q+1}\|\varrho_0^{-\frac12}\xi_r^\prime\|_{L^\infty}^{\frac {4q(q+2)}{9q-4pq+4p+3}}\left((\sup\limits_{y\in{\mathbb R}}\varrho_0)^{-8\alpha} \Big(\frac1 {\underline J}\Big)^{8}\overline{J}^{8}\|(\varrho_0^\alpha)_y\|_{L^\infty}^{8}\|(\varrho_0^\alpha)_y\|_{L^p}^{8}\right)\nonumber
\end{eqnarray}}
holds for any given $\eta\in(0,1).$

{\sl\textit{Step 5.}} Multiplying both sides of  (\ref{1dlg-E14}) by $\varrho_0^{\alpha}v_t\xi_r^3$ and integrating with respect to $y$, it runs that
\begin{eqnarray}\label{5dlg-E21}
&&\frac1 q\frac d {dt}\int\frac{\varrho_0^{\alpha}}{J_0}|\frac{v_y}{J}|^qJ\xi_r^3dy+\int\varrho_0^{\alpha+1}v_t^2\xi_r^3dy\nonumber\\
&&=-\int|\frac{v_y}{J}|^{q-2}\frac {v_y}J(\frac{\varrho_0^{\alpha}}{J_0})_yv_t\xi_r^3dy+\frac1 q\int \frac{\varrho_0^{\alpha}}{J_0}|\frac{v_y}{J}|^qJ_t\xi_r^3dy-\int\frac{\varrho_0^{\alpha}}{J_0}|\frac{v_y}{J}|^{q-2}\frac {1}{J^2}J_t{v_y^2}\xi_r^3dy\nonumber\\
&&-\int R(\varrho\Theta)_y\frac{\varrho_0^{\alpha}}{J_0} v_t\xi_r^3dy-3\int\frac{\varrho_0^{\alpha}}{J_0}|\frac{v_y}{J}|^{q-2}\frac {v_y}Jv_t\xi_r^2\xi_r^{\prime}dy\nonumber\\
&&:=\sum_{i=1}^5{\mathbb V}_i.
    \end{eqnarray}
The right hand of (\ref{5dlg-E21}) is estimated as follows
\begin{eqnarray*}
&|{\mathbb V}_1|&=\Big|-\int|\frac{v_y}{J}|^{q-2}\frac {v_y}J(\frac{\varrho_0^{\alpha}}{J_0})_yv_t\xi_r^3dy\Big|\\
&&\leqslant \epsilon\int\varrho_0^{\alpha+1}v_t^2\xi_r^3dy+\eta\int\frac{\varrho_0^\alpha}{J_0}|\frac{v_y}{J}|^qJdy
\\
&&+C(\eta)\left((\sup\limits_{y\in{\mathbb R}}\varrho_0)^{-\frac{(3q-2)\alpha+q}{2q}}
\Big(\frac{1}{\underline J}\Big)^{q-1}\|(\varrho_0^{\alpha})_y\|_{L^{\frac{2q}{2-q}}}\right)^{\frac{2q}{2-q}}\\
&&+C(\epsilon,\eta)\left((\sup\limits_{y\in{\mathbb R}}\varrho_0)^{-\frac{(q-1)(2\alpha+1)}{2+q}}
\Big(\frac{1}{\underline J}\Big)^{q-1}{\overline J}^{\frac{(q-1)^2}{2+q}}\|\varrho_0^{\frac{\alpha}{2}}J_{0y}\|_{L^2}\right)^{\frac{2(q+2)}{3(2-q)}},\\
&|{\mathbb V}_2|&=\Big|\frac1 q\int \frac{\varrho_0^{\alpha}}{J_0}|\frac{v_y}{J}|^qJ_t\xi_r^3dy\Big|\\
&&\leqslant \left(\int\varrho_0^{\alpha+1}v_y^2dy+\int\frac{\varrho_0^\alpha}{J_0}|\frac{v_y}{J}|^qJdy\right)^{\frac{2q}{q-1}}
+\eta\int\frac{\varrho_0^\alpha}{J_0}|\frac{v_y}{J}|^{q-2}\frac {v_{yy}^2}{J}dy+C(\eta)\left(\overline{\varrho}^{-\frac{2\alpha+1}{q+2}}\overline{J}^{\frac{q-1}{q+2}}\right)^{\frac{2q(q+2)}{q^2+3}},\\
&|{\mathbb V}_3|&=|-\int\frac{\varrho_0^{\alpha}}{J_0}|\frac{v_y}{J}|^{q-2}\frac {1}{J^2}J_t{v_y^2}dy|\\
&&\leqslant \left(\int\varrho_0^{\alpha+1}v_y^2dy\right)^{\frac{2q}{q-1}}+\eta\int\frac{\varrho_0^\alpha}{J_0}|\frac{v_y}{J}|^qJdy
+\eta\int\frac{\varrho_0^\alpha}{J_0}|\frac{v_y}{J}|^{q-2}\frac {v_{yy}^2}{J}dy\\
&&+C(\eta)\left((\sup\limits_{y\in{\mathbb R}}\varrho_0)^{-\frac{2\alpha+1}{q+2}}\overline{J}^{\frac{q-1}{q+2}}\right)^{\frac{2q(q+2)}{q^2-q+1}},\\
&|{\mathbb V}_4|&=\Big|-\int R(\varrho\Theta)_y\frac{\varrho_0^{\alpha}}{J_0} v_t\xi_r^3dy\Big|\\
&&\leqslant R\left|\int\frac{1}{J}J_{0y}\frac{\varrho_0^{\alpha+1}}{J_0}\Theta v_t \xi_r^3dy\right|+R\left|\int \frac{1}{J}\varrho_0^{\alpha}\varrho_{0y}\Theta v_t \xi_r^3dy\right|
+R\left|\int \frac{1}{J^2}J_y\varrho_0^{\alpha+1}\Theta v_t \xi_r^3dy\right|\\
&&+R\left|\int \frac{1}{J}\varrho_0^{\alpha+1}\Theta_y v_t\xi_r^3dy\right|:=\sum\limits_{i=1}^4J_i,
   \end{eqnarray*}
where
\begin{eqnarray*}
&|J_1|&=R\left|\int\frac{1}{J}J_{0y}\frac{\varrho_0^{\alpha+1}}{J_0}\Theta v_t \xi_r^3dy\right|\\
&&\leqslant\left(\int\rho_0^{\alpha+2}\Theta^2dy+\int\rho_0^{\alpha+1}\Theta_y^2dy\right)^{\frac{2q}{q-1}}+\epsilon\int\rho_0^{\alpha+1}v_t^2\xi_r^3dy\\
&&+C(\epsilon)\left(\frac{1}{\underline{J}}
(\sup\limits_{y\in{\mathbb R}}\varrho_0)^{-\frac{2\alpha+1}{4}}\|\varrho_0^{\frac{\alpha}{2}}J_{0y}\|_{L^2}\right)^{\frac{4q}{q+1}},\\
&|J_2|&=R\left|\int \frac{1}{J}\varrho_0^{\alpha}\varrho_{0y}\Theta v_t \xi_r^3dy\right|\\
&&\leqslant\left(\int\rho_0^{\alpha+2}\Theta^2dy\right)^{\frac{2q}{q-1}}+\epsilon\int\rho_0^{\alpha+1}v_t^2\xi_r^3dy
+C(\epsilon)\left(\frac{1}{\underline{J}}\overline{J}(\sup\limits_{y\in{\mathbb R}}\varrho_0)^{-(2\alpha+1)}\|\left(\varrho_0^\alpha\right)_y\|_{L^\infty}\right)^{\frac{4q}{q+1}},\\
&|J_3|&=R\left|\int \frac{1}{J^2}J_y\varrho_0^{\alpha+1}\Theta v_t \xi_r^3dy\right|\\
&&\leqslant\left(\int\rho_0^{\alpha+2}\Theta^2dy+\int\rho_0^{\alpha+1}\Theta_y^2dy+\int\varrho_0^\alpha J_y^2dy\right)^{\frac{2q}{q-1}}+\epsilon\int\rho_0^{\alpha+1}v_t^2\xi_r^3dy\\
&&+C(\epsilon)\left(\left(\frac{1}{\underline{J}}\right)^2\overline{J}(\sup\limits_{y\in{\mathbb R}}\varrho_0)^{-\frac{2\alpha+1}{4}}\right)^{2q},\\
&|J_4|&=R\left|\int \frac{1}{J}\varrho_0^{\alpha+1}\Theta_y v_t\xi_r^3dy\right|\\
&&\leqslant\left(\int\rho_0^{\alpha+1}\Theta_y^2dy+\int\varrho_0^\alpha J_y^2dy\right)^{\frac{2q}{q-1}}+\epsilon\int\rho_0^{\alpha+1}v_t^2\xi_r^3dy
+\eta\int\varrho_0^\alpha|\frac{\Theta _y} J|^{p-2}\frac {\Theta _{yy}^2}Jdy\\
&&+C(\epsilon,\eta)\left(\left(\frac{1}{\underline{J}}\right)\overline{J}^{\frac{p}{2(p+2)}}
\overline{\varrho}^{-\frac{2\alpha-2p-3}{2(p+2)}}\right)^{\frac{4q(p+2)}{2pq+q+2p+5}}
    \end{eqnarray*}
hold for any given $\epsilon\in(0,1)$ and $\eta\in(0,1).$ Moreover,
\begin{eqnarray*}
&|{\mathbb V}_5|&=\Big|-3\int\frac{\varrho_0^{\alpha}}{J_0}|\frac{v_y}{J}|^{q-2}\frac {v_y}Jv_t\xi_r^2\xi_r^{\prime}dy\Big|\\
&&\leqslant \eta\int\frac{\varrho_0^\alpha}{J_0}|\frac{v_y}{J}|^qJdy+\eta\int\varrho_0^{\alpha+1}v_t^2\xi_r^3dy+C(\eta)\int\varrho_0^{-\frac{3\alpha q-q-2\alpha}{2-q}}|\xi_r^{\prime}|^{\frac{2q}{2-q}}dy
	\end{eqnarray*}
holds for any given $\epsilon\in(0,1)$ and $\eta\in(0,1)$. So,
\begin{eqnarray}\label{5dlg-E22}
&&\frac d {dt}\int\frac{\varrho_0^\alpha}{J_0}|\frac{v_y}{J}|^qJ\xi_r^3dy+\int\varrho_0^{\alpha+1}v_t^2\xi_r^3dy\\
&&\leqslant \left(\int\rho_0^{\alpha+2}\Theta^2dy+\int\rho_0^{\alpha+1}\Theta_y^2dy+\int\varrho_0^\alpha J_y^2dy+\int\varrho_0^{\alpha+1}v_y^2dy+\int\frac{\varrho_0^\alpha}{J_0}|\frac{v_y}{J}|^qJdy\right)^{\frac{2q}{q-1}}\nonumber\\
&&+\eta\int\frac{\varrho_0^\alpha}{J_0}|\frac{v_y}{J}|^qJdy+\eta\int\frac{\varrho_0^\alpha}{J_0}|\frac{v_y}{J}|^{q-2}\frac {v_{yy}^2}Jdy+C(\eta)\int\varrho_0^{-\frac{3\alpha q-q-2\alpha}{2-q}}|\xi_r^{\prime}|^{\frac{2q}{2-q}}dy\nonumber\\
&&+C(\eta)\left(\|(\varrho_0^\alpha)_y\|_{L^q}\|(\varrho_0^\alpha)_y\|_{L^\infty}\right)^{\frac{2q}{2-q}}\overline{J}^{\frac{4q}{q+1}}
\left(\frac{1}{\underline{J}}\right)^{\max\{\frac{2q}{2-q},4q\}}(\sup\limits_{y\in{\mathbb R}}\varrho_0)^{-\min\{\frac{(3q-2)\alpha+q}{2-q},\frac{4q(2\alpha+1)}{q+1}\}}\nonumber
    \end{eqnarray}
holds for any given $\eta\in(0,1)$.

{\sl\textit{Step 6.}} Multiplying equation (\ref{1dlg-E15}) by $\varrho_0^{\alpha+2}\Theta_t\xi_r^3$ and integrating the resultant over $\mathbb R$, one has that
\begin{eqnarray}\label{5dlg-E23}
&&\frac1 p\frac d {dt}\int\frac{\varrho_0^{\alpha+2}}{J_0}|\frac{\Theta_y} J|^pJ\xi_r^3dy+\int\varrho_0^{\alpha+3}\Theta_t^2\xi_r^3dy\nonumber\\
&&=-\int|\frac{\Theta_y} J|^{p-2}\frac {\Theta_y}J(\frac{\varrho_0^{\alpha+2}}{J_0})_y\Theta_t\xi_r^3dy+\frac1 p\int \frac{\varrho_0^{\alpha+2}}{J_0}|\frac{\Theta_y} J|^pJ_t\xi_r^3dy-\int\frac{\varrho_0^{\alpha+2}}{J_0}|\frac{\Theta_y} J|^{p-2}\frac {1}{J^2}J_t{\Theta_y^2}\xi_r^3dy\nonumber\\
&&-\int R\varrho\frac{\varrho_0^{\alpha+2}}{J_0}v_y\Theta\Theta_t\xi_r^3dy+\int\left|\frac{v_y} {J}\right|^qJ\frac{\varrho_0^{\alpha+2}}{J_0}\Theta_t\xi_r^3dy-3\int\frac{\varrho_0^{\alpha+2}}{J_0}|\frac{\Theta_y} J|^{p-2}\frac {\Theta_y}J\Theta_t\xi_r^2\xi_r^{\prime}dy\nonumber\\
&&:=\sum_{i=1}^6{\mathbb S}_i.
   \end{eqnarray}
Now, we estimate each term on the right hand of (\ref{5dlg-E23}) as follows
\begin{eqnarray*}
&|{\mathbb S}_1|&=\Big|-\int|\frac{\Theta_y} J|^{p-2}\frac {\Theta_y}J(\frac{\varrho_0^{\alpha+2}}{J_0})_y\Theta_t \xi_r^3dy\Big|\\
&&\leqslant\Big(\int\rho_0^{\alpha+1}\Theta_y^2dy+\int\frac{\varrho_0^{\alpha+2}}{J_0}|\frac{\Theta_y} J|^pJdy\Big)^{\frac{2q}{q-1}} +\epsilon\int\varrho_0^{\alpha+3}\Theta_t^2\xi_r^3dy
+\eta\int\varrho_0^\alpha|\frac{\Theta_y} J|^{p-2}\frac {\Theta_{yy}^2}Jdy\\
&&+C(\eta,\epsilon)\left((\sup\limits_{y\in{\mathbb R}}\varrho_0)^{-\frac{(3p-2)(\alpha+1)}{2p}}\|(\varrho_0^{\alpha})_y\|_{L^\infty}\right)^{\frac{4q}{q+1}}\\
&&+C(\eta,\epsilon)\left((\sup\limits_{y\in{\mathbb R}}\varrho_0)^{-\frac{(p-1)(2\alpha+1)}{2+p}}\overline{J}^{\frac{(p-1)^2}{p+2}}\left(\frac{1}{\underline{J}}\right)^{p-1}
\|\varrho_0^{\frac{\alpha-1}{2}}J_{0y}\|_{L^2}\right)^{\frac{2q(p+2)}{5q-pq+p-1}},\\
&|{\mathbb S}_2|&=\Big|\frac1 p\int \frac{\varrho_0^{\alpha+2}}{J_0}|\frac{\Theta_y} J|^pJ_t\xi_r^3dy\Big|\\
&&\leqslant\left(\int\frac{\varrho_0^{\alpha+2}}{J_0}|\frac{\Theta_y} J|^pJdy+\int\varrho_0^{\alpha+1}v_y^2dy\right)^{\frac{2q}{q-1}}
+\eta\int\frac{\varrho_0^\alpha}{J_0}|\frac{v_y}{J}|^{q-2}\frac {v_{yy}^2}Jdy\\
&&+C(\eta)\left((\sup\limits_{y\in{\mathbb R}}\varrho_0)^{-\frac{2\alpha+1}{q+2}}\overline{J}^{\frac{q-1}{q+2}}\left(\frac{1}{\underline{J}}\right)\right)^{\frac{2q(q+2)}{q^2+3}},\\
&|{\mathbb S}_3|&=\Big|-\int\frac{\varrho_0^{\alpha+2}}{J_0}|\frac{\Theta_y} J|^{p-2}\frac {1}{J^2}J_t{\Theta_y^2}\xi_r^3dy\Big|\\
&&\leqslant\left(\int\frac{\varrho_0^{\alpha+2}}{J_0}|\frac{\Theta_y} J|^pJdy+\int\varrho_0^{\alpha+1}v_y^2dy\right)^{\frac{2q}{q-1}}
+\eta\int\frac{\varrho_0^\alpha}{J_0}|\frac{v_y}{J}|^{q-2}\frac {v_{yy}^2}Jdy\\
&&+C(\eta)\left((\sup\limits_{y\in{\mathbb R}}\varrho_0)^{-\frac{2\alpha+1}{q+2}}\overline{J}^{\frac{q-1}{q+2}}\left(\frac{1}{\underline{J}}\right)\right)^{\frac{2q(q+2)}{q^2+3}},\\
&|{\mathbb S}_4|&=\Big|-\int R\varrho\frac{\varrho_0^{\alpha+2}}{J_0}v_y\Theta\Theta_t\xi_r^3dy\Big|\\
&&\leqslant\left(\int\varrho_0^{\alpha+2}\Theta^2dy+\int\varrho_0^{\alpha+1}v_y^2dy\right)^{\frac{2q}{q-1}}
+\eta\int\frac{\varrho_0^\alpha}{J_0}|\frac{v_y}{J}|^{q-2}\frac {v_{yy}^2}Jdy+\epsilon\int\varrho_0^{\alpha+3}\Theta_t^2\xi_r^3dy\\
&&+C(\eta,\epsilon)\left((\sup\limits_{y\in{\mathbb R}}\varrho_0)^{-\frac{4\alpha-q}{q+2}}\Big(\frac1 {\underline J}\Big)\overline{J}^{\frac{2q+1}{q+2}}\right)^{\frac{4q(q+2)}{q^2-2q+3}},\\
&|{\mathbb S}_5|&=\Big|\int|\frac{v_y}J|^{q}J\frac{\varrho_0^{\alpha+2}}{J_0}\Theta_t\xi_r^3dy\Big|\\
&&\leqslant \left( \int\varrho_0^{\alpha+1}v_y^2dy\right)^{\frac{2q}{q-1}}+\eta\int\frac{\varrho_0^\alpha}{J_0}|\frac{v_y}{J}|^{q-2}\frac {v_{yy}^2}Jdy+\epsilon\int\varrho_0^{\alpha+3}\Theta_t^2\xi_r^3dy\\
&&+C(\eta,\epsilon)\left(\Big(\frac1 {\underline J}\Big)^{\frac{5-q}{q+2}}(\sup\limits_{y\in{\mathbb R}}\varrho_0)^{-\frac{(q-1)(2\alpha+1)}{q+2}}\right)^{\frac{4(q+2)}{11-5q}}\\
 &|{\mathbb S}_6|&=\Big|-3\int\frac{\varrho_0^{\alpha+2}}{J_0}|\frac{\Theta_y} J|^{p-2}\frac {\Theta_y}J\Theta_t\xi_r^2\xi_r^{\prime}dy\Big|\\
&&\leqslant\eta\int\varrho_0^{\alpha+3}\Theta_t^2\xi_r^3dy+\eta\int\frac{\varrho_0^{\alpha+1}}{J_0}|\frac{\Theta_y} J|^pJdy+C(\eta)\int\varrho_0^{\frac{2\alpha-p\alpha+2}{2-p}}|\xi_r^{\prime}|^{\frac{2p}{2-p}}dy
    \end{eqnarray*}
hold for any given $\eta\in(0,1).$ It is easy to get that
\begin{eqnarray}\label{5dlg-E24}
&&\frac d {dt}\int\frac{\varrho_0^{\alpha+2}}{J_0}|\frac{\Theta_y} J|^pJ\xi_r^3dy+\int\varrho_0^{\alpha+3}\Theta_t^2\xi_r^3dy\nonumber\\
&&\leqslant \Big(\int\rho_0^{\alpha+1}\Theta_y^2dy+\int\varrho_0^{\alpha+2}\Theta^2dy+\int\frac{\varrho_0^{\alpha+2}}{J_0}|\frac{\Theta_y} J|^pJdy +\int\varrho_0^{\alpha+1}v_y^2dy\Big)^{\frac{2q}{q-1}}\nonumber\\
&&+\eta\int\frac{\varrho_0^{\alpha+1}}{J_0}|\frac{\Theta_y} J|^pJdy+\eta\int\frac{\varrho_0^\alpha}{J_0}|\frac{v_y}{J}|^{q-2}\frac {v_{yy}^2}Jdy+\eta\int\varrho_0^\alpha|\frac{\Theta_y} J|^{p-2}\frac {\Theta_{yy}^2}Jdy\nonumber\\
&&+C(\eta)(\sup\limits_{y\in{\mathbb R}}\varrho_0)^{-8\alpha}\overline{J}^{12}\left(\frac{1}{\underline{J}}\right)^{12}\|(\varrho_0^{\alpha})_y\|_{L^\infty}^{12}
+C(\eta)\int\varrho_0^{\frac{2\alpha-p\alpha+2}{2-p}}|\xi_r^{\prime}|^{\frac{2p}{2-p}}dy
   \end{eqnarray}
holds for any fixed $\eta\in (0,1).$

{\sl\textit{Step 7.}} Differentiating (\ref{1dlg-E13}) by with respect to $y,$ multiplying the above differential equation by $\varrho_0^{\alpha}J_y\xi_r^2$ and integrating the resultant over $\mathbb R$, one deduces that
\begin{eqnarray}\label{5dlg-E25}
&&\frac1 2\frac d {dt}\int\varrho_0^{\alpha} J_y^2\xi_r^2dy=\int\varrho_0^{\alpha}{J_y}v_{yy}\xi_r^2dy\nonumber\\
&&\leqslant \Big(\int\varrho_0^{\alpha}J_y^2dy+\int\varrho_0^{\alpha+1}v_y^2dy\Big)^{\frac{2q}{q-1}}+\eta\int\frac{\varrho_0^\alpha}{J_0}|\frac{v_y}{J}|^{q-2}\frac {v_{yy}^2}Jdy\nonumber\\
&&+C(\eta)\left((\sup\limits_{y\in{\mathbb R}}\varrho_0)^{-\frac{(3-q)(2\alpha+1)}{2(q+2)}}{\overline J}^{\frac{6-q}{2(q+2)}}\right)^{\frac{q(q+2)}{q^2-3q+4}}
 \end{eqnarray}
holds for any given $\eta\in (0,1).$

Now, adding all obtained estimates in Step 1-Step 7, one arrives at
\begin{eqnarray}\label{5dlg-E26}
&&\frac d {dt}\int\varrho_0^{\alpha}\xi_r^2\Big({\varrho_0}v^2\xi_r+{\varrho_0}\Theta\xi_r+\varrho_0^2\Theta^2\xi_r+{\varrho_0}v_y^2
+{\varrho_0}\Theta_y^2+\frac{1}{J_0}|\frac{v_y}{J}|^qJ\xi_r+\frac{\varrho_0^2}{J_0}|\frac{\Theta_y} J|^pJ\xi_r+ J_y^2\Big)dy\nonumber\\
&&+\int\varrho_0^\alpha\xi_r^2\Big(\frac{1}{J_0}|\frac{v_y}{J}|^qJ\xi_r+\frac{\varrho_0}{J_0}|\frac{\Theta_y} J|^pJ\xi_r+\frac{1}{J_0}|\frac{v_y}{J}|^{q-2}\frac{ v_{yy}^2}J+\frac{1}{J_0}|\frac{\Theta_y} J|^{p-2}\frac {\Theta_{yy}^2}J+\varrho_0^3\Theta_t^2\xi_r+{\varrho_0}v_t^2\xi_r\Big)dy\nonumber\\
&&\leqslant\bigg(\int\varrho_0^{\alpha}\Big({\varrho_0}v^2\xi_r+{\varrho_0}\Theta\xi_r+\varrho_0^2\Theta^2\xi_r+{\varrho_0}v_y^2
+{\varrho_0}\Theta_y^2+\frac{1}{J_0}|\frac{v_y}{J}|^qJ\xi_r+\frac{\varrho_0^2}{J_0}|\frac{\Theta_y} J|^pJ\xi_r+ J_y^2\Big)dy\bigg)^{\frac{2q}{q-1}}\nonumber\\
&&+\eta\int\varrho_0^\alpha\Big(\frac{1}{J_0}|\frac{v_y}{J}|^qJ\xi_r+\frac{\varrho_0}{J_0}|\frac{\Theta_y} J|^pJ\xi_r+\frac{1}{J_0}|\frac{v_y}{J}|^{q-2}\frac{ v_{yy}^2}J+\frac{1}{J_0}|\frac{\Theta_y} J|^{p-2}\frac {\Theta_{yy}^2}J+\varrho_0^3\Theta_t^2\xi_r+{\varrho_0}v_t^2\xi_r\Big)dy\nonumber\\
&&+C(\eta)\Big(\int\varrho_0^\alpha|\xi_r^{\prime}|^2dy\Big)^\frac{2q(q+2)}{5-q}
\|\xi_r^\prime\|_{L^\infty}^\frac{4q(q+2)}{5-q}\|\varrho_0^{-1}\xi_r^\prime\|_{L^\infty}^\frac{4q(q+2)}{q+1}{\overline J}^{12}\Big(\frac1{\underline J}\Big)^{\frac{12q}{2-q}}(\sup\limits_{y\in{\mathbb R}}\varrho_0)^{-\min\{\frac{(3q-2)\alpha+q}{2-q},\frac{10\alpha}{q-1}\}}
\nonumber\\
&&+C(\eta)\left(\int\left(\left(\varrho_0^{\frac{(2-p)\alpha-4(p-1)}{2-p}}+\varrho_0^{\frac{2\alpha-p\alpha+2}{2-p}}\right)|\xi_r^\prime|^\frac{2p}{2-p}+\varrho_0^{-\frac{3\alpha q-q-2\alpha}{2-q}}|\xi_r^{\prime}|^{\frac{2q}{2-q}}\right)dy\right)
 \end{eqnarray}
holds for any given $\eta\in (0,1).$ Note that $\varrho_0(y)\geqslant \frac {A_0} {(1+|y|)^{l}}~(\forall y \in \mathbb R)$
and
$$|\xi_r^{\prime}(y)|\leqslant\frac{\|\xi^{\prime}\|_{L^\infty}}{|y|}\leqslant{\frac{4\|\xi^{\prime}\|_{L^\infty}}{A_0^{\frac{1}{l}}}}\varrho_0(y)^{\frac{1}{l}}\quad(\forall |y|\geqslant1).
$$
for some $l>1.$ Letting  $r\to\infty$ and taking small enough $\eta,$ one deduces from (\ref{5dlg-E26}) that
\begin{eqnarray}\label{5dlg-E26-1}
&&\frac d {dt}\int\varrho_0^{\alpha}\Big({\varrho_0}v^2+{\varrho_0}\Theta+\varrho_0^2\Theta^2+{\varrho_0}v_y^2
+{\varrho_0}\Theta_y^2+\frac{1}{J_0}|\frac{v_y}{J}|^qJ+\frac{\varrho_0^2}{J_0}|\frac{\Theta_y} J|^pJ+ J_y^2\Big)
dy\nonumber\\
&&+\int\varrho_0^\alpha\Big(|\frac{v_y}{J}|^qJ+{\varrho_0}|\frac{\Theta_y} J|^pJ+|\frac{v_y}{J}|^{q-2}\frac{ v_{yy}^2}J+|\frac{\Theta_y} J|^{p-2}\frac {\Theta_{yy}^2}J+\varrho_0^3\Theta_t^2+{\varrho_0}v_t^2\Big)dy\nonumber\\
&&\leqslant\bigg(\int\varrho_0^{\alpha}\Big({\varrho_0}v^2+{\varrho_0}\Theta+\varrho_0^2\Theta^2+{\varrho_0}v_y^2
+{\varrho_0}\Theta_y^2+\frac{1}{J_0}|\frac{v_y}{J}|^qJ+\frac{\varrho_0^2}{J_0}|\frac{\Theta_y} J|^pJ+ J_y^2\Big)
dy\bigg)^{\frac{2q}{q-1}}\nonumber\\
&&+C{\overline J}^{12}\Big(\frac1{\underline J}\Big)^{\frac{12q}{2-q}}(\sup\limits_{y\in{\mathbb R}}\varrho_0)^{-\min\{\frac{(3q-2)\alpha+q}{2-q},\frac{10\alpha}{q-1}\}}
\left(\|(\varrho_0^\alpha)_y\|_{L^p}\|(\varrho_0^\alpha)_y\|_{L^\infty}\right)^{\frac{16q}{2-q}}, \end{eqnarray}
where the constant $l$ is selected such that
\begin{equation}\label{5dlg-E27}
0<l<\min\{1,-\frac{3p-2}{(2-p)\alpha+3}\}.
    \end{equation}
Thus, (\ref{5dlg-E11}) follows from (\ref{5dlg-E26-1}) directly according to (\ref{5dlg-E7}) and (\ref{5dlg-E8}).
\end{proof}

Now, we are in the position to prove Theorem \ref{1dlg-thm2}.

\textbf{Proof of Theorem \ref{1dlg-thm2}} According to Theorem \ref{1dlg-thm1}, the problem (\ref{1dlg-E16})-(\ref{1dlg-E17}) admits a strong solution $(J,\varrho,v,\Theta)$ on $\mathbb R \times [0,\widetilde{T_0}]$ for some positive time $\widetilde{T_0}$. In accordance with Theorem \ref{1dlg-thm1} and Proposition \ref{5dlg-P5.2}, this local strong solution can be extended to the maximal time of existence $T_1$, which is given by (\ref{5dlg-E31-1}). So, the quaternion $(J,\varrho,v,\Theta)$ is a strong solution to the problem (\ref{1dlg-E16})-(\ref{1dlg-E17}) on $\mathbb R \times [0,T_1]$. Set $\delta_1(t)=\inf\limits_{y\in{\mathbb R}}J(y,t)$ on $[0,T_1]$. With the help of Proposition \ref{5dlg-P5.2} and Lemma \ref{2dlg-L2.2}, one finds that
\begin{eqnarray}\label{5dlg-E29}
&&\bigg(\int\varrho_0^{\alpha}\Big({\varrho_0}v^2+{\varrho_0}\Theta+\varrho_0^2\Theta^2+{\varrho_0}v_y^2
+{\varrho_0}\Theta_y^2+\frac{1}{J_0}|\frac{v_y}{J}|^qJ+\frac{\varrho_0^2}{J_0}|\frac{\Theta_y} J|^pJ+ J_y^2\Big)dy\bigg)(t)\nonumber\\
&&\leqslant CH_1(t)\bigg(1-\frac{q+1}{q-1}H_1^{\frac{q+1}{q-1}}(t)t\bigg)^{-\frac{q-1}{q+1}},
    \end{eqnarray}
\begin{eqnarray}\label{5dlg-E30}
&&\sup\limits_{s\in[0,t]}\bigg(\int\varrho_0^{\alpha}\Big({\varrho_0}v^2+{\varrho_0}\Theta+\varrho_0^2\Theta^2+{\varrho_0}v_y^2
+{\varrho_0}\Theta_y^2+\frac{1}{J_0}|\frac{v_y}{J}|^qJ+\frac{\varrho_0^2}{J_0}|\frac{\Theta_y} J|^pJ+ J_y^2\Big)dy\bigg)(s)\nonumber\\
&&+\int_0^t\int\varrho_0^\alpha\Big(\frac{1}{J_0}|\frac{v_y}{J}|^qJ+\frac{\varrho_0}{J_0}|\frac{\Theta_y} J|^pJ+\frac{1}{J_0}|\frac{v_y}{J}|^{q-2}\frac{ v_{yy}^2}J+\frac{1}{J_0}|\frac{\Theta_y} J|^{p-2}\frac {\Theta_{yy}^2}J+\varrho_0^3\Theta_t^2+{\varrho_0}v_t^2\Big)dyds\nonumber\\
&&\leqslant C\bigg(H_1(t)+\int_0^tH_1^{\frac{2q} {q-1}}(s)\bigg(1-\frac{q+1}{q-1}H_1^{\frac{q+1}{q-1}}(s)s\bigg)^{-\frac{2q}{q+1}}ds\bigg)
    \end{eqnarray}
hold for any $t\in[0,{T_1}]$, where
\begin{eqnarray*}
&&H_1(t)=\int\varrho_0^{\alpha}\Big({\varrho_0}v_0^2+{\varrho_0}\Theta_0+\varrho_0^{2}\Theta_0^2+{\varrho_0}v_{0y}^2
+{\varrho_0}\Theta_{0y}^2+|\frac{v_{0y}} {J_0}|^q+\varrho_0^{2}|\frac{\Theta_{0y}} {J_0}|^p+J_{0y}^2\Big)dy\\
&&+C({\sup\limits_{y\in\mathbb R}\varrho_0})^{-\min\{\frac{(3q-2)\alpha+q}{2-q},\frac{10\alpha}{q-1}\}}
\left(\|(\varrho_0^\alpha)_y\|_{L^p}\|(\varrho_0^\alpha)_y\|_{L^\infty}\right)^{\frac{16q}{2-q}}\int_0^t\left(\delta(s)\right)^{-\frac{24}{2-q}}ds
    \end{eqnarray*}
holds for any $t\in(0,T_1]$, where the positive constant $C$ is dependent on $p$, $q$ and $A_0$. In fact, the time $T_1$ satisfies the condition
\begin{equation}\label{5dlg-E31-1}
\frac{q+1}{q-1}H_1^{\frac{q+1}{q-1}}({T_1}){T_1}<1.
    \end{equation}
Furthermore, one takes a similar argument in Proposition \ref{3dlg-P3.3} to arrive that
\begin{eqnarray*}
&&\sup\limits_{y\in\mathbb R}J(y,t)\leqslant2M_1(t)\exp\bigg({(\sup\limits_{y\in\mathbb R}\varrho_0)}^{-\frac{\alpha}{q}}\frac{(q-1)t}{q}\bigg),\\
&&\inf\limits_{y\in\mathbb R}J(y,t)\geqslant(\inf\limits_{y\in\mathbb R}J_0)\bigg(1+C(\inf\limits_{y\in\mathbb R}J_0)\Big[q(M_1(t)-\sup\limits_{y\in\mathbb R}J_0)\Big]^{\frac1 q}\Big(\int_0^tF_1(s)ds\Big)^{\frac{q-1} q}\bigg)^{-1},
    \end{eqnarray*}
on $[0,T_1]$, where
\begin{eqnarray*}
&&M_1(t)=\frac{C}{q}\bigg(q\sup\limits_{y\in\mathbb R}J_0+H_1(t)+\int_0^tH_1^{\frac{2q}{q-1}}(s)\Big(1-\frac{q+1}{q-1}H_1^{\frac{q+1}{q-1}}(s)s\Big)^{-\frac{2q}{q-1}}ds\bigg),\\
&&F_1(t)={(\sup\limits_{y\in\mathbb R}\varrho_0)}^{-\frac{\alpha}{q-1}}M_1^{-\frac{q+1}{q-1}}(t)\exp\bigg(-\frac{q+1}{q}{(\sup\limits_{y\in\mathbb R}\varrho_0)}^{-\frac{\alpha}{q}}t\bigg).
    \end{eqnarray*}
Moreover, the initial date $J_0$ has lower bound and upper bound. It is deduced from (\ref{5dlg-E7})-(\ref{5dlg-E10}) that
\begin{eqnarray*}
&&\inf\limits_{y\in(-r,r)}\varrho(y,t)>C\delta_1(t)\ (\forall r\in(0,\infty)),\quad\varrho(y,t)\leqslant C\delta_1^{-1}(t)~(\forall y\in\mathbb R),\\
&&\inf\limits_{y\in\mathbb R}J(y,t)>C\delta_1(t),\quad J(y,t)\leqslant C\delta_1^{-1}(t)~(\forall y\in\mathbb R),\\
&&\varrho(y,t)\geqslant\frac{C\delta_1(t)}{(1+|y|)^{l}}~~(\forall y\in\mathbb R)
    \end{eqnarray*}
hold for any $t\in[0,T_1]$, where $C$ is a positive constant depending on $p, q, H_0$ and $A_0$ with
\begin{equation*}
H_0=\int\varrho_0^{\alpha}\Big({\varrho_0}v_0^2+{\varrho_0}\Theta_0+\varrho_0^{2}\Theta_0^2+{\varrho_0}v_{0y}^2
+{\varrho_0}\Theta_{0y}^2+|\frac{v_{0y}} {J_0}|^q+\varrho_0^{2}|\frac{\Theta_{0y}} {J_0}|^p+J_{0y}^2\Big)dy.
    \end{equation*}
Note that
\begin{equation}
\frac{1}{\inf\limits_{y\in\mathbb R}J(y,t)}+\sup\limits_{y\in\mathbb R}J(y,t)+\|J^{\prime}(\cdot,t)\|_{L^2}+\|v(\cdot,t)\|_{W^{1,q}}+\|\Theta(\cdot,t)\|_{W^{1,p}}
\notag
\end{equation}
is finite for any given $t\in[0,T_1],$
\begin{eqnarray*}
&&\int((\varrho^\alpha)_y)^pdy\leqslant C\delta_1^{-p\alpha}(t)\Big(\int\varrho_0^\alpha J_y^2dy+\int((\varrho_0^\alpha)_y)^pdy\Big),\\
&&\|(\varrho^\alpha)_y\|_{L^\infty}\leqslant C\delta_1^{-\alpha}(t)\|(\varrho_0^\alpha)_y\|_{L^\infty}
    \end{eqnarray*}
hold for any $t\in(0,T_1]$. Set
\begin{eqnarray*}
J_1(y)=J(y,T_1),\quad\varrho_1(y)=\varrho(y,T_1),\quad v_1(y)=v(y,T_1),\quad\Theta_1(y)=\Theta(y,T_1).
    \end{eqnarray*}
Obviously, $(J,\varrho,v,\Theta)|_{t=T_1}=(J_1,\varrho_1,v_1,\Theta_1)$ fulfills the conditions on the initial data stated in Theorem \ref{1dlg-thm1}. Taking $T_1$ as the initial time, one finds that there exists a strong solution $(J,\varrho,v,\Theta)$ to the system (\ref{1dlg-E16})-(\ref{1dlg-E17}) on $\mathbb R\times[T_1,T_2]$ with the help of Theorem \ref{1dlg-thm1} and Proposition \ref{5dlg-P5.2}, where $T_2$ is given in (\ref{5dlg-E33}). Also, one obtains that
\begin{eqnarray}\label{5dlg-E32}
&&\frac d {dt}\bigg(\int\varrho_1^\alpha\Big({\varrho_1}v^2+{\varrho_1}\Theta
+\varrho_1^2\Theta^2+{\varrho_1}v_y^2+{\varrho_1}\Theta_y^2+\frac{1}{J_1}|\frac{v_y}{J}|^qJ
+\frac{\varrho_1^2}{J_1}|\frac{\Theta_y} J|^pJ+J_y^2\Big)dy\bigg)(t)\nonumber\\
&&+\int\varrho_1^\alpha\Big(\frac{1}{J_1}|\frac{v_y}{J}|^qJ+\frac{\varrho_1}{J_1}|\frac{\Theta_y} J|^pJ+\frac{1}{J_1}|\frac{v_y}{J}|^{q-2}\frac { v_{yy}^2}J+\frac{1}{J_1}|\frac{\Theta_y} J|^{p-2}\frac {\Theta_{yy}^2}J+{\varrho_1}v_t^2+\varrho_1^3\Theta_t^2\Big)dy\nonumber\\
&&\leqslant\bigg(\int\varrho_1^\alpha\Big({\varrho_1}v^2+{\varrho_1}\Theta
+\varrho_1^2\Theta^2+{\varrho_1}v_y^2+{\varrho_1}\Theta_y^2+\frac{1}{J_1}|\frac{v_y}{J}|^qJ
+\frac{\varrho_1^2}{J_1}|\frac{\Theta_y} J|^pJ+J_y^2\Big)dy\bigg)^{\frac{2q}{q-1}}\nonumber\\
&&+C\left(\delta_2(t)\right)^{-\frac{24}{2-q}}({\sup\limits_{y\in\mathbb R}{\varrho_1}})^{-\min\{\frac{(3q-2)\alpha+q}{2-q},\frac{10\alpha}{q-1}\}}
\left(\|(\varrho_1^\alpha)_y\|_{L^p}\|(\varrho_1^\alpha)_y\|_{L^\infty}\right)^{\frac{16q}{2-q}}
    \end{eqnarray}
holds for any $t\in(T_1,T_2]$, where $\delta_2(t)=\inf\limits_{y\in{\mathbb R}}J(y,t)$ on $[T_1,T_2],$
\begin{eqnarray*}
&&M_2(t)=\frac{C}{q}\bigg(q\sup\limits_{y\in\mathbb R}J_1+H_2(t)+\int_{T_1}^tH_2^{\frac{2q}{q-1}}(s)\Big(1-\frac{q+1}{q-1}H_2^{\frac{q+1}{q-1}}(s)s\Big)^{-\frac{2q}{q-1}}ds\bigg),\\
&&F_2(t)={(\sup\limits_{y\in\mathbb R}\varrho_1)}^{-\frac{\alpha}{q-1}}M_2^{-\frac{q+1}{q-1}}(t)\exp\bigg(-\frac{q+1}{q}{(\sup\limits_{y\in\mathbb R}\varrho_1)}^{-\frac{\alpha}{q}}t\bigg),\\
&&\sup\limits_{y\in\mathbb R}J(y,t)\leqslant2M_2(t)\exp\bigg({(\sup\limits_{y\in\mathbb R}\varrho_1)}^{-\frac{\alpha}{q}}\frac{(q-1)t}{q}\bigg),\\
&&H_2(t)=\int\varrho_1^\alpha\Big({\varrho_1}v_1^2+{\varrho_1}\Theta_1+\varrho_1^2\Theta_1^2
+\varrho_1v_{1y}^2+{\varrho_1}\Theta_{1y}^2+|\frac{v_{1y}} {J_1}|^q+\varrho_1^2|\frac{\Theta_{1y}} {J_1}|^p+J_{1y}^2\Big)dy\\
&&~~~~~~~~~~~+C({\sup\limits_{y\in\mathbb R}{\varrho_1}})^{-\min\{\frac{(3q-2)\alpha+q}{2-q},\frac{10\alpha}{q-1}\}}
\left(\|(\varrho_1^\alpha)_y\|_{L^p}\|(\varrho_1^\alpha)_y\|_{L^\infty}\right)^{\frac{16q}{2-q}}
\int^t_{T_1}\left(\delta_2(s)\right)^{-\frac{24}{2-q}}ds
    \end{eqnarray*}
on $[T_1,T_2]$. In fact, the functions $H_2(t), M_2(t)$ and $\delta_2^{-1}(t)$ are positive and non-decreasing on the interval $[T_1,T_2].$  According to Lemma \ref{2dlg-L2.2} and (\ref{5dlg-E32}), we obtain that
\begin{eqnarray*}
&&\bigg(\int\varrho_1^\alpha\Big({\varrho_1}v^2+{\varrho_1}\Theta
+\varrho_1^2\Theta^2+{\varrho_1}v_y^2+{\varrho_1}\Theta_y^2+\frac{1}{J_1}|\frac{v_y}{J}|^qJ
+\frac{\varrho_1^2}{J_1}|\frac{\Theta_y} J|^pJ+J_y^2\Big)dy\bigg)(t)\\
&&\leqslant CH_2(t)\bigg(1-\frac{q+1}{q-1}H_2^{\frac{q+1}{q-1}}(t)(t-{T_1})\bigg)^{-\frac{q-1}{q+1}}.
    \end{eqnarray*}
Furthermore,
\begin{eqnarray*}
&&\sup\limits_{s\in[T_1,t]}\bigg(\int\varrho_1^\alpha\Big({\varrho_1}v^2+{\varrho_1}\Theta
+\varrho_1^2\Theta^2+{\varrho_1}v_y^2+{\varrho_1}\Theta_y^2+\frac{1}{J_1}|\frac{v_y}{J}|^qJ
+\frac{\varrho_1^2}{J_1}|\frac{\Theta_y} J|^pJ+J_y^2\Big)dy\bigg)(s)\\
&&+\int_{T_1}^t\int\varrho_1^\alpha\Big(\frac{1}{J_1}|\frac{v_y}{J}|^qJ+\frac{\varrho_1}{J_1}|\frac{\Theta_y} J|^pJ+\frac{1}{J_1}|\frac{v_y}{J}|^{q-2}\frac { v_{yy}^2}J+\frac{1}{J_1}|\frac{\Theta_y} J|^{p-2}\frac {\Theta_{yy}^2}J+{\varrho_1}v_t^2+\varrho_1^3\Theta_t^2\Big)dyds\\
&&\leqslant C\bigg(H_2(t)+\int_{T_1}^tH_2^{\frac{2q} {q-1}}(s)\bigg(1-\frac{q+1}{q-1}H_2^{\frac{q+1}{q-1}}(s)(s-T_1)\bigg)^{-\frac{2q}{q+1}}ds\bigg)
    \end{eqnarray*}
holds for all $t\in[{T_1},{T_2}]$, where positive constant $T_2$ satisfies
\begin{equation}\label{5dlg-E33}
\frac{q+1}{q-1}H_2^{\frac{q+1}{q-1}}({T_2})(T_2-{T_1})<1
    \end{equation}
by Lemma \ref{2dlg-L2.2} and (\ref{5dlg-E32}). Under the definition of $H_2(t)$, there
exists some $t_2 > T_1$ (in fact $t_2 > T_2$) such that
\begin{equation}
\lim_{t\to t_2-}\frac{q+1}{q-1}H_2^{\frac{q+1}{q-1}}({t})(t-{T_1})=1
\notag
    \end{equation}
Now, we concentrate on the selection of $T_2$. One can get from (\ref{5dlg-E33}) that
\begin{equation}
\begin{aligned}
\delta_2(t)\geqslant&(\inf\limits_{y\in\mathbb R}J_1)\bigg[1+C(\inf\limits_{y\in\mathbb R}J_1)\Big[H_2(t)+H_2^{\frac{2q} {q-1}}(t)\Big(1-\frac{q+1}{q-1}H_2^{\frac{q+1}{q-1}}(t)(t-T_1)\Big)^{-\frac{2q}{q+1}}(t-T_1)\Big]^{\frac1 q}\\
&\cdot(t-T_1)^{\frac{q-1} q}\bigg]^{-1}.
\notag
\end{aligned}
    \end{equation}
The function
\begin{equation*}
\Big[H_2(t)+H_2^{\frac{2q} {q-1}}(t)\Big(1-\frac{q+1}{q-1}H_2^{\frac{q+1}{q-1}}(t)(t-T_1)\Big)^{-\frac{2q}{q+1}}(t-T_1)\Big]^{\frac1 q}(t-T_1)^{\frac{q-1} q}
    \end{equation*}
is strictly increasing on $[T_1,t_2)$. Assume that
\begin{equation}\label{5dlg-E34}
\lim_{t\to t_2-}\Big[H_2(t)+H_2^{\frac{2q} {q-1}}(t)\Big(1-\frac{q+1}{q-1}H_2^{\frac{q+1}{q-1}}(t)(t-T_1)\Big)^{-\frac{2q}{q+1}}(t-T_1)\Big]^{\frac1 q}(t-T_1)^{\frac{q-1} q}=\eta_0
    \end{equation}
for some positive constant $\eta_0$. Since
\begin{equation}
\begin{aligned}
&\lim_{t\to t_2-}\frac{q+1}{q-1}H_2^{\frac{q+1}{q-1}}(t)(t-T_1)=1,\\
&\lim_{t\to t_2-}\Big(1-\frac{q+1}{q-1}H_2^{\frac{q+1}{q-1}}(t)(t-T_1)\Big)^{-\frac{2q}{q+1}}=+\infty,
\notag
\end{aligned}
    \end{equation}
one deduces from (\ref{5dlg-E34}) that
\begin{equation}
\lim_{t\to t_2-}H_2(t)=0.
\notag
    \end{equation}
However,
\begin{equation*}
\lim_{t\to t_2-}\frac{q+1}{q-1}H_2^{\frac{q+1}{q-1}}(t)(t-T_1)=1.
    \end{equation*}
This contradiction implies that
\begin{equation}\label{5dlg-E35}
\lim_{t\to t_2-}\Big[H_2(t)+H_2^{\frac{2q} {q-1}}(t)\Big(1-\frac{q+1}{q-1}H_2^{\frac{q+1}{q-1}}(t)(t-T_1)\Big)^{-\frac{2q}{q+1}}(t-T_1)\Big]^{\frac1 q}(t-T_1)^{\frac{q-1} q}=+\infty.
    \end{equation}
Setting
\begin{equation}\label{5dlg-E36}
1-\frac{q+1}{q-1}H_2^{\frac{q+1}{q-1}}(t)(t-T_1)=\varepsilon,
    \end{equation}
one obtains from (\ref{5dlg-E36}) that
\begin{eqnarray*}
&&\lim_{\varepsilon\to0+}t(\varepsilon)=t_2{\small -},\quad\lim_{\varepsilon\to1-}t(\varepsilon)=T_1{\small +},\\
&&\lim_{\varepsilon\to0+}(t(\varepsilon)-T_1)^{\frac{q-1}{q+1}}\bigg(\Big(\frac{q-1}{q+1}\Big)^{\frac{q-1}{q+1}}(1-\varepsilon)^{\frac{q-1}{q+1}}+\Big(\frac{q-1}{q+1}\Big)^{\frac{2q}{q+1}}(1-\varepsilon)^{\frac{2q}{q+1}}\varepsilon^{-\frac{2q}{q-1}}\bigg)^{\frac1 q}=+\infty,\\
&&\lim_{\varepsilon\to1-}(t(\varepsilon)-T_1)^{\frac{q-1}{q+1}}\bigg(\Big(\frac{q-1}{q+1}\Big)^{\frac{q-1}{q+1}}(1-\varepsilon)^{\frac{q-1}{q+1}}+\Big(\frac{q-1}{q+1}\Big)^{\frac{2q}{q+1}}(1-\varepsilon)^{\frac{2q}{q+1}}\varepsilon^{-\frac{2q}{q-1}}\bigg)^{\frac1 q}=0,\\
&&\lim_{\varepsilon\to0+}\bigg(\Big(\frac{q-1}{q+1}\Big)^{\frac{q-1}{q+1}}(1-\varepsilon)^{\frac{q-1}{q+1}}+\Big(\frac{q-1}{q+1}\Big)^{\frac{2q}{q+1}}(1-\varepsilon)^{\frac{2q}{q+1}}\varepsilon^{-\frac{2q}{q-1}}\bigg)^{\frac1 q}=+\infty
    \end{eqnarray*}
and the functions $t(\varepsilon)$ and $\Big(\frac{q-1}{q+1}\Big)^{\frac{q-1}{q+1}}(1-\varepsilon)^{\frac{q-1}{q+1}}+\Big(\frac{q-1}{q+1}\Big)^{\frac{2q}{q+1}}(1-\varepsilon)^{\frac{2q}{q+1}}\varepsilon^{-\frac{2q}{q-1}}$ are strictly decreasing on $(0, 1).$

Next, we determine the maximal time step $t(\varepsilon_0)-T_1$ for some $\varepsilon_0\in(0,1)$, which is the key to prove Theorem \ref{1dlg-thm2}. For the convenience of discussion, we replace $\varepsilon$ by $\frac1 k$ and find that
\begin{eqnarray}
&&\lim_{\frac1 k\to0+}\bigg(t(\frac1 k)-T_1\bigg)^{\frac{q-1}{q+1}}\bigg(\Big(\frac{q-1}{q+1}\Big)^{\frac{q-1}{q+1}}\Big(1-\frac1 k\Big)^{\frac{q-1}{q+1}}+\Big(\frac{q-1}{q+1}\Big)^{\frac{2q}{q+1}}\Big(1-\frac1 k\Big)^{\frac{2q}{q+1}}\Big(\frac1 k\Big)^{-\frac{2q}{q-1}}\bigg)^{\frac1 q}\nonumber\\
&&~~~~~~~=+\infty,\label{5dlg-E37}\\
&&\lim_{\frac1 k\to0+}\bigg(\Big(\frac{q-1}{q+1}\Big)^{\frac{q-1}{q+1}}\Big(1-\frac1 k\Big)^{\frac{q-1}{q+1}}+\Big(\frac{q-1}{q+1}\Big)^{\frac{2q}{q+1}}\Big(1-\frac1 k\Big)^{\frac{2q}{q+1}}\Big(\frac1 k\Big)^{-\frac{2q}{q-1}}\bigg)^{\frac1 q}=+\infty.\label{5dlg-E38}
    \end{eqnarray}
Introduce the function
\begin{equation*}
g(k)\Big(\frac{q-1}{q+1}\Big)^{\frac{q-1}{q(q+1)}}\bigg(\Big(1-\frac1 k\Big)^{\frac{q-1}{q+1}}+\frac{q-1}{q+1}\Big(1-\frac1 k\Big)^{\frac{2q}{q+1}}\Big(\frac1 k\Big)^{-\frac{2q}{q-1}}\bigg)^{\frac1 q}.
    \end{equation*}
So,
\begin{equation}\label{5dlg-E39}
g(k)=\Big(\frac{q-1}{q+1}\Big)^{\frac{q-1}{q(q+1)}}\Big(1-\frac1 k\Big)^{\frac{q-1}{q(q+1)}}\bigg(1-\frac{q-1}{q+1}k^{\frac{q+1}{q-1}}(k-1)\bigg)^{\frac1 q}.
\end{equation}
Clearly, $\lim\limits_{k\to\infty}g(k)=+\infty$, $g(k)$ is increasing on $[1,+\infty),$ and (\ref{5dlg-E37}) is rewritten as
\begin{equation}\label{5dlg-E40}
\lim_{k\to\infty}\bigg(t(\frac1 k)-T_1\bigg)^{\frac{q-1}{q+1}}g(k)=+\infty.
    \end{equation}
For any fixed $k_1\in\mathbb Z^+$, it is deduced from (\ref{5dlg-E40}) that there exists $k_2\in\mathbb Z^+$ such that
\begin{equation}\label{5dlg-E41}
\bigg(t(\frac1 k)-T_1\bigg)^{\frac{q-1}{q+1}}g(k)>g(k_1),
    \end{equation}
when $k>k_1$. Especially, (\ref{5dlg-E41}) holds when $k>\max\{k_1,k_2\}$. Representing $k$ as $\eta^{\frac {(q-1)^2}{2(q+1)}} k_1$ with $\eta>1$, one finds that
\begin{equation}
\bigg(t(\frac{1} {\eta^{\frac {(q-1)^2}{2(q+1)}}k_1})-T_1\bigg)^{\frac{q-1}{q+1}}g(\eta^{\frac {(q-1)^2}{2(q+1)}} k_1)>g(k_1)
\notag
    \end{equation}
holds for any $\eta^{\frac {(q-1)^2}{2(q+1)}}>\eta_0^{\frac {(q-1)^2}{2(q+1)}}$, where $\eta_0^{\frac {(q-1)^2}{2(q+1)}}=\frac{\max\{k_1,k_2\}}{k_1}$.
That is, for any fixed $k_1\in\mathbb Z^+$, there exists $\eta_0>1$ such that
\begin{equation}\label{5dlg-E42}
\bigg(t(\frac{1} {\eta^{\frac {(q-1)^2}{2(q+1)}} k_1})-T_1\bigg)^{\frac{q-1}{q+1}}>\frac{g(k_1)}{g(\eta^{\frac {(q-1)^2}{2(q+1)}} k_1)}
    \end{equation}
holds for any $\eta>\eta_0$. Consider the right hand of (\ref{5dlg-E42}) and set
\begin{equation}\label{5dlg-E43}
\begin{aligned}
h(k,\eta)&\triangleq\frac{g(k)}{g(\eta^{\frac {(q-1)^2}{2(q+1)}} k)}\\
&=\Bigg(\frac{(k-1)^{\frac{q-1}{q+1}}+\frac{q-1}{q+1}k^{\frac{q+1}{q-1}}(k-1)^{\frac{2q}{q+1}}}{(k-\frac{1}{\eta^{\frac {(q-1)^2}{2(q+1)}}})^{\frac{q-1}{q+1}}+\frac{q-1}{q+1}(\eta^{\frac {(q-1)^2}{2(q+1)}})^{\frac{4q}{(q-1)(q+1)}}k^{\frac{q-1}{q+1}}(\eta^{{\frac {(q-1)^2}{2(q+1)}}}k-1)^{\frac{2q}{q+1}}}\Bigg)^{\frac1 q}.
\end{aligned}
    \end{equation}
Obviously, for any fixed $k\in\mathbb Z^+$, $h(k,\eta)<1$ for any $\eta>1$, $h(k,\eta)$ is continuous on $[1, +\infty)$ with respect to $\eta$, $h(k, 1)=1$ and $\lim\limits_{\eta\to\infty}h(k,\eta)=0$.

Due to (\ref{5dlg-E42}), there exists $2\widetilde\eta_2^{\frac {(q-1)^2}{2(q+1)}}\in\mathbb Z^+$ for some $\widetilde{\eta_2}>1$ such that
\begin{equation}\label{5dlg-E44}
\bigg(t\left(\frac{1} {2\widetilde\eta_2^{\frac {(q-1)^2}{2(q+1)}}}\right)-T_1\bigg)^{\frac{q-1}{q+1}}g(2\widetilde\eta_2^{\frac {(q-1)^2}{2(q+1)}})>g(2),
    \end{equation}
which implies that
\begin{equation}
t\left(\frac{1} {2\widetilde\eta_2^{\frac {(q-1)^2}{2(q+1)}}}\right)-T_1>g^{\frac{q+1}{q-1}}(2)g^{-\frac{q+1}{q-1}}(2\widetilde\eta_2^{\frac {(q-1)^2}{2(q+1)}})=h^{\frac{q+1}{q-1}}(2,\widetilde\eta_2).
\notag
    \end{equation}
The properties of $h(2,\eta)$, imply that there exists a $\eta_2>1$ such that
\begin{equation}
h^{\frac{q+1}{q-1}}(2,\eta_2)<h^{\frac{q+1}{q-1}}(2,\widetilde\eta_2).
\notag
    \end{equation}
In fact, $\eta_2>\widetilde\eta_2$, and $t\left(\frac{1} {2\eta_2^{\frac {(q-1)^2}{2(q+1)}}}\right)-T_1>t\left(\frac{1} {2\widetilde\eta_2^{\frac {(q-1)^2}{2(q+1)}}}\right)-T_1$. Therefore, $T_2$ can be regarded as
$$T_2=t\left(\frac{1} {2\eta_2^{\frac {(q-1)^2}{2(q+1)}}}\right)\mbox{ and }T_2-T_1>h^{\frac{q+1}{q-1}}(2,\eta_2)\mbox{ for some } \eta_2>1.$$
Using the iterative approach, we set
\begin{equation}
J_{l-1}(y)=J(y,T_{l-1}),\ \varrho_{l-1}(y)=\varrho(y,T_{l-1}),\  v_{l-1}(y)=v(y,T_{l-1}),\ \Theta_{l-1}(y)=\Theta(y,T_{l-1})
\notag
    \end{equation}
for $l\geqslant3$. Clearly, $(J,\varrho,v,\Theta)|_{t=T_{l-1}}=(\ J_{l-1},\varrho_{l-1}, v_{l-1},\Theta_{l-1})$ fulfills the conditions on the initial data stated in Theorem \ref{1dlg-thm1}. Taking $T_{l-1}$ as the initial time, one gets that there exists a strong solution $(J,\varrho,v,\Theta)$ to the system (\ref{1dlg-E16})-(\ref{1dlg-E17})  on $\mathbb R\times[T_{l-1},T_l]$ with the help of Theorem \ref{1dlg-thm1} and Proposition \ref{5dlg-P5.2}, where $T_l$ is given in (\ref{5dlg-E45}). Moreover, one obtains that
\begin{eqnarray*}
&&\bigg(\int\varrho_{l-1}^{\alpha}\left({\varrho_{l-1}}v^2+{\varrho_{l-1}}\Theta+\varrho_{l-1}^2\Theta^2
+{\varrho_{l-1}}v_y^2+{\varrho_{l-1}}\Theta_y^2+\frac{1}{J_{l-1}}|\frac{v_y}{J}|^qJ\right.\\
&&\ \ \ \ \ \ \ \ \ \ \ \ \ \ \ \ \ \left.+\frac{\varrho_{l-1}^{\alpha}}{J_{l-1}}|\frac{\Theta_y} J|^pJ+J_y^2\right)dy\bigg)(t)\\
&&\leqslant CH_l(t)\bigg(1-\frac{q+1}{q-1}H_l^{\frac{q+1}{q-1}}(t)(t-{T_{l-1}})\bigg)^{-\frac{q-1}{q+1}}.
    \end{eqnarray*}
Furthermore,
\begin{eqnarray*}
&&\sup\limits_{s\in[T_{l-1},t]}\bigg(\int\varrho_{l-1}^{\alpha}\left({\varrho_{l-1}}v^2+{\varrho_{l-1}}\Theta+\varrho_{l-1}^2\Theta^2
+{\varrho_{l-1}}v_y^2+{\varrho_{l-1}}\Theta_y^2+\frac{1}{J_{l-1}}|\frac{v_y}{J}|^qJ\right.\\
&&\ \ \ \ \ \ \ \ \ \ \ \ \ \ \ \ \ \left.+\frac{\varrho_{l-1}^{\alpha}}{J_{l-1}}|\frac{\Theta_y} J|^pJ+J_y^2\right)dy\bigg)(s)\\
&&+\int_{T_{l-1}}^t\int\varrho_{l-1}^{\alpha}\left(\frac{1}{J_{l-1}}|\frac{v_y}{J}|^qJ+\frac{\varrho_{l-1}}{J_{l-1}}|\frac{\Theta_y} J|^pJ+\frac{1}{J_{l-1}}|\frac{v_y}{J}|^{q-2}\frac { v_{yy}^2}J+\frac{1}{J_{l-1}}|\frac{\Theta_y} J|^{p-2}\frac {\Theta_{yy}^2}J\right.\\
&&\ \ \ \ \ \ \ \ \ \ \ \ \ \ \ \ \ \ \ \ \ \ \left.+\varrho_{l-1}^3\Theta_t^2+{\varrho_{l-1}}v_t^2\right)dyds\\
&&\leqslant C\bigg(H_l(t)+\int_{T_{l-1}}^tH_l^{\frac{2q} {q-1}}(s)\bigg(1-\frac{q+1}{q-1}H_l^{\frac{q+1}{q-1}}(s)(s-T_{l-1})\bigg)^{-\frac{2q}{q+1}}ds\bigg)
\end{eqnarray*}
holds for all $t\in[{T_{l-1}},{T_l}]$, where $\delta_{l}(t)=\inf\limits_{y\in{\mathbb R}}J(y,t)$ on $[T_{l-1},T_l]$ and
\begin{eqnarray*}
&&H_l(t)=\int\varrho_{l-1}^{\alpha}\Big({\varrho_{l-1}}v_{l-1}^2+{\varrho_{l-1}}\Theta_{l-1}
+\varrho_{l-1}^2\Theta_{l-1}^2+{\varrho_{l-1}}v_{l-1,y}^2+{\varrho_{l-1}}\Theta_{l-1,y}^2+|\frac{v_{l-1,y}} {J_{l-1}}|^q\\
&&\ \ \ \ \ \ \ \ \ \ \ \ +\varrho_{l-1}^2|\frac{\Theta_{l-1,y}} {J_{l-1}}|^p+ J_{l-1,y}^2\Big)dy
\\&&\ \ \ \ \ \ \ \ \ \ \ \ +C({\sup\limits_{y\in\mathbb R}{\varrho_{l-1}}})^{-\min\{\frac{(3q-2)\alpha+q}{2-q},\frac{10\alpha}{q-1}\}}
\left(\|(\varrho_{l-1}^{\alpha})_y\|_{L^p}\|(\varrho_{l-1}^{\alpha})_y\|_{L^\infty}\right)^{\frac{16q}{2-q}}
\int^t_{T_1}\left(\delta_{l}(s)\right)^{-\frac{24}{2-q}}ds,\\
&&M_l(t)=\frac{C}{q}\bigg(q\sup\limits_{y\in\mathbb R}J_{l-1}+H_l(t)+\int_{T_{l-1}}^tH_l^{\frac{2q}{q-1}}(s)\Big(1-\frac{q+1}{q-1}H_l^{\frac{q+1}{q-1}}(s)s\Big)^{-\frac{2q}{q-1}}ds\bigg),\\
&&F_l(t)={(\sup\limits_{y\in\mathbb R}\varrho_{l-1})}^{-\frac{\alpha}{q-1}}M_l^{-\frac{q+1}{q-1}}(t)\exp\bigg(-\frac{q+1}{q}{(\sup\limits_{y\in\mathbb R}\varrho_{l-1})}^{-\frac{\alpha}{q}}t\bigg),
\end{eqnarray*}
on $[T_{l-1},T_l]$. In order to find $T_l$,  one deduces from the steps to choose $T_2$ and get that $T_l$ satisfies the condition
\begin{equation}\label{5dlg-E45}
T_l-T_{l-1}>t\left(\frac{1} {2\widetilde\eta_l^{\frac {(q-1)^2}{2(q+1)}}}\right)-T_{l-1}>g^{\frac{q+1}{q-1}}(2)g^{-\frac{q+1}{q-1}}(2\widetilde\eta_l^{\frac {(q-1)^2}{2(q+1)}})=h^{\frac{q+1}{q-1}}(2,\widetilde\eta_l)
    \end{equation}
for some $\widetilde\eta_l>1$, where $g(k)$ and $h(k,\eta)$ are given in (\ref{5dlg-E39}) and (\ref{5dlg-E43}) respectively. Therefore, the strong solution $(J,\varrho,v,\Theta)$ is extended uniquely on $\mathbb R\times[0,T_l]~(l=2,3,\cdots),$ by bonding the  strong solution on $\mathbb R\times[0,T_{l-1}]$ and the  strong solution on $\mathbb R\times[T_{l-1},T_l]$ together at time $T_{l-1}$. Moreover, the extended step $T_l-T_{l-1}$ is with the following property
\begin{equation}\label{5dlg-E46}
T_l-T_{l-1}>h^{\frac{q+1}{q-1}}(2,\widetilde\eta_l).
    \end{equation}
It is easy to get that
\begin{eqnarray}\label{5dlg-E47}
&&\frac{1}{\bigg((k-\frac{1}{\eta^{\frac {(q-1)^2}{2(q+1)}}})^{\frac{q-1}{q+1}}+\frac{q-1}{q+1}(\eta^{\frac {(q-1)^2}{2(q+1)}})^{\frac{4q}{(q-1)(q+1)}}k^{\frac{q+1}{q-1}}(\eta^{{\frac {(q-1)^2}{2(q+1)}}}k-1)^{\frac{2q}{q+1}}\bigg)^{\frac{q+1}{q(q-1)}}}\nonumber\\
&&\geqslant
\bigg({\frac{q+1}{2q}}\bigg)^{\frac{q+1}{q(q-1)}}\frac{1}{k^{\frac{1}{q}}\Big(1+k^{\frac{q^2+4q-1}{(q-1)^2}}(\eta^{\frac {(q-1)^2}{2(q+1)}})^{\frac{2(q+1)}{(q-1)^2}}\Big)}
\end{eqnarray}
holds for any fixed $\eta>1$ and $k\in\mathbb Z^+$ with $k>2$. So, the integral
\begin{eqnarray}\label{5dlg-E48}
&&\int_1^\infty h^{\frac{q+1}{q-1}}(2,\eta)d\eta\\
&&=\int_1^\infty \bigg(\frac{1+\frac{q-1}{q+1}2^{\frac{q+1}{q-1}}}{(2-\frac{1}{\eta^{\frac {(q-1)^2}{2(q+1)}}})^{\frac{q-1}{q+1}}+\frac{q-1}{q+1}(\eta^{\frac {(q-1)^2}{2(q+1)}})^{\frac{4q}{(q-1)(q+1)}}2^{\frac{q+1}{q-1}}(2\eta^{{\frac {(q-1)^2}{2(q+1)}}}-1)^{\frac{2q}{q+1}}}\bigg)^{\frac{q+1}{q(q-1)}}d\eta\nonumber\\
&&\geqslant\bigg(1+\frac{q-1}{q+1}2^{\frac{q+1}{q-1}}\bigg)^{\frac{q+1}{q(q-1)}}
\bigg({\frac{q+1}{2q}}\bigg)^{\frac{q+1}{q(q-1)}}\frac{1}{2^{\frac{1}{q}}}\int_1^\infty\frac{1}{1+2^{\frac{q^2+4q-1}{(q-1)^2}}(\eta^{\frac {(q-1)^2}{2(q+1)}})^{\frac{2(q+1)}{(q-1)^2}}}d\eta=+\infty.\nonumber
    \end{eqnarray}
Thus, it is deduced from (\ref{5dlg-E46}) that
\begin{equation*}
\sum_{l=1}^{\infty}(T_l-T_{l-1})=\infty.
\end{equation*}
Therefore, the system (\ref{1dlg-E16})-(\ref{1dlg-E17}) admits a unique strong solution $(J,\varrho,v,\Theta)$ on $\mathbb R\times[0,T]$ for any $T\in(0,\infty)$.

\section*{Acknowledgement}

The work of Li Fang was supported in part  by the National Natural Science Foundation of China (Grant no. 11501445). The work of Aibin Zang  was partially supported by the National Science Foundation of China (Grant no. 12261039), and Jiangxi Provincial Natural Science Foundation  (Grant no. 20224ACB201004).


\begin{thebibliography}{sl}
\bibitem{Abbatiello-2020}
A. Abbatiello, E. Feireisl, A. Novotn\'{y}, Generalized solutions to models of compressible viscous fluids, Discrete Contin. Dyn. Syst. 41(1) (2021) 1-28.

\bibitem{Bohme-1987}
G. B\"{o}hme, Non-Newtonian Fluid Mechanics, Translations of North-Holland Series in Applied Mathematics and Mechanics, 31. North-Holland Publishing Co., Amsterdam, 1987.

\bibitem{Bellout-1994}
H. Bellout, F. Bloom, J. Ne\v{c}as, Young measure-valued solutions for non-Newtonian incompressible fluids, Commun. Partial Differential Equations, 19 (1994) 1763--1803.

\bibitem{Chhabra-2008}
R.P. Chhabra, J.F. Richardson, Non-Newtonian Flow and Applied Rheology (Second Edition). Oxford, 2008.


\bibitem{Diening-2002}L. Diening, Theoretical and numberical results for electrorheological fluids, PhD thesis, Univ, Freiburg im Breisgau, Mathematische Fakult\"{a}t, 156 p., 2002.

\bibitem{Diening-2010}
L. Diening, M. R\.{u}\v{z}i\v{c}ka, J. Wolf, Existence of weak solutions for unsteady motion of generalized Newtonian fluids, Ann. Scuola Norm. Sup. Pisa., 9 (2010) 1-46.


\bibitem{Feireisl-2004}
E. Feireisl, Dynamics of viscous compressible fluids, Oxford University Press, Oxford, 2004.

\bibitem{Fang-Guo-2012}
L. Fang, Z.H. Guo, Analytical solutions to a class of non-Newtonian fluids with free boundaries, J. Math. Phys, 53(2012), 103701.

\bibitem{Fang-Li-2015}
L. Fang, Z.L. Li, On the existence of local classical solution for a class of one-dimensional compressible non-Newtonian fluids, Acta Math. Sci. Ser. B (English Ed)., 35(2015), 157-181.

\bibitem{Fang-Zang-2023}
L. Fang, A.B. Zang, Global existence of strong solutions to the Cauchy problem for a one-dimensional  compressible non-Newtonian fluid, J. Math. Fluid Mech. (2023), 25:15.

\bibitem{Figueredo-2003}R. Figueredo, E. Sabadini, Firefighting foam stability, the effect of the drag reducer poly(ethylene) oxide, 2003, 215(1-3): 77-86.


\bibitem{Guo-Shang-2006}
B.L. Guo, Y.D. Shang, Dynamics of non-Newtonian fluid, National Defence Industry Press, Beijing, 2006.

\bibitem{Guo-Zhu-2002}
B.L. Guo, P.C. Zhu, Partial regularity of suitable weak solutions to the system of the incompressible non-Newtoniana fluids, J. Differential Equations, 178 (2002) 281-297.


\bibitem{Kalousek-Macha-Necasova-2020}
M. Kalousek, V. M\'{a}cha, \u{S}. Ne\v{c}asov\'{a}, Local-in-time existence of strong solutions to a class of compressible non-Newtonian Navier-Stokes equations, Math. Ann. 384(3-4) (2022) 1057-1089.

\bibitem{Ladyzhenskaya-1967}O A. Ladyzhenskaya, New equation for description of motion of viscous incompressible fluids and solvability in the
large boundary value problems for them. Proc Steklov Inst Math, 1967, 102: 95-118.

\bibitem{Ladyzhenskaya-1969}
O.A. Ladyzhenskaya, The mathematical theory of viscous incompressible flow 2nd ed, Gordon and Breach, New York, (1969).


\bibitem{Ladyzhenskaya-1970}
O.A. Ladyzhenskaya, Mathematical questions of the dynamics of a viscous incompressible fluid, Second revised and supplemented edition Izdat, Moscow, 1970.


\bibitem{Li-Xin-2020}
J.K. Li, Z.P. Xin, Entropy bounded solutions to the one-dimensional compressible Navier-Stokes equations with zero heat conduction and far field vacuum, Advances in Mathematics, 361(2020), 106923.

\bibitem{Li-Xin-2021}
J.K. Li, Z.P. Xin, Entropy-bounded solutions to the one-dimensional heat conductive compressible Navier-Stokes equations with far field vacuum, Communications on Pure and Applied Mathematics, (2021).

\bibitem{Lions-1969}
J.L. Lions, Quelques methods de resolutions des problemes aux limites non linaires, Gauthier-Villars, Paries, 1969.


\bibitem{Lindqvist-2006}P. Lindqvist, Notes on the p-Laplace equation, University of Jyv\"{a}skyl\"{a}, Jyv\"{a}skyl\"{a}, 2006.



\bibitem{Mamontov-1999}
A.E. Mamontov, On the global solvability of the multidimensional Navier-Stokes equations of a nonlinearly viscous fluid I, Sibirsk. Mat. Zh. 40(2) (1999) 408-420.

\bibitem{Mamontov-1999-1}
A.E. Mamontov, On the global solvability of the multidimensional Navier-Stokes equations of a nonlinearly viscous fluid II, Sibirsk. Mat. Zh. 40(3) (1999) 635-649.


\bibitem{Malek-1997}
J. M\'{a}lek, J. Ne\^{c}as, M. Rokyta, M. Ru\v{z}i\v{c}ka, Weak and Measure-Valued Solution to Evolutionary PDEs. Applied Mathematics and Mathematical Computation, 13. Chapman and Hall, London (1996)

\bibitem{Bartlomiej-Aneta-2018}
B. Matejczyk, A. Wr\'{o}blewska-Kami\'{n}ska, Unsteady flows of heat-conducting non-Newtonian fluids in Musielak-Orlicz spaces, Nonlinearity, 2018

\bibitem{Mowla-2006}D. Mowla,  A. Naderi, Experimental study of drag reduction by a polymeric additive in slug two-phase flow of crude oil and air in horizontal pipes, Chemical Engineering Science, 2006, 61(5): 1549-1554.


\bibitem{Necasova-1993}
\u{S}. Ne\v{c}asov\'{a}, A. Novotn\'{y}, Measure-valued solution for non-Newtonian compressible isothermal monopolar fluid, Acta Appl. Math., 37(1-2) (1994) 109-128, 1994. 

\bibitem{Necasova-2001}
\u{S}. Ne\v{c}asov\'{a}, P. Penel, $L^2-$decay for weak solution to equations of non-Newtonian incompressible fluids in the whole space, Nonlinear Anal., 47 (2001) 4181-4192.



\bibitem{Yuan-Si-Feng-2019}
H.J. Yuan, X. Si, Z.S. Feng, Global strong solutions of a class of non-Newtonian fluids with small initial energy, J. Math. Anal. Appl., 474 (2019), 72-93.

\bibitem{Yuan-Xu-2008}
H.J. Yuan, X.J. Xu, Existence and uniqueness of solutions for a class of non-Newtonian fluids with singularity and vacuum, J. Differential Equations, 245 (2008) 2871-2916.

\bibitem{Zang-2018}
A.B. Zang, Existence of weak solutions for non-stationary flows of fluids with shear thinning dependent viscosities under slip boundary conditions in half space, Sci. China Math., 61 (2018) 727-744.

\bibitem{Zhikov-2009}
V.V. Zhikov, S.E. Pastukhova, On the solvability of the Navier-Stokes system for a compressible non-Newtonian fluid(Russian), Dokl. Akad. Nauk, 427 (2009) 303-307; translation in Dokl. Math., 80 (2009) 511-515.


































%














\end{thebibliography}
\end{document}